\documentclass{amsart}

\usepackage{amssymb}
\usepackage{amsthm}
\usepackage{amsmath}

\usepackage[scr=dutchcal]{mathalfa}
\usepackage{accents}
\usepackage{mathtools}

\usepackage{tikz}
\usetikzlibrary{arrows,calc}

\usepackage{tikz-cd}

\usepackage{hyperref}
\usepackage[foot]{amsaddr}

\theoremstyle{plain}
\newtheorem{proposition}{Proposition}[section]
\newtheorem{theorem}[proposition]{Theorem}
\newtheorem{lemma}[proposition]{Lemma}
\newtheorem{corollary}[proposition]{Corollary}
\newtheorem{observation}[proposition]{Observation}
\theoremstyle{definition}
\newtheorem{example}[proposition]{Example}
\newtheorem{definition}[proposition]{Definition}

\theoremstyle{remark}
\newtheorem{remark}[proposition]{Remark}

\DeclareMathOperator{\Aut}{Aut}

\DeclareMathOperator{\diam}{diam}

\DeclareMathOperator{\Hom}{Hom}

\DeclareMathOperator{\End}{End}

\DeclareMathOperator{\Gr}{Gr} 
\DeclareMathOperator{\id}{id} 
\DeclareMathOperator{\Haus}{Haus} 
 
\DeclareMathOperator{\Isom}{Isom}

\DeclareMathOperator{\SL}{\mathsf{SL}}
\DeclareMathOperator{\GL}{\mathsf{GL}}

\DeclareMathOperator{\SU}{\mathsf{SU}}
\DeclareMathOperator{\PSL}{\mathsf{PSL}}
\DeclareMathOperator{\PGL}{\mathsf{PGL}}

\DeclareMathOperator{\dist}{d}

\DeclareMathOperator{\Ac}{\mathcal{A}}
\DeclareMathOperator{\Bc}{\mathcal{B}}
\DeclareMathOperator{\Cc}{\mathcal{C}}

\DeclareMathOperator{\Fc}{\mathcal{F}}
\DeclareMathOperator{\Gc}{\mathcal{G}}
\DeclareMathOperator{\Hc}{\mathcal{H}}
\DeclareMathOperator{\Kc}{\mathcal{K}}
\DeclareMathOperator{\Lc}{\mathcal{L}}
\DeclareMathOperator{\Nc}{\mathcal{N}}
\DeclareMathOperator{\Oc}{\mathcal{O}}

\DeclareMathOperator{\Uc}{\mathcal{U}}
\DeclareMathOperator{\Wc}{\mathcal{W}}

\DeclareMathOperator{\Pc}{\mathcal{P}}

\DeclareMathOperator{\Bb}{\mathbb{B}}
\DeclareMathOperator{\Cb}{\mathbb{C}}

\DeclareMathOperator{\Hb}{\mathbb{H}}

\DeclareMathOperator{\Kb}{\mathbb{K}}
\DeclareMathOperator{\Nb}{\mathbb{N}}
\DeclareMathOperator{\Rb}{\mathbb{R}}
\DeclareMathOperator{\Sb}{\mathbb{S}}

\DeclareMathOperator{\Zb}{\mathbb{Z}}

\DeclareMathOperator{\Psf}{\mathsf{P}}

\newcommand{\abs}[1]{\left|#1\right|}

\newcommand{\norm}[1]{\left\|#1\right\|}
\newcommand{\wt}[1]{\widetilde{#1}}
\newcommand{\wh}[1]{\widehat{#1}}
\newcommand{\ip}[1]{\left\langle #1\right\rangle}

\newcommand{\geodflow}{\phi}
\newcommand{\flatflow}{\varphi}
\newcommand{\homflow}{\psi}

\newcommand{\into}{\hookrightarrow}
\newcommand{\proj}{\mathbf{P}}
\newcommand{\peripherals}{\Pc}

\begin{document}

\title[Relatively Anosov representations]{Relatively Anosov representations via flows II: examples}
\author{Feng Zhu}
\email{fzhu52@wisc.edu}
\author{Andrew Zimmer}\address{Department of Mathematics, University of Wisconsin-Madison, Madison, WI 53706, U.S.A}\email{amzimmer2@wisc.edu}
\date{\today}

\begin{abstract} This is the second in a series of two papers that develops a theory of relatively Anosov representations using the original ``contracting flow on a bundle" definition of Anosov representations introduced by Labourie and Guichard--Wienhard. In this paper we focus on building families of examples. 

\end{abstract}

\maketitle

\tableofcontents

\section{Introduction}

Anosov representations were introduced by Labourie~\cite{L2006}, and further developed by Guichard--Wienhard~\cite{GW}, as a generalization of convex cocompact representations into the isometry group of real hyperbolic space. Informally speaking, an Anosov representation is a representation of a word-hyperbolic group into a semisimple Lie group which has a equivariant boundary map into a flag manifold with good dynamical properties.

This is the second in a series of two papers whose purpose is to develop a theory of relatively Anosov representations, extending the theory of Anosov representations to relatively hyperbolic groups, using the original  ``contracting flow on a bundle" definition of Labourie and Guichard--Wienhard.  The general theory was developed in the first paper. In this paper we will focus on examples. 

Throughout the paper, we will let $\Kb$ denote either the real numbers $\Rb$ or the complex numbers $\Cb$.

\subsection{Some results from the first paper}\label{sec:results from first paper} We briefly recall some of the results from the first paper. Relatively Anosov representations are perhaps most naturally defined using the following boundary map definition (which is equivalent to being ``asymptotically embedded'' in the sense of Kapovich--Leeb~\cite{KL} and ``relatively dominated'' in the sense of~\cite{reldomreps}, see~\cite[Sec.\ 4]{ZZ2022a} for details). 

\begin{definition}\label{defn:Pk Anosov} Suppose that $(\Gamma,\peripherals)$ is relatively hyperbolic with Bowditch boundary $\partial(\Gamma, \peripherals)$. A representation $\rho\colon \Gamma \to \SL(d,\Kb)$ is \emph{$\Psf_k$-Anosov relative to $\peripherals$} if there exists a continuous map 
$$
\xi = (\xi^k, \xi^{d-k}) \colon \partial(\Gamma, \peripherals) \to \Gr_k(\Kb^d) \times \Gr_{d-k}(\Kb^d)
$$
which is 
\begin{enumerate} 
\item \emph{$\rho$-equivariant}:  if $\gamma \in \Gamma$, then $\rho(\gamma) \circ \xi = \xi \circ \gamma$,
\item \emph{transverse}: if $x,y \in \partial(\Gamma, \peripherals)$ are distinct, then $\xi^k(x) \oplus \xi^{d-k}(y) = \Kb^d$, 
\item \emph{strongly dynamics preserving}: if $(\gamma_n)_{n \geq 1}$ is a sequence of elements in $\Gamma$ where $\gamma_n \to x \in \partial(\Gamma, \peripherals)$ and $\gamma_n^{-1} \to y \in \partial(\Gamma, \peripherals)$, then 
$$
\lim_{n \to \infty} \rho(\gamma_n)V = \xi^k(x)
$$
for all $V \in \Gr_k(\Kb^d)$ transverse to $\xi^{d-k}(y)$. 
\end{enumerate}
\end{definition} 

One of the main results in the first paper shows that the definition above can be recast in terms of a contracting flow on a certain vector bundle associated to the representation.

Given a relatively hyperbolic group $(\Gamma, \peripherals)$ we can realize $\Gamma$ as a subgroup of $\Isom(X)$ where $X$ is a proper geodesic Gromov-hyperbolic metric space such that every point in $X$ is within a uniformly bounded distance of a geodesic, $\Gamma$ acts geometrically finitely on the Gromov boundary $\partial_\infty X$ of $X$, and the stabilizers of the parabolic fixed points are exactly the conjugates of $\peripherals$. Following the terminology in~\cite{HealyHruska}, we call such an $X$ a \emph{weak cusp space for $(\Gamma, \peripherals)$}.

Given such an $X$, let $\Gc(X)$ denote the space of parametrized geodesic lines in $X$ and for $\sigma \in \Gc(X)$ let $\sigma^\pm : = \lim_{t \to \pm\infty} \sigma(t) \in \partial_\infty X$. The space $\Gc(X)$ has a natural flow $\geodflow^t$  given by 
$\geodflow^t(\sigma) = \sigma(\cdot + t)$ which descends to a flow, which we also denote by $\geodflow^t$, on the quotient $\wh{\Gc}(X) := \Gamma \backslash \Gc(X)$. 

Given a representation $\rho\colon \Gamma \to \SL(d,\Kb)$, let 
$$
E(X) := \Gc(X) \times \Kb^d \quad \text{and} \quad \wh{E}_\rho(X) := \Gamma \backslash E(X)
$$
where $\Gamma$ acts on  $E(X)$ by 
$\gamma \cdot (\sigma, Y) = (\gamma \circ \sigma, \rho(\gamma) Y)$. 
The flow $\geodflow^t$ extends to a flow on $E(X)$, which we call $\flatflow^t$, which acts trivially on the second factor. This in turn descends to a flow on $\wh{E}_\rho(X)$ which we also call $\flatflow^t$. 

Given a continuous, $\rho$-equivariant, transverse map 
$$
\xi = (\xi^k, \xi^{d-k}) \colon \partial(\Gamma, \peripherals) \to \Gr_k(\Kb^d) \times \Gr_{d-k}(\Kb^d)
$$
we can define vector bundles $\Theta^k, \Xi^{d-k} \to \Gc(X)$ by setting 
$\Theta^k(\sigma) := \xi^k(\sigma^+)$ and $\Xi^{d-k}(\sigma) := \xi^{d-k}(\sigma^-)$.
Since $\xi$ is transverse, we have $E(X) = \Theta^k \oplus \Xi^{d-k}$. Since $\xi$ is $\rho$-equivariant, this descends to a vector bundle decomposition  $\wh{E}_\rho(X) = \wh{\Theta}^k \oplus \wh{\Xi}^{d-k}$. Also, by construction, these subbundles are $\flatflow^t$-invariant. We can then consider the bundle 
$
\Hom(\wh{\Xi}^{d-k}, \wh{\Theta}^k) \to \wh{\Gc}(X)
$
and, since the subbundles are $\flatflow^t$-invariant, we can define a flow on $\Hom(\wh{\Xi}^{d-k}, \wh{\Theta}^k)$ by 
$\homflow^t(f) := \flatflow^t \circ f \circ \flatflow^{-t}$.
Finally, we note that any metric on $\wh{E}_\rho(X) \to \wh{\Gc}(X)$ induces, via the operator norm, a continuous family of norms on the fibers of $\Hom(\wh{\Xi}^{d-k}, \wh{\Theta}^k) \to \wh{\Gc}(X)$. 

\begin{definition}\label{defn:flow definition}
With the notation above, we say that $\rho$ is  \emph{$\Psf_k$-Anosov relative to $X$} if there exists a metric $\norm{\cdot}$ on the vector bundle $\wh{E}_\rho(X)\to \wh{\Gc}(X)$ such that the flow $\homflow^t$ on $\Hom(\wh{\Xi}^{d-k}, \wh{\Theta}^k)$ is exponentially contracting (with respect to the associated operator norms).
\end{definition}

In~\cite{ZZ2022a} we proved these two definitions are equivalent, and indeed one can always make a particular choice of weak cusp space. These are what are  often called \emph{Groves--Manning cusp spaces} and they are formed by attaching so-called combinatorial horoballs to a Cayley graph of the group (see Definition~\ref{def:cusp spaces}). These spaces are perhaps the most canonical choice of weak cusp space, see~\cite{HealyHruska,GrovesManning}. 

\begin{theorem}[{\cite[Th.\ 1.3]{ZZ2022a}}] \label{thm:equivalence of definitions}
Suppose that $(\Gamma,\peripherals)$ is relatively hyperbolic and $\rho\colon \Gamma \to \SL(d,\Kb)$ is a representation. Then the following are equivalent: 
\begin{enumerate}
\item $\rho$ is $\Psf_k$-Anosov relative to $\peripherals$,
\item there is a weak cusp space $X$ for $(\Gamma,\peripherals)$ such that $\rho$ is $\Psf_k$-Anosov relative to $X$,
\item if $X$ is any Groves--Manning cusp space of $(\Gamma,\peripherals)$, then $\rho$ is $\Psf_k$-Anosov relative to $X$.
\end{enumerate}
\end{theorem}

\begin{remark} Theorem~\ref{thm:equivalence of definitions} leaves open the question if the above conditions are equivalent to $\rho$ is being $\Psf_k$-Anosov relative to any weak cusp space. Using a different definition of flow spaces (which is equivalent to ours when $X$ is ${\rm CAT}(-1)$), Wang showed that this is the case~\cite{Wang2023}.
\end{remark}

As a consequence of Theorem~\ref{thm:equivalence of definitions}, standard dynamical arguments can be used to prove a relative stability result. Given a representation $\rho_0\colon (\Gamma, \peripherals) \to \SL(d,\Kb)$, we let $\Hom_{\rho_0} (\Gamma, \SL(d,\Kb))$ denote the set of representations $\rho\colon \Gamma \to \SL(d,\Kb)$ such that for each $P \in \peripherals$, the representations $\rho|_P$ and $\rho_0|_P$ are conjugate.

\begin{theorem}[{\cite[Th.\ 1.6]{ZZ2022a}}]\label{thm:stability}
Suppose that $(\Gamma, \peripherals)$ is relatively hyperbolic and $X$ is a weak cusp space for $(\Gamma, \peripherals)$. If $\rho_0\colon\Gamma \to \SL(d,\Kb)$ is $\Psf_k$-Anosov relative to $X$, then there exists an open neighborhood $\Oc$ of $\rho_0$ in $\Hom_{\rho_0}(\Gamma, \SL(d,\Kb))$ such that every representation in $\Oc$ is $\Psf_k$-Anosov relative to $X$. 
\end{theorem}

\begin{remark} In recent work, Weisman~\cite{W2022} introduces a new class of representations of relatively hyperbolic groups called extended geometrically finite representations which includes the class of relatively Anosov representations. For this class of representations, Weisman  proves a general stability result which implies, in the context of  Theorem~\ref{thm:stability}, that being $\Psf_k$-Anosov relative to $\peripherals$ is an open condition in $\Hom_{\rho_0}(\Gamma, \SL(d,\Kb))$.
\end{remark}

In the relatively hyperbolic case, the space $\wh{\Gc}(X)$ will be non-compact and thus it is possible for a metric on the vector bundle $\wh{E}_\rho(X) \to \wh{\Gc}(X)$ to be quite badly behaved. In~\cite{ZZ2022a}  we introduced a subclass of relatively Anosov representations where the metric is assumed to have additional regularity properties and proved that this special class has nicer properties. This class is defined as follows.  

\begin{definition}\label{defn:locally uniform norms} Suppose that $(\Gamma,\peripherals)$ is relatively hyperbolic, $X$ is a weak cusp space for $(\Gamma, \peripherals)$, and $\rho \colon \Gamma \to \SL(d,\Kb)$ is a representation. 

\begin{itemize}
\item A metric $\norm{\cdot}$ on $\wh{E}_\rho(X) \to \wh{\Gc}(X)$ is \emph{locally uniform} if its lift to $\Gc(X) \times \Kb^d \to \Gc(X)$ has the following property: For any $r > 0$ there exists $L_r > 1$ such that
\begin{align*}
\frac{1}{L_r} \norm{\cdot}_{\sigma_1} \leq \norm{\cdot}_{\sigma_2} \leq L_r \norm{\cdot}_{\sigma_1}
\end{align*}
for all $\sigma_1, \sigma_2 \in \Gc(X)$ with $\dist_X(\sigma_1(0), \sigma_2(0)) \leq r$. 
\item $\rho$ is \emph{uniformly $\Psf_k$-Anosov relative to $X$} if it is $\Psf_k$-Anosov relative to $\peripherals$ and there exists a locally uniform metric $\norm{\cdot}$ on $\wh{E}_\rho(X) \to \wh{\Gc}(X)$ such that the flow $\homflow^t$ on $\Hom(\wh{\Xi}^{d-k}, \wh{\Theta}^k)$ is exponentially contracting (with respect to the associated operator norms).
\end{itemize}
\end{definition} 

In~\cite{ZZ2022a}  we proved that uniformly relatively Anosov representations are very nicely behaved. In particular, one can construct an equivariant quasi-isometric map of the entire weak cusp space into the symmetric space associated to $\SL(d,\Kb)$ and the boundary map is H\"older relative to any visual metric on the Bowditch boundary and Riemannian distance on the Grassmanian \cite[Th.\ 1.13]{ZZ2022a}. We also proved that the uniformly Anosov representations form an open set in the constrained space of representations considered in Theorem~\ref{thm:stability}.

\subsection{Results of this paper} The main aim of this paper is to produce classes of examples of relatively Anosov representations. Just as Anosov representations can be thought of as a generalization of convex cocompact representations, so relatively Anosov representations can be thought of as a generalization of geometrically finite representations into rank-one semisimple Lie groups. 

In fact, essentially by definition, these two notions coincide for rank-one semisimple Lie groups. More precisely, in~\cite[Sec.\ 13]{ZZ2022a}, we extended Definition~\ref{defn:Pk Anosov} to relatively Anosov representations into general semisimple Lie groups and with that definition we have the following observation.

\begin{observation} Suppose that $X$ is a negatively-curved symmetric space and $\mathsf{P}^+$, $\mathsf{P}^-$ is a pair of opposite parabolic subgroups in $\Isom_0(X)$, the connected component of the identity in the isometry group of $X$.

If $(\Gamma, \peripherals)$ is relatively hyperbolic and $\rho \colon \Gamma \to \Isom_0(X)$ is a representation, then the following are equivalent:
\begin{enumerate}
\item $\rho$ is $\mathsf{P}^\pm$-Anosov relative to $\peripherals$ (in the sense of~\cite[Def.\ 13.1]{ZZ2022a}). 
\item $\ker \rho$ is finite, $\rho(\Gamma)$ is geometrically finite, and $\rho(\peripherals)$ is a set of representatives of the conjugacy classes of maximal parabolic subgroups in $\rho(\Gamma)$. 
\end{enumerate}
\end{observation} 

\begin{proof} This follows directly from the ``F2'' definition in~\cite{Bowditch_GF_Riem} of  geometrically finite subgroups in $\Isom(X)$ and~\cite[Def.\ 13.1]{ZZ2022a}. \end{proof} 

\begin{remark}  $\Isom_0(X)$ only contains one conjugacy class of opposite parabolic subgroups and so by definition (see ~\cite[Def.\ 13.1]{ZZ2022a}) a representation is $\mathsf{P}^\pm$-Anosov relative to $\peripherals$ if and only if it is $\mathsf{Q}^\pm$-Anosov relative to $\peripherals$ for any choice of opposite parabolic subgroups in $\mathsf{Q}^\pm \leq \Isom_0(X)$. 
\end{remark} 

Motivated by this observation, we construct additional examples of relatively Anosov representations. The first set of examples come from considering representations of geometrically finite subgroups of rank-one semisimple Lie groups. 

The second set of examples are motivated by the Klein-Beltrami model of hyperbolic geometry. In particular, this model realizes real hyperbolic $n$-space as a convex domain of $\proj(\Rb^{n+1})$ in such a way that the hyperbolic metric coincides with the Hilbert metric on the convex domain. We observe that one can consider ``geometrically finite'' subgroups acting on more general convex domains to construct additional examples of relatively Anosov representations. 

We also consider additional classes of examples, described in Section~\ref{sec:misc_examples below}.

\subsubsection{Geometric finiteness in rank one}\label{subsec: Geometric finiteness in rank one} For the rest of this subsection, suppose $X$ is a negatively-curved symmetric space and let $\mathsf{G}:=\Isom_0(X)$ denote the connected component of the identity in the isometry group of $X$. Let $\partial_\infty X$ denote the geodesic boundary of $X$. Then given a discrete group $\Gamma \leq \mathsf{G}$, let $\Lambda_X(\Gamma) \subset \partial_\infty X$ denote the limit set of $\Gamma$ and let $\Cc_X(\Gamma)$ denote the convex hull of the limit set in $X$. 

When $\Gamma \leq \mathsf{G}$ is geometrically finite, we will let $\peripherals(\Gamma)$ denote a set of representatives of the conjugacy classes of maximal parabolic subgroups in $\Gamma$. Then $(\Gamma,\peripherals(\Gamma) )$ is relatively hyperbolic and $\Cc_X(\Gamma)$ is a weak cusp space for $(\Gamma,\peripherals(\Gamma) )$. 

We will observe that restricting a proximal linear representation of $\mathsf{G}$ to a geometrically finite subgroup produces a uniformly relatively Anosov representation. 

\begin{proposition}[see Proposition~\ref{prop:representations of rank one groups}]\label{prop:homog_cusps} Suppose that $\tau \colon \mathsf{G} \to \SL(d,\Kb)$ is $\Psf_k$-proximal (i.e.\ the image of $\tau$ contains a $\Psf_k$-proximal element). If $\Gamma \leq \mathsf{G}$ is geometrically finite, then $ \tau|_{\Gamma}$ is uniformly $\Psf_k$-Anosov relative to $\Cc_X(\Gamma)$.
 \end{proposition} 
 
\begin{remark}  A version of Proposition~\ref{prop:homog_cusps} also holds for representations into general semisimple Lie groups, in fact using~\cite[Proposition 13.4]{ZZ2022a} the general case follows immediately from the $\SL(d,\Kb)$ case. 
 \end{remark}
 
 In the context of Proposition~\ref{prop:homog_cusps}, we can obtain additional examples by starting with the representation $\rho_0 : =\tau|_\Gamma$ and deforming it in $\Hom_{\rho_0}(\Gamma, \SL(d,\Kb))$. By Theorem~\ref{thm:stability}, any sufficiently small deformation will be a uniformly relatively Anosov representation. 
  
 Using Proposition~\ref{prop:homog_cusps} we will also construct the following example. 
 
 \begin{example}[see Section~\ref{sec:not uniform rel GM cusp space}]\label{ex:non uniform to GM cusp space} Let $X:=\Hb^2_{\Cb}$ denote complex hyperbolic 2-space. There exists a geometrically finite subgroup $\Gamma \leq \Isom_0(X)$ and a representation $\rho\colon \Gamma \to \SL(3,\Cb)$ which is uniformly $\Psf_1$-Anosov relative to $\Cc_{X}(\Gamma)$, but not uniformly $\Psf_1$-Anosov relative to any Groves--Manning cusp space associated to $(\Gamma,\peripherals(\Gamma) )$. 
 \end{example} 
 
We remark that the example makes crucial use of the fact that for horoballs in complex hyperbolic space, distances decay at different exponential rates as we approach the cusp. In fact in real hyperbolic geometry, one can show that the convex hull of the limit set of a geometrically finite group is quasi-isometric to the associated Groves--Manning cusp space. 
 
 This example shows that there is value in studying bundles associated to general weak cusp spaces and not just the Groves--Manning cusp spaces. In future work we will further explore how to select the ``best'' weak cusp spaces to study a given relatively Anosov representation. 
 
We can relax the condition in Proposition~\ref{prop:homog_cusps} to only assuming that the representation extends on each peripheral  subgroup. More precisely, if $\Gamma \leq \mathsf{G}$ is geometrically finite and  $\rho \colon \Gamma \to \SL(d,\Kb)$ is $\Psf_k$-Anosov relative to $\peripherals(\Gamma)$, then we say that $\rho$ has \emph{almost homogeneous cusps} if there exists a finite cover $\pi\colon \widetilde{\mathsf{G}} \to \mathsf{G}$ such that for each $P \in\peripherals(\Gamma)$ there is a representation $\tau_P\colon \widetilde{\mathsf{G}} \to \SL(d,\Kb)$ where 
\begin{align*}
\left\{ \tau_P(g)(\rho\circ \pi)(g)^{-1} : g \in \pi^{-1}(P)\right\}
\end{align*}
is relatively compact in $\SL(d,\Kb)$. This technical definition informally states that the representation restricted to each peripheral subgroup extends to a representation of $\mathsf{G}$. 

\begin{theorem}[see Theorem~\ref{thm:repn with almost homog cusps}]\label{thm:repn with almost homog cusps in intro} Suppose that $\Gamma \leq\mathsf{G}$ is geometrically finite and  $\rho \colon \Gamma \to \SL(d,\Kb)$ is $\Psf_k$-Anosov relative to $\peripherals(\Gamma)$. If $\rho$ has almost homogeneous cusps, then $\rho$ is uniformly $\Psf_k$-Anosov relative to $\Cc_X(\Gamma)$.
\end{theorem} 

Proposition 3.6 in ~\cite{CZZ2021} implies that every relatively Anosov representation of a geometrically finite Fuchsian group has almost homogeneous cusps and hence is uniform. This also follows directly from the construction of canonical norms in~\cite[Sec.\ 3.1]{CZZ2021}. 

\begin{corollary}\label{cor:geometrically finite Fuchsian groups} If $X = \Hb^2_{\Rb}$ is real hyperbolic 2-space, $\Gamma \leq \Isom_0(X)$ is geometrically finite, and $\rho \colon \Gamma \to \SL(d,\Kb)$ is $\Psf_k$-Anosov relative to $\peripherals(\Gamma)$, then $\rho$ is uniformly $\Psf_k$-Anosov relative to $\Cc_X(\Gamma)$.
\end{corollary} 

Allowing representations of finite covers in the definition of almost homogeneous cusps is motivated by the following examples. 

\begin{example} Identify $\Isom_0(\Hb_{\Rb}^2)$ with $\PSL(2,\Rb)$ and let $\pi \colon \SL(2,\Rb) \to \PSL(2,\Rb)$ denote the double cover.  Let $P \leq \PSL(2,\Rb)$ be the cyclic subgroup generated by the projection of 
$$
\mathscr{u} := \begin{pmatrix} 1 & 1 \\ 0 & 1 \end{pmatrix} \in \SL(2,\Rb)
$$ 
to $\PSL(2,\Rb)$. Also, let $\tau_d \colon \SL(2,\Rb) \to \SL(d,\Rb)$ denote the standard irreducible representation.
\begin{itemize}
\item The representation $\rho_1 \colon P \to \SL(5,\Rb)$ defined by 
$$
\rho_1( [\mathscr{u}]) = (\tau_2 \oplus \tau_3)(\mathscr{u})
$$
does not extend to a representation of $\PSL(2,\Rb)$ since $(\tau_2 \oplus \tau_3)(-\id_2) \neq \id_5$. However, 
$$
\{ (\tau_2 \oplus \tau_3)(g) \cdot (\rho_1 \circ \pi)(g)^{-1} : g \in \pi^{-1}(P)\} = \{ (-\id_2) \oplus \id_3\}
$$
is compact. 
\item The representation $\rho_2 \colon P \to \SL(4,\Rb)$ defined by 
$$
\rho_2( [\mathscr{u}]) = (-\tau_2(\mathscr{u})) \oplus \tau_2(\mathscr{u})
$$
also does not extend to a representation of $\PSL(2,\Rb)$. However, 
$$
\{ (\tau_2 \oplus \tau_2)(g) \cdot (\rho_2 \circ \pi)(g)^{-1} : g \in \pi^{-1}(P)\} = \{ (-\id_2) \oplus \id_2\}
$$
is compact. 
\end{itemize}
\end{example} 

\subsubsection{Geometric finiteness in convex projective geometry} \label{introsec:convex projective geometry}

We will also apply our general results to the setting of convex real projective geometry.

Given a properly convex domain $\Omega \subset \proj(\Rb^d)$, the \emph{automorphism group of $\Omega$}, denoted $\Aut(\Omega)$, is the subgroup of $\PGL(d,\Rb)$ which preserves $\Omega$. Such a domain also has a natural $\Aut(\Omega)$-invariant metric, the Hilbert metric $\dist_\Omega$ (see Section \ref{subsec:hilbert metric} for the definition).
The \emph{limit set} of a subgroup $\Gamma \leq \Aut(\Omega)$ is defined to be 
$$
\Lambda_\Omega(\Gamma) := \partial \Omega \cap \bigcup_{p \in \Omega} \overline{\Gamma \cdot p}.
$$
Following~\cite{CZZ2022}, we say that $\Gamma$ is a \emph{projectively visible subgroup of $\Aut(\Omega)$} if 
\begin{enumerate}
\item for all $p,q \in \Lambda_\Omega(\Gamma)$ distinct, the open line segment in $\overline{\Omega}$ joining $p$ to $q$ is contained in $\Omega$, and
\item every point in $\Lambda_\Omega(\Gamma)$ is a $\Cc^1$-smooth point of $\partial \Omega$.
\end{enumerate} 

\begin{example}
The Klein-Beltrami model identifies real hyperbolic $n$-space with the properly convex domain 
$$
\Bb := \left\{ [1:x_1 : \dots : x_n] \in \proj(\Rb^{n+1}) : \sum x_j^2 < 1\right\}
$$
endowed with its Hilbert metric $\dist_{\Bb}$. The domain $\Bb$ is strictly convex and has $\Cc^\infty$-smooth boundary, so any discrete subgroup in $\Aut(\Bb)$ is a projectively visible subgroup. 
\end{example} 

A projectively visible subgroup acts as a convergence group on its limit set and if, in addition, the action on the limit set is geometrically finite then the inclusion representation is relatively $\Psf_1$-Anosov. These assertions follow from~\cite[Prop.\ 3.5]{CZZ2022}, see Proposition~\ref{prop:projvis+geomfin=relAnosov} below. 

Conversely, we characterize exactly when the image of a relatively $\mathsf{P}_1$-Anosov representation  is a projectively visible subgroup which acts geometrically finitely on its limit set. This characterization is in terms of a lifting property of the Anosov boundary map, see Definition~\ref{defn:the lifting property}  below. 

\begin{proposition}[see Proposition~\ref{prop:lifting property one}]\label{prop:lifting property one in intro} Suppose that $(\Gamma, \peripherals)$ is relatively hyperbolic and $\rho \colon \Gamma \to \PGL(d,\Rb)$ is $\Psf_1$-Anosov relative to $\peripherals$. Then the following are equivalent: 
\begin{enumerate}
\item $\rho$ has the lifting property (in the sense of Definition~\ref{defn:the lifting property}), 
\item there exists a properly convex domain $\Omega_0 \subset \proj(\Rb^d)$ where $\rho(\Gamma) \leq \Aut(\Omega_0)$,
\item there exists a properly convex domain $\Omega \subset \proj(\Rb^d)$ where $\rho(\Gamma) \leq \Aut(\Omega)$ is a projectively visible subgroup which acts geometrically finitely on its limit set. 
\end{enumerate}
\end{proposition} 

We will also prove that the lifting property is an open and closed condition in the following sense. 

\begin{proposition}[see Proposition~\ref{prop:lifting property two}]\label{prop:lifting property two in intro} Suppose that $(\Gamma, \peripherals)$ is relatively hyperbolic and $\rho_0 \colon \Gamma \to \PGL(d,\Rb)$ is a representation. Let $\Ac_1(\rho_0)$ denote the set of representations in $ \Hom_{\rho_0}(\Gamma, \PGL(d,\Rb))$  which are $\Psf_1$-Anosov relative to $\peripherals$. Then the subset $\Ac_1^+(\rho_0) \subset \Ac_1(\rho_0)$ of representations with the lifting property is open and closed in $\Ac_1(\rho_0)$. 

\end{proposition} 

\begin{remark} In the case when $\peripherals=\varnothing$ (i.e. $\Gamma$ is word hyperbolic), the above proposition follows from \cite[Prop.\ 1.2]{ST2018}. In fact, in ~\cite{ST2018}, they consider lifting properties for Anosov representations into general semisimple Lie groups. It seems likely that some version of their result should hold in the relative case as well. 
\end{remark}

As a corollary to~\cite[Cor.\ 13.6]{ZZ2022a} and Proposition~\ref{prop:lifting property two in intro}, we obtain the following stability result. 

\begin{corollary}\label{cor:real proj stability in intro} Suppose that $\Gamma \leq \Aut(\Omega)$ is a projectively visible subgroup acting geometrically finitely on its limit set and $\iota \colon \Gamma \into \PGL(d,\Rb)$ is the inclusion representation. Then there is an open neighborhood $\Oc \subset \Hom_{\iota}(\Gamma, \PGL(d,\Rb))$ of $\iota$ such that: if $\rho \in \Oc$, then there exists a properly convex domain $\Omega_\rho \subset \proj(\Rb^d)$ where $\rho(\Gamma) \leq \Aut(\Omega_\rho)$ is a projectively visible subgroup acting geometrically finitely on its limit set.
\end{corollary} 

\begin{remark} For other stability results in the context of convex real projective geometry, see~\cite{Koszul1968,BenoistIII,Marquis2010,CLT2018,Choi}. \end{remark}

Using the methods in~\cite{DGK2017} and~\cite{Z2017}, we will construct the following examples, which brings the examples in Sections~\ref{subsec: Geometric finiteness in rank one} into the convex real projective setting. 

\begin{proposition}[see Propositions~\ref{prop:not real hyperbolic plane}  and \ref{prop: if tau preserves a properly convex domain}]\label{prop: repn of rank one groups in convex setting intro} Suppose that $X$ is a negatively-curved symmetric space which is not isometric to real hyperbolic 2-space and $\mathsf{G} := \Isom_0(X)$. If $\tau \colon \mathsf{G} \to \PGL(d,\Rb)$ is $\Psf_1$-proximal, then there exists a $\tau(\mathsf{G})$-invariant properly convex domain $\Omega \subset \proj(\Rb^d)$ such that: if $\Gamma \leq \mathsf{G}$ is geometrically finite, then 
\begin{enumerate}
\item $\tau(\Gamma)$ is a projectively visible subgroup of $\Aut(\Omega)$ and acts geometrically finitely on its limit set. 
\item If $\Cc_\Gamma: = \Cc_\Omega(\tau(\Gamma))$, then $(\Cc_\Gamma, \dist_\Omega)$ is Gromov-hyperbolic. 
\end{enumerate}
\end{proposition} 

\begin{remark} 
We also characterize the $\Psf_1$-proximal representations of $\Isom_0(\Hb_{\Rb}^2)$ which satisfy the conclusion of Proposition~\ref{prop: repn of rank one groups in convex setting intro}, see Proposition~\ref{prop:real hyperbolic plane case} below. 
\end{remark} 

In the context of Proposition~\ref{prop: repn of rank one groups in convex setting intro}, we can obtain additional examples in the convex real projective setting by starting with the representation $\rho_0 : =\tau|_\Gamma$ and deforming it in $\Hom_{\rho_0}(\Gamma, \PGL(d,\Rb))$. By Corollary~\ref{cor:real proj stability in intro}, any sufficiently small deformation will be a projectively visible subgroup of some properly convex domain which acts geometrically finitely on its limit set.

\subsubsection{Examples beyond geometric finiteness}\label{sec:misc_examples below}

We also describe three more families of examples which do not clearly fit within either of the two geometric finiteness frameworks above.

In Section~\ref{sec:pingpong}, we use a ping-pong argument to show that certain free-products of linear discrete groups give rise to relatively Anosov representations. This effort is motivated by the following question: which linear discrete groups appear as the image of a peripheral subgroup under a relatively $\Psf_k$-Anosov representation? Delaying definitions until later, it follows fairly easily from the definition that  any such linear group is
\begin{enumerate}
\item weakly unipotent, 
\item $\Psf_k$-divergent, and
\item has $(k,d-k)$-limit set consisting of a single point
\end{enumerate} 
(see Proposition~\ref{prop:eigenvalue data in rel Anosov repn} and Observation~\ref{obs:limit set of rel anosov}).  Using a ping-pong argument, we will show that these properties are essentially the only constraints. More precisely, we have the following. 

\begin{proposition}[see Proposition~\ref{prop:pingpong with unipotents}] Suppose that $U \leq \SL(d,\Kb)$ is a discrete group which is weakly unipotent, $\Psf_k$-divergent, and whose $(k,d-k)$-limit set is a single point. Then there is a relatively hyperbolic group $(\Gamma, \peripherals)$, a $\Psf_k$-Anosov representation $\rho\colon \Gamma \to  \PSL(d,\Kb)$, and $P \in  \peripherals$ such that $\rho(P) \leq U$ has finite index. 
\end{proposition} 

This allows us to construct new examples of relatively Anosov representations where the peripherals are non-Abelian nilpotent groups, for instance using the linear representation of the integer Heisenberg group constructed in \cite{Cooper_Heis}.

In Section~\ref{sec:pappus--schwartz}, we show that certain representations of $\PSL(2,\Zb)$ into $\PGL(3,\Rb)$ constructed by Rich Schwartz~\cite{Schwartz}  are $\Psf_1$-Anosov relative to certain cyclic subgroups. Schwartz' beautiful construction comes from iterating Pappus's theorem \cite{Schwartz} and he also showed that these representations have many of the properties that relatively Anosov representations (not yet defined at the time) have. We should also note that Barbot--Lee--Val\'erio proved that these representations are limits of families of Anosov representations of word hyperbolic groups \cite{BLV}.

Finally, in Section~\ref{sec:semisimplification}, we show that if a representation $\rho\colon \Gamma \to \SL(d,\Kb)$ is $\Psf_k$-Anosov relative to $\peripherals$, then so is any semisimplification $\rho^{ss}\colon \Gamma \to \SL(d,\Kb)$ of $\rho$. On the other hand, we exhibit a counter-example to the statement that if some semisimplification $\rho^{ss}$ of $\rho$ is $\Psf_k$-Anosov relative to $\peripherals$, then $\rho$ is $\Psf_k$-Anosov relative to $\peripherals$. In particular, the notion of relative Anosovness is not well-defined on the level of the character variety of $\Gamma$ in $\SL(d,\Kb)$, which can be viewed as the quotient of $\Hom(\Gamma,\SL(d,\Kb))$ by the relation ``having the same semisimplification.'' One can ask if there is some finer equivalence relation on the space of representations, such that the notion of relative Anosovness is well-defined with respect to this equivalence relation.

\subsection*{Acknowledgements} 
Zhu was partially supported by Israel Science Foundation grants 18/171 and 737/20.
Zimmer was partially supported by grants DMS-2105580 and DMS-2104381 from the National Science Foundation.
We thank Sara Maloni and Kostas Tsouvalas for ideas and questions which inspired parts of this work. We also thank Fanny Kassel and Ilia Smilga for pointing out a mistake in the first version of this paper, and the referee for a thorough reading and helpful comments.

\section{Preliminaries}

\subsection{Ambiguous notation} \label{sec:notation}
Here we fix any possibly ambiguous notation. 
\begin{itemize}
\item We let $\norm{\cdot}_2$ denote the standard Euclidean norm on $\Kb^d$. 
\item A \emph{metric} $\norm{\cdot}$ on a vector bundle $V \to B$ is a continuous varying family of norms on the fibers each of which is induced by an inner product. 
\item Given a metric space $X$, we will use $\Bc_X(p,r)$ to denote the open ball of radius $r$ centered at $p \in X$ and $\Nc_X(A,r)$ to denote the $r$-neighborhood of a subset $A \subset X$. 
\item Given functions $f,g \colon S \to [0,\infty)$ we write $f \lesssim g$ or equivalently $g \gtrsim f$ if there exists a constant $C > 0$ such that $f(s) \leq C g(s)$ for all $s \in S$. If $f \lesssim g$ and $g \lesssim f$, then we write $f \asymp g$. 

\item Except where otherwise specified, all logarithms are taken to base $e$.
\item Note that constants often carry over between statements in the same section, but not across sections.
\end{itemize}

\subsection{Weak cusp spaces}

Here we recall facts about weak cusp spaces that are used in the paper. For a more in-depth discussion of relative hyperbolicity using the same notation/perspective, we refer the reader to Section 3 in~\cite{ZZ2022a}. 

\begin{definition} Suppose that  $(\Gamma,\peripherals)$ is relatively hyperbolic and $\Gamma$ acts properly discontinuously and by isometries on a proper geodesic Gromov-hyperbolic metric space $X$. If 
\begin{enumerate}
\item $\Gamma$ acts on $\partial_\infty X$ as a geometrically finite convergence group and the maximal parabolic subgroups are exactly $\{ \gamma P \gamma^{-1} : P \in \peripherals, \gamma \in \Gamma\}$,
\item every point in $X$ is within a uniformly bounded distance of a geodesic line, 
\end{enumerate}
then $X$ is a \emph{weak cusp space of $(\Gamma,\peripherals)$}. 
\end{definition} 

The main result in~\cite{Yaman} implies that any relatively hyperbolic group has a weak cusp space. 

By work of Bowditch \cite{Bowditch_relhyp} (also see the exposition in \cite[Sec.\ 3]{HealyHruska}), one can alternatively define weak cusp spaces in terms of the action of $\Gamma$ on $X$.

A relatively hyperbolic group can have non-quasi-isometric weak cusp spaces, see~\cite{Healy2020}. Perhaps the most canonical is the construction due to Groves--Manning, obtained by attaching combinatorial horoballs to a standard Cayley graph. The precise construction is described as follows. 

\begin{definition} Suppose $Y$ is a graph with the simplicial distance $\dist_Y$. The \emph{combinatorial horoball} $\Hc(Y)$ is the graph, also equipped with the simplicial distance, that has vertex set $Y^{(0)} \times \Nb$ and two types of edges:
	\begin{itemize}
		\item \emph{vertical edges} joining vertices $(v,n)$ and $(v,n+1)$, 
		\item \emph{horizontal edges} joining vertices $(v,n)$ and $(w,n)$ when $d_Y(v,w) \leq 2^{n-1}$. 
	\end{itemize}
\end{definition} 

\begin{definition} \label{def:cusp spaces} \label{defn: cusped Cayley graph}
	Let $(\Gamma,\peripherals)$ be a relatively hyperbolic group. A finite symmetric generating set $S \subset \Gamma$ is \emph{adapted} if $S \cap P$ is a generating set of $P$ for every $P \in \peripherals$. Given such an $S$, we let $\Cc(\Gamma, S)$ and $\Cc(P, S \cap P)$ denote the associated Cayley graphs. Then the associated \emph{Groves--Manning cusp space}, denoted $\Cc_{GM}(\Gamma, \peripherals, S)$, is obtained from the Cayley graph $\Cc(\Gamma, S)$ by attaching, for each $P \in \peripherals$ and $\gamma \in \Gamma$, a copy of the combinatorial horoball $\Hc( \gamma\Cc(P, S \cap P))$  by identifying $\gamma\Cc(P, S \cap P)$ with the $n=1$ level of $\Hc( \gamma\Cc(P, S \cap P))$.

\end{definition}

\begin{theorem}[{\cite[Th.\ 3.25]{GrovesManning}}] If $(\Gamma, \peripherals)$ is relatively hyperbolic and $S$ is an adapted finite generating set, then $\Cc_{GM}(\Gamma, \peripherals, S)$ is a weak cusp space for $(\Gamma, \peripherals)$. 
\end{theorem} 

\subsection{The geometry of the Grassmanians} \label{sec:angle metrics}
Throughout the paper, we will let $\dist_{\proj(\Kb^d)}$ denote the \emph{angle distance} on $\proj(\Kb^d)$, that is: if $\ip{\cdot,\cdot}$ is the standard Euclidean inner product on $\Kb^d$, then 
$$
\dist_{\proj(\Kb^d)}([v], [w]) = \cos^{-1} \left( \frac{\abs{\ip{v,w}}}{\sqrt{\ip{v,v}}\sqrt{\ip{w,w}}} \right)
$$
for all non-zero $v,w \in \Kb^d$. 

Using the Pl\"ucker embedding, we can view $\Gr_k(\Kb^d)$ as a subset of $\proj(\wedge^k \Kb^d)$. Let $\dist_{\proj(\wedge^k \Kb^d)}$ denote the angle distance associated to the inner product on $\wedge^k \Kb^d$ which makes 
$$
\{ e_{i_1} \wedge \dots \wedge e_{i_k} : i_1 < \dots < i_k\}
$$
an orthonormal basis. We then let $\dist_{\Gr_k(\Kb^d)}$ denote the distance on $\Gr_k(\Kb^d)$ obtained by restricting $\dist_{\proj(\wedge^k \Kb^d)}$.

\subsection{The singular value decomposition} Given  $g \in \SL(d,\Kb)$, we let 
$$
\mu_1(g) \geq \dots \geq \mu_d(g)
$$
denote the singular values of $g$. By the singular value decomposition, we can write $g = m a \ell$ where $m,\ell \in \SU_{\Kb}(d)$ and $a$ is a diagonal matrix with $\mu_1(g) \geq \dots \geq \mu_d(g)$
down the diagonal. In general this decomposition is not unique, but when $\mu_k(g) > \mu_{k+1}(g)$ the subspace 
$$
U_k(g) := m \ip{e_1,\dots, e_k}
$$
is well-defined. Geometrically, $U_k(g)$ is the subspace spanned by the $k$ largest axes of the ellipse $g \cdot \{ x \in \Kb^d : \norm{x}_2=1\}$. 

We will frequently use the following observation. 

\begin{observation} \label{obs:strongly_dynamics_pres_div_cartan}
Suppose $(g_n)_{n \geq 1}$ is a sequence in $\SL(d,\Kb)$, $V_0 \in \Gr_k(\Kb^d)$, and $W_0 \in \Gr_{d-k}(\Kb^d)$. Then the following are equivalent: 
\begin{enumerate}
\item $g_n(V) \to V_0$ uniformly on compact subsets of 
$$
\left\{ V \in \Gr_k(\Kb^d) : V \text{ transverse to }W_0\right\}.
$$
\item $\frac{\mu_k}{\mu_{k+1}} (g_n) \to \infty$, $U_k(g_n) \to V_0$, and $U_{d-k}(g_n^{-1}) \to W_0$. 
\item There exist open sets $\Oc \subset \Gr_k(\Kb^d)$ and $\Oc^\prime \subset \Gr_{d-k}(\Kb^d)$ such that $g_n(V) \to V_0$ for all $V \in \Oc$ and $g_n^{-1}(W) \to W_0$ for all $W \in \Oc^\prime$. 
\end{enumerate}
\end{observation}

\begin{proof} See for instance Appendix A in~\cite{ZZ2022a}.\end{proof} 

\subsection{Eigenvalues and proximal/weakly unipotent elements}\label{sec:proximal and weakly unipotent elements} Given  $g \in \SL(d,\Kb)$, we let 
$$
\lambda_1(g) \geq \cdots \geq \lambda_d(g) 
$$
denote the absolute values of the eigenvalues of $g$

An element $g \in \SL(d,\Kb)$ is \emph{$\Psf_k$-proximal} if $\lambda_k(g) > \lambda_{k+1}(g)$. In this case, $g$ has a unique attracting fixed point  $V_g^+\in\Gr_k(\Kb^d)$, namely the space corresponding to $\lambda_1(g), \dots, \lambda_k(g)$, and a unique repelling point $W_g^- \in \Gr_{d-k}(\Kb^d)$, namely the space corresponding to $\lambda_{k+1}(g), \dots, \lambda_d(g)$. By writing $g$ is its normal form, it is easy to see that 
$$
g^n(V) \to V_g^+
$$
for all $V \in \Gr_k(\Kb^d)$ transverse to $W_g^-$. Further, $V_g^+ \oplus W_g^- = \Kb^d$. 

An element $g \in \SL(d,\Kb)$ is \emph{weakly unipotent} if $\lambda_j(g) = 1$ for all $j$ and a subgroup $U \leq \SL(d,\Kb)$ is \emph{weakly unipotent} if every element in $U$ is weakly unipotent. 

In~\cite{ZZ2022a} we observed the following. 

\begin{proposition}[{\cite[Prop.\ 4.2]{ZZ2022a}}]\label{prop:eigenvalue data in rel Anosov repn} Suppose that $(\Gamma,\peripherals)$ is relatively hyperbolic and $\rho\colon \Gamma \to \SL(d,\Kb)$ is $\Psf_k$-Anosov relative to $\peripherals$. 
\begin{enumerate}
\item If $P \in \peripherals$, then $\rho(P)$ is weakly unipotent. 
\item If $\gamma \in \Gamma$ is non-peripheral and has infinite order, then $\rho(\gamma)$ is $\mathsf{P}_k$-proximal.
\end{enumerate}
\end{proposition} 

\begin{remark} Recall that an element $\gamma \in \Gamma$ of a relatively hyperbolic group $(\Gamma,\peripherals)$ is non-peripheral if it is not contained in $\bigcup_{\gamma \in \Gamma} \bigcup_{P \in \peripherals} \gamma P \gamma^{-1}$. \end{remark}

\subsection{The symmetric space associated to the special linear group} We will consider the symmetric spaces $M := \SL(d,\Kb) / \SU(d,\Kb)$ normalized so that the distance is given by 
\begin{equation}
\label{eqn:symmetric distance in prelims}
\dist_M\left( g \SU(d,\Kb), h \SU(d,\Kb) \right) = \sqrt{ \sum_{j=1}^d (\log \mu_j(g^{-1} h) )^2 },
\end{equation}
see~\cite[Chap.\ II.10]{BH1999} for more details.

\subsection{Dominated splitting and contraction on Hom bundles} In this section we observe that the exponential contraction of the flow on the Hom bundle described in Section~\ref{sec:results from first paper} can be recast in terms of a dominated splitting condition. This is well known in the word-hyperbolic case~\cite{BCLS2015, BPS} and the same arguments work in the relative case as well. 

Suppose, for the rest of this section, that $(\Gamma, \peripherals)$ is a relatively hyperbolic group, $\rho \colon \Gamma \to \SL(d,\Kb)$ is a representation, $X$ is a weak cusp space for $(\Gamma, \peripherals)$, and $\norm{\cdot}$ is a metric on the vector bundle $\wh{E}_\rho(X) \to \wh{\Gc}(X)$. 

If $V,W \subset \wh{E}_\rho(X)$ are subbundles, we can consider the bundle $\Hom(V,W) \to \wh{\Gc}(X)$ with the associated family of operator norms defined by
$$
\norm{f}_\sigma : = \max \left\{\norm{f(Y)}_{\sigma} : Y \in V|_\sigma, \ \norm{Y}_{\sigma} = 1\right\}
$$
when $f \in \Hom(V,W)|_\sigma$. In particular, given a continuous $\rho$-equivariant transverse map 
$$
\xi \colon \partial(\Gamma, \peripherals) \to \Gr_k(\Kb^d) \times \Gr_{d-k}(\Kb^d)
$$
let $\wh{\Theta}^k, \wh{\Xi}^{d-k} \subset \wh{E}_\rho(X)$ denote the subbundles defined in Section~\ref{sec:results from first paper} and endow 
$$
\Hom\left( \wh{\Xi}^{d-k}, \wh{\Theta}^{k}\right) \to \wh{\Gc}(X)
$$
with the operator norm. We then have the following connection between the dynamics on these bundles. 

\begin{proposition}\label{prop: Hom bundles contraction/expansions} With the notation above and $c,C > 0$ fixed, the following are equivalent: 
\begin{enumerate}
\item For all $t \geq 0$, $\sigma \in \wh{\Gc}(X)$, $Y \in \wh{\Theta}^k|_{\sigma}$, and $Z\in \wh{\Xi}^{d-k}|_{\sigma}$  non-zero,  
$$
\frac{\norm{\flatflow^t(Y) }_{\geodflow^t(\sigma)}}{\norm{\flatflow^t(Z) }_{\geodflow^t(\sigma)}} \leq Ce^{-ct} \frac{\norm{Y}_{\sigma} }{\norm{Z}_{\sigma} }. 
$$
\item For all $t \geq 0$, $\sigma \in \wh{\Gc}(X)$, and $f \in \Hom\left( \wh{\Xi}^{d-k}, \wh{\Theta}^{k}\right)|_{\sigma}$, 
$$
\norm{\homflow^t(f)}_{\geodflow^t(\sigma)} \leq C e^{-ct} \norm{f}_\sigma.
$$
\end{enumerate}
\end{proposition} 

\begin{proof} 
One can argue exactly as in Proposition 2.3 in~\cite{BCLS2015}.
\end{proof}

\part{Representations of geometrically finite groups in negatively-curved symmetric spaces} 

\section{Reminders on negatively-curved symmetric spaces} \label{sec:reminders negatively curved symspaces} 

Suppose $\mathsf{G}$ is a connected simple non-compact Lie group with rank one and finite center. Fix a maximal compact subgroup $\mathsf{K} \leq \mathsf{G}$, then the quotient manifold $X = \mathsf{G}/\mathsf{K}$ is simply connected and has a $\mathsf{G}$-invariant negatively-curved symmetric Riemannian metric. The possible spaces $X$ are described in ~\cite[Chap.\ 19]{Mostow1973}.

Since $X$ is simply connected and has pinched negative curvature, it is Gromov-hyperbolic, and we will let $\partial_\infty X$ denote the Gromov boundary of $X$. We will also let $T^1 X$ denote the unit tangent bundle of $X$ and let $\pi \colon T^1 X \to X$ denote the natural projection. We will use $\geodflow^t$ to denote the geodesic flow on $T^1 X$. Also, for $v \in T^1 X$ we let $v^+, v^- \in \partial_\infty X$ denote the forward/backward endpoint of the geodesic line tangent to $v$, equivalently 
$$
v^\pm = \lim_{t \rightarrow \pm \infty} \pi( \geodflow^t(v)). 
$$

By construction $\mathsf{G}$ acts isometrically on $X$. The induced homomorphism $\Phi\colon \mathsf{G} \to \Isom(X)$ maps onto $\Isom_0(X)$, the connected component of the identity, and has kernel $Z(\mathsf{G})$, the center of $\mathsf{G}$. Given a sequence $(g_n)_{n \geq 1}$ and $x \in \partial_\infty X$, we write 
\begin{equation}\label{eqn:boundary compactification of G}
g_n \to x 
\end{equation}
if $g_n(p) \to x$ for some (any) $p \in X$.

An element of $\mathsf{G}$ is either 
\begin{itemize}
\item \emph{elliptic}, that is it fixes a point in $X$,
\item \emph{parabolic}, that is it is not elliptic and fixes exactly one point in $\partial_\infty X$, or 
\item \emph{loxodromic}, that is it is not elliptic and fixes exactly two points in $\partial_\infty X$.
\end{itemize} 
Parabolic and loxodromic elements have the following behavior:
\begin{enumerate}
\item If $g \in \mathsf{G}$ is parabolic and $x_g^+$ is the unique fixed point of $g$, then 
$$
\lim_{n \to \pm \infty} g^n(y) =x_g^+
$$ 
for all $y \in (X \cup \partial_\infty X) \smallsetminus \{x_g^+\}$. 
\item If $g \in \mathsf{G}$ is loxodromic, then it is possible to label the fixed points of $g$ as $x_g^+, x_g^-$ so that 
$$
\lim_{n \to \pm \infty} g^n(y) =x_g^\pm
$$
for all $y \in (X \cup \partial_\infty X) \smallsetminus \{x_g^\mp\}$. 
\end{enumerate}
In both cases, the limits are locally uniform.

Given a discrete subgroup $\Gamma \leq \mathsf{G}$, we can consider the limit set $\Lambda_X(\Gamma) \subset \partial_\infty X$ of all accumulation points of any $\Gamma$-orbit in $X$. We then define $\Cc_X(\Gamma)$ to be the convex hull of $\Lambda_X(\Gamma)$ in $X$, i.e. the smallest closed geodesically convex subset of $X$ whose closure in $X \cup \partial_\infty X$ contains $\Lambda_X(\Gamma)$. Finally, we define $\Uc(\Gamma)$ to be the subspace of the unit tangent bundle $T^1 X$ consisting of vectors tangent to geodesics with both endpoints in the limit set $\Lambda_X(\Gamma)$ and let $\wh{\Uc}(\Gamma) = \Gamma \backslash \Uc(\Gamma)$. 

\begin{example}
If $\Gamma$ is a lattice in $\mathsf{G}$, then $\Lambda_X(\Gamma) = \partial_\infty X$, $\Cc_X(\Gamma) = X$, and $\Uc(\Gamma) = T^1 X$.
\end{example} 

A discrete group $\Gamma \leq \mathsf{G}$ acts as a convergence group on $\partial_\infty X$ and such a group is \emph{geometrically finite} if it acts its limit set $\Lambda_X(\Gamma)$ as a geometrically finite convergence group (for definitions, see for instance \cite[Sec.\ 3.3]{ZZ2022a}). There are also equivalent characterizations in terms of the action of $\Gamma$ on $X$, see \cite{Bowditch_GF_Riem}. 

In this case, if $\peripherals$ is a set of representatives of the conjugacy classes of maximal parabolic subgroups in $\Gamma$, then $(\Gamma, \peripherals)$ is a relatively hyperbolic group. Moreover, $\Cc_X(\Gamma)$ is a weak cusp space of $(\Gamma, \peripherals)$ (see the ``F4'' definition and Section 3.5 in \cite{Bowditch_GF_Riem}). The flow space $\Uc(\Gamma)$ then naturally identifies with the space of geodesics $\Gc( \Cc_X(\Gamma))$ in $\Cc_X(\Gamma)$. When considering a relatively Anosov representation $\rho$ of $\Gamma$ it is more convenient to view the bundles in Definition~ \ref{defn:flow definition} as having base $\wh{\Uc}(\Gamma)$.

\section{Representations of rank one groups}\label{sec:homogeneous_repn}

Let $\mathsf{G}$, $\mathsf{K}$, and $X=\mathsf{G}/\mathsf{K}$ be as in Section~\ref{sec:reminders negatively curved symspaces}.
In this section we will prove the following expanded version of Proposition~\ref{prop:homog_cusps} from the introduction. First we present a definition. 

\begin{definition} Given a representation $\tau \colon \mathsf{G} \to \SL(d,\Kb)$, we say that a continuous $\tau$-equivariant map $\zeta \colon \partial_\infty X \to \Gr_k(\Kb^d) \times \Gr_{d-k}(\Kb^d)$ is 
\begin{enumerate} 
\item \emph{transverse}: if $x,y \in \partial_\infty X$ are distinct, then $\zeta^k(x) \oplus \zeta^{d-k}(y) = \Kb^d$, 
\item \emph{strongly dynamics preserving}: if $(g_n)_{n \geq 1}$ is a sequence of elements in $\mathsf{G}$ where $\gamma_n \to x \in \partial_\infty X$ and $\gamma_n^{-1} \to y \in \partial_\infty X$ (here we use the notation from Equation~\eqref{eqn:boundary compactification of G}), then 
$$
\lim_{n \to \infty} \tau(\gamma_n)V = \zeta^k(x)
$$
for all $V \in \Gr_k(\Kb^d)$ transverse to $\zeta^{d-k}(y)$. 
\end{enumerate}
\end{definition} 

\begin{proposition}\label{prop:representations of rank one groups} If $\tau \colon \mathsf{G} \to \SL(d,\Kb)$ is $\Psf_k$-proximal (i.e.\ $\tau(\mathsf{G})$ contains a $\Psf_k$-proximal element) and $\norm{\cdot}_{v \in T^1 X}$ is a $\tau$-equivariant family of norms on $\Kb^d$, then the following statements hold:
\begin{enumerate}
\item There exists a continuous $\tau$-equivariant, transverse, strongly dynamics preserving map
$$
\zeta_\tau =(\zeta_\tau^k, \zeta_\tau^{d-k})\colon \partial_\infty X \to \Gr_k(\Kb^d) \times \Gr_{d-k}(\Kb^d).
$$
 \item There exist $C, c > 0$ such that: if $t \geq 0$, $v \in T^1 X$, $Y \in \zeta_\tau^k(v^+)$, and $Z \in \zeta_\tau^{d-k}(v^-)$ is non-zero, then 
 \begin{align*}
 \frac{\norm{Y}_{\geodflow^t(v)}}{\norm{Z}_{\geodflow^t(v)}} \leq C e^{-ct} \frac{\norm{Y}_v}{\norm{Z}_v}.
 \end{align*}
 \item For any $r > 0$ there exists $L_r > 1$ such that: if $v,w \in T^1 X$ satisfy $\dist_X(\pi(v), \pi(w)) \leq r$, then 
 $$
 \frac{1}{L_r} \norm{\cdot}_v \leq \norm{\cdot}_w \leq L_r \norm{\cdot}_v.
 $$
 \end{enumerate}
 In particular,  if $\Gamma \leq \mathsf{G}$ is geometrically finite, then $\rho = \tau|_{\Gamma}$ is uniformly $\Psf_k$-Anosov relative to $\Cc_X(\Gamma)$.
 \end{proposition} 
 
 \begin{remark} To be precise, a family of norms $\norm{\cdot}_{v \in T^1 X}$ is $\tau$-equivariant if 
 $$
 \norm{\cdot}_v = \norm{\tau(g)(\cdot)}_{g(v)}
 $$
 for all $v \in T^1 X$ and $g \in \mathsf{G}$. 
 \end{remark}

 The rest of the section is devoted to the proof of the proposition. So fix a representation $\tau \colon \mathsf{G} \to \SL(d,\Kb)$ as in the statement.

Let $p_0:=[\mathsf{K}] \in X$ and notice that $\mathsf{K} = {\rm Stab}_{\mathsf{G}}(p_0)$. Fix a unit vector $v_0 \in T^1_{p_0} X$ and a Cartan subgroup $\mathsf{A}=\{a_t\}$ of $\mathsf{G}$ such that $t \mapsto a_t(p_0)$ parametrizes the geodesic through $p_0$ with initial velocity $v_0$. Let $\mathsf{M}$ denote the centralizer of $\mathsf{A}$ in $\mathsf{K}$. 
 
We can conjugate $\tau$ so that $\tau(\mathsf{A})$ is a subgroup of the diagonal matrices and $\tau(\mathsf{K}) \leq \SU(d,\Kb)$, see for instance~\cite{M1955}.
 
The next two lemmas are used to define the maps in part (1) of the proposition.

 \begin{lemma} If $t > 0$, then $\tau(a_t)$ is $\Psf_k$-proximal. \end{lemma} 
 
 \begin{proof} By hypothesis, there exists $g \in \mathsf{G}$ such that $\tau(g)$ is $\Psf_k$-proximal. By the Cartan decomposition,  there exist $m_n, \ell_n \in \mathsf{K}$ and $t_n \to \infty$ such that $g^n = m_n a_{t_n} \ell_n$. Then 
 \begin{align*}
 0 < \log \frac{\lambda_k}{\lambda_{k+1}}(\tau(g)) &= \lim_{n \to \infty} \frac{1}{n} \log \frac{\mu_k}{\mu_{k+1}}(\tau(g^n)) =  \lim_{n \to \infty} \frac{1}{n} \log \frac{\mu_k}{\mu_{k+1}}(\tau(a_{t_n}))\\
 &=  \lim_{n \to \infty} \frac{1}{n} \log \frac{\lambda_k}{\lambda_{k+1}}(\tau(a_{t_n}))
 \end{align*}
 (the first equality follows from Gelfand's formula for the spectral radius applied to the linear operators $\wedge^k g$ and $\wedge^{k+1}g$; in the last equality we use the fact that $\tau(a_t)$ is diagonal). So when $n$ is large $\lambda_k(\tau(a_{t_n})) > \lambda_{k+1}(\tau(a_{t_n}))$. Since $\tau(a_t)$ is diagonal, this implies that $\lambda_k(\tau(a_t)) > \lambda_{k+1}(\tau(a_t))$ for all $t > 0$. 
 \end{proof} 
 
 Let $V^+ \in \Gr_k(\Kb^d)$ and $V^- \in \Gr_{d-k}(\Kb^d)$ denote the attracting and repelling fixed points of $\tau(a_t)$ when $t > 0$. Then $\Kb^d = V^+ \oplus V^-$ and 
 \begin{equation}
 \label{eqn: a_t acts nicely}
 \lim_{t \to \infty} \tau(a_t) V = V^+
 \end{equation}
 for all $V \in \Gr_k(\Kb^d)$ transverse to $V^-$. Let $\mathsf{P}^\pm$ denote the stabilizer of $v_0^\pm \in \partial_\infty X$ in $\mathsf{G}$.

 \begin{lemma} $\tau(\mathsf{P}^\pm) V^\pm = V^\pm$. \end{lemma}
 
 \begin{proof} Fix $g \in \mathsf{P}^+$. Then 
 $$
g^\prime : = \lim_{t \to \infty} a_{-t} g a_{t}
 $$
exists and is contained in $\mathsf{M}\mathsf{A}$, see for instance~\cite[Prop.\ 2.17.3]{E1996}. Since $\mathsf{M}$ commutes with $\mathsf{A}$, $\tau(\mathsf{M})$ fixes $V^+$. Hence $\tau(g^\prime) V^+ = V^+$. So 
 $$
  \lim_{t \to \infty} \tau(a_{-t} g a_t) V^+ = \tau(g^\prime)V^+=V^+
  $$
  which implies, by Equation~\eqref{eqn: a_t acts nicely}, that
 $$
\tau(g)V^+ = \lim_{t \to \infty} \tau(a_{t}) \tau(a_{-t} g a_t) V^+ =  V^+. 
$$
Thus $\tau(\mathsf{P}^+)V^+= V^+$. 

Similar reasoning shows that $\tau(\mathsf{P}^-)V^- = V^-$. 
\end{proof} 
 
 Since $\mathsf{G}$ acts transitively on $\partial_\infty X$ and ${\rm Stab}_{\mathsf{G}}(v_0^\pm) = \mathsf{P}^\pm$, the last lemma implies that the expressions 
 $$
 \zeta^k( gv_0^+) = \tau(g) V^+ \quad \text{and} \quad \zeta^{d-k}(g v_0^-) = \tau(g) V^- \quad \text{for all} \quad g \in \mathsf{G}
 $$ 
define a smooth $\tau$-equivariant map $\zeta=(\zeta^k, \zeta^{d-k}) \colon \partial_\infty X \to \Gr_k(\Kb^d) \times \Gr_{d-k}(\Kb^d)$. 

\begin{lemma} $\zeta$ is transverse. \end{lemma} 

\begin{proof} Fix $x,y \in \partial_\infty X$ distinct. Since $\mathsf{G}$ acts transitively on pairs of distinct points in $\partial_\infty X$, there exists $g \in \mathsf{G}$ such that $(x,y) = g \cdot (v_0^+, v_0^-)$. Then
 \begin{equation*}
 \zeta^k(x) +  \zeta^{d-k}(y) = \tau(g) (\zeta^k(v_0^+) + \zeta^{d-k}(v_0^-)) = \tau(g) (V^+ + V^-)=\Kb^d. \qedhere
\end{equation*}
\end{proof} 

\begin{lemma} $\zeta$ is strongly dynamics preserving. 
 \end{lemma} 
 
 \begin{proof} Suppose that $(g_n)_{n \geq 1}$ is a sequence in $\mathsf{G}$ such that $g_n \to x \in \partial_\infty X$ and $g_n^{-1} \to y \in \partial_\infty X$. By the Cartan decomposition, there exist $m_n,\ell_n \in \mathsf{K}$ and $t_n \to \infty$ such that $g_n = m_n a_{t_n} \ell_n$. Passing to a subsequence we can suppose that $m_n \to m$ and $\ell_n \to \ell$. Then $m_n(v_0^+) \to m(v_0^+)=x$ and $\ell_n^{-1}(v_0^-) \to \ell^{-1}(v_0^-)=y$. Then by Equation~\eqref{eqn: a_t acts nicely}, if $V \in \Gr_k(\Kb^d)$ is transverse to $\xi^{d-k}(y) = \tau(\ell)^{-1} V^-$, then 
 $\tau(\ell_n) V$ is transverse to $\tau(\ell_n) \tau(\ell)^{-1} V^-$ and hence, for large $n$, to $V^-$, and so
\begin{equation*}
\lim_{n \to \infty} \tau(g_n) V = \tau(m) \lim_{n \to \infty} \tau(a_{t_n}) \tau(\ell_n) V = \tau(m) V^+ = \xi^k(x). \qedhere
\end{equation*}
\end{proof} 

Next we prove parts (2) and (3). Since any two families of $\tau$-equivariant norms are bi-Lipschitz, it is enough to consider the norms
$$
\norm{\cdot}_{g(v_0)} := \norm{\tau(g)^{-1}(\cdot)}_{2}
$$
where $\norm{\cdot}_2$ is the standard Euclidean norm. Since $\tau(\mathsf{K}) \leq \mathsf{U}(d,\Kb)$ and $\mathsf{K} =  {\rm Stab}_{\mathsf{G}}(p_0)$, this is indeed a well-defined family.

Since each $\tau(a_t)$ is $\Psf_k$-proximal and diagonal, there exists $\lambda > 0$ such that 
$$
\frac{\lambda_{k}}{\lambda_{k+1}}(\tau(a_t)) = e^{\lambda t} 
$$
when $t \geq 0$. 

\begin{lemma}\label{lem: homog implies anosov} If $t \geq 0$, $v \in T^1 X$, $Y \in \zeta^k(v^+)$, and $Z \in \zeta^{d-k}(v^-)$ is non-zero, then 
 \begin{align*}
 \frac{\norm{Y}_{\geodflow^t(v)}}{\norm{Z}_{\geodflow^t(v)}} \leq e^{-\lambda t} \frac{\norm{Y}_v}{\norm{Z}_v}.
 \end{align*}
 \end{lemma}

 \begin{proof} Fix $g \in \mathsf{G}$ such that $g(v_0)=v$. Then $\geodflow^t(v) = g\geodflow^t(v_0) = g a_t(v_0)$ for all $t$ and $g(V^+, V^-) = (\zeta^k(v^+), \zeta^{d-k}(v^-))$. Since $\tau(\mathsf{A})$ is a subgroup of the diagonal matrices and $V^+,V^-$ are the attracting, repelling spaces of $\tau(a_t)$ when $t > 0$, then 
   \begin{align*}
 \frac{\norm{Y}_{\geodflow^t(v)}}{\norm{Z}_{\geodflow^t(v)}} = \frac{\norm{\tau(a_{-t} g^{-1})Y}_{2}}{\norm{\tau(a_{-t} g^{-1}) Z}_{2}} \leq \frac{\frac{1}{\lambda_{k}(\tau(a_t))}}{\frac{1}{\lambda_{k+1}(\tau(a_t))}}\frac{\norm{\tau(g^{-1})Y}_{2}}{\norm{\tau(g^{-1}) Z}_{2}} =e^{-\lambda t} \frac{\norm{Y}_v}{\norm{Z}_v}.
 \end{align*}
 \end{proof} 
 
 Since $\tau(\mathsf{A})=\{\tau(a_t)\}$ is a one-parameter group of diagonal matrices, there exists $\mu > 0$ such that 
$$
\frac{\mu_{1}}{\mu_{d}}(\tau(a_t)) = e^{\mu t} 
$$
when $t \geq 0$. 
 
 \begin{lemma}\label{lem:homog are uniform} If $v_1, v_2 \in T^1 X$, then 
 $$
 e^{-\mu \dist_X(\pi(v_1), \pi(v_2))}\norm{\cdot}_{v_1}  \leq \norm{\cdot}_{v_2} \leq  e^{\mu \dist_X(\pi(v_1), \pi(v_2))}\norm{\cdot}_{v_1}.
 $$
 \end{lemma} 
 
 \begin{proof} Since the family of norms is $\tau$-equivariant, it is enough to consider the case where $v_1 = v_0$ and $v_2 = g(v_0)$. By the Cartan decomposition, there exist $m,\ell \in \mathsf{K}$ and $t \geq 0$ such that $g = ma_t \ell$. Notice that 
 $$
 \dist_X(\pi(v_1), \pi(v_2)) = \dist_X(p_0, ma_t\ell(p_0)) = \dist_X(p_0, a_t(p_0)) = t
 $$
 since $t \mapsto a_t(p_0)$ is a unit speed geodesic. Further, $\norm{\cdot}_{v_1} = \norm{\cdot}_2$ and
 $$
 \norm{\cdot}_{v_2} = \norm{\tau(g^{-1})(\cdot)}_2 = \norm{\tau(ma_{t})^{-1}( \cdot)}_2.
 $$
 Hence 
 $$
\frac{\mu_{d}}{\mu_{1}}(\tau(ma_{t})^{-1})  \norm{\cdot}_{v_1}  \leq \norm{\cdot}_{v_2} \leq  \frac{\mu_{1}}{\mu_{d}}(\tau(ma_{t})^{-1}) \norm{\cdot}_{v_1}.
 $$
 Since $\frac{\mu_{1}}{\mu_{d}}(\tau(ma_{t})^{-1}) = \frac{\mu_{1}}{\mu_{d}}(\tau(a_{-t}))= e^{\mu  t}$, the lemma follows. 
 \end{proof} 
 
 \begin{lemma}  If $\Gamma \leq \mathsf{G}$ is geometrically finite, then $\rho = \tau|_{\Gamma}$ is uniformly $\Psf_k$-Anosov relative to $\Cc_X(\Gamma)$.\end{lemma}
 
 \begin{proof} Recall that $\Cc_X(\Gamma)$ is a weak cusp space of $\Gamma$ and $\Uc(\Gamma)$ naturally identifies with the space of geodesic lines in $\Cc_X(\Gamma)$.  Then Lemma~\ref{lem: homog implies anosov}, Lemma~\ref{lem:homog are uniform}, and Proposition~\ref{prop: Hom bundles contraction/expansions} imply that $\rho = \tau|_{\Gamma}$ is uniformly $\Psf_k$-Anosov relative to $\Cc_X(\Gamma)$.
 \end{proof}

\section{Almost homogeneous cusps}

Let $\mathsf{G}$, $\mathsf{K}$, and $X=\mathsf{G}/\mathsf{K}$ be as in Section~\ref{sec:reminders negatively curved symspaces}. In this section we consider the following setup: 
\begin{enumerate}
\item $\Gamma_0 \leq \mathsf{G}$ is a finitely generated discrete group which fixes a horoball $H \subset X$ and $\overline{H} \cap \partial_\infty X = \{\eta^+\}$.
\item $\tau \colon \mathsf{G} \to \SL(d,\Kb)$ is a $\Psf_k$-proximal representation and 
$$
\zeta_\tau \colon \partial_\infty X \to \Gr_k(\Kb^d) \times \Gr_{d-k}(\Kb^d)
$$
 is the boundary map constructed in Proposition~\ref{prop:representations of rank one groups}. 
\item  \label{cond:three} $\rho \colon \Gamma_0 \to \SL(d,\Kb)$ is a representation where $\left\{ \tau(g)\rho(g)^{-1} : g \in \Gamma_0 \right\}$ is relatively compact in $\SL(d,\Kb)$.
\item $\Lc \subset \partial_\infty X$ is a closed, $\Gamma_0$-invariant set where the quotient $\Gamma_0 \backslash (\Lc \smallsetminus \{\eta^+\})$ is compact.
\item $ \xi \colon \Lc \to \Gr_k(\Kb^d) \times \Gr_{d-k}(\Kb^d)$ is continuous, $\rho$-equivariant, transverse, and $\xi(\eta^+) = \zeta_\tau(\eta^+)$. 
\item $\Uc := \{ v \in T^1 X : v^+, v^- \in \Lc\}$. 
\end{enumerate}

In the next section we will apply the results of this section to the case where $\Gamma_0$ is a peripheral subgroup in a geometrically finite group $\Gamma \leq \mathsf{G}$, $\Lc$ is the limit set of $\Gamma$, and $\Uc$ is the flow space $\Uc(\Gamma)$. 

The first result establishes a type of infinitesimal homogeneity of a limit curve at the fixed point of a peripheral subgroup.

 \begin{proposition}\label{prop:local homogeneity} 
 If 
 \begin{itemize}
 \item $\gamma \in \mathsf{G}$ is a hyperbolic element with $\gamma^+ = \eta^+$, and
 \item $(x_n)_{n \geq 1} \subset \partial_\infty X$ is a sequence where $\{\gamma^n(x_n)\} \subset \Lc$ and $x_n \to x \in \partial_\infty X$,
 \end{itemize}
 then 
 $$
 \lim_{n \to \infty} \tau(\gamma)^{-n} \circ \xi \circ \gamma^n(x_n) = \zeta_\tau(x).
 $$
\end{proposition} 

\begin{remark} In the case when $\Lc = \partial_\infty X$, this says that $\tau(\gamma)^{-n} \circ \xi \circ \gamma^n$ converges uniformly to $\zeta_\tau$. 

\end{remark}

The second result constructs good norms over the horoball $H$. It will be helpful to use the following notation: given a subset $S \subset X$, let 
$$
\Uc|_S :=  \Uc\cap \bigcup_{p \in S} T_p^1 X.
$$
  
 \begin{proposition}\label{prop:almost homogeneous norms} There exists a $\rho$-equivariant family of norms $\norm{\cdot}_{v \in T^1 X}$ on $\Kb^d$ with the following properties:
 \begin{enumerate}
 \item Each $\norm{\cdot}_v$ is induced by an inner product.
 \item For any $r > 1$ there exists $L_r > 1$ such that: if $v,w \in \Uc|_H$ satisfy $\dist_X(\pi(v), \pi(w)) \leq r$, then 
 $$
 \frac{1}{L_r} \norm{\cdot}_v \leq \norm{\cdot}_{w} \leq L_r \norm{\cdot}_v.
 $$
\item There exist $C_1,c_1 > 0$ and a horoball $H^\prime \subset H$ such that: if $t \geq 0$ and $v, \geodflow^t(v) \in \Uc|_{H^\prime}$, then 
$$
\frac{\norm{Y}_{\geodflow^t(v)}}{\norm{Z}_{\geodflow^t(v)}} \leq C_1e^{-c_1t}\frac{\norm{Y}_v}{\norm{Z}_v} 
$$
for all $Y \in \xi^k(v^+)$ and non-zero $Z \in \xi^{d-k}(v^-)$. 
 \end{enumerate}
 \end{proposition}

 \subsection{Proof of Proposition~\ref{prop:local homogeneity}} 
The following argument is similar to the proof of ~\cite[Prop.\ 5.3]{CZZ2021}.
 
Fix a Riemannian distance $\dist_{\Fc}$ on $\Gr_k(\Kb^d) \times \Gr_{d-k}(\Kb^d)$.  Suppose the proposition is false. Then there exist $x \in \partial_\infty X$, a sequence $(x_j)_{j \geq 1}$ in $\partial_\infty X$, and a sequence $(n_j)_{j \geq 0}$ in $\Nb$ such that $x_j \to x$, $n_j \to \infty$, $\{ \gamma^{n_j}(x_j) \} \subset \Lc$, and $\tau(\gamma)^{-n_j} \circ \xi \circ \gamma^{n_j}(x_j)$ does not converge to $\zeta_\tau(x)$. After passing to a subsequence there exists $\epsilon > 0$ such that 
$$
\inf_{j \geq 1} \dist_{\Fc}\left( \tau(\gamma)^{-n_j} \circ \xi \circ \gamma^{n_j}(x_j), \zeta_\tau(x) \right) \geq \epsilon.
$$

Notice that 
$$
\tau(\gamma)^{-n_j} \circ \xi \circ \gamma^{n_j}(\eta^+) = \tau(\gamma)^{-n_j} \circ \xi(\eta^+) = \tau(\gamma)^{-n_j} \circ \zeta_\tau(\eta^+) = \zeta_\tau(\eta^+)
$$
and so after possibly passing to a subsequence $x_j \neq \eta^+$ for all $j$. Then there exists a sequence $(h_j)_{j \geq 1}$ in $\Gamma_0$ such that $y_j : = h_j \gamma^{n_j}(x_j)$ is relatively compact in $\Lc\smallsetminus\{\eta^+\}$. Passing to a subsequence we can suppose that $y_j \to y \in \Lc\smallsetminus\{\eta^+\}$ and 
$$
g:=\lim_{j \to \infty} \tau(h_j) \rho(h_j)^{-1}  \in \SL(d,\Kb). 
$$

Notice that 
$$
g\,\xi(\eta^+) = \lim_{j \to \infty} \tau(h_j) \rho(h_j)^{-1}\xi(\eta^+) = \lim_{j \to \infty} \tau(h_j) \xi(\eta^+)=\lim_{j \to \infty} \tau(h_j) \zeta_\tau(\eta^+) = \zeta_\tau(\eta^+).
$$
Then since $\xi(y)$ is transverse to $\xi(\eta^+)$, we see that $g\,\xi(y)$ is transverse to $\zeta_\tau(\eta^+)$. 

Also, by construction, $h_j \gamma^{n_j}(p) \to \eta^+$ for all $p \in X$. By passing to a subsequence, we can suppose that 
$$
z:=\lim_{j \to \infty} \gamma^{-n_j}h_j^{-1}(p) \in \partial_\infty X
$$
for all $p \in (X \cup \partial_\infty X) \smallsetminus \{\eta^+\}$ and the convergence is locally uniform. Since $\gamma^{-n_j}h_j^{-1}(y_j) = x_j \to x$, and $\{y_j\}$ is relatively compact in $(X \cup \partial_\infty X) \smallsetminus \{\eta^+\}$, we must have $z=x$. So, by the strongly dynamics preserving property of $\tau$,
$$
\lim_{j \to \infty} \tau(\gamma^{-n_j} h_j^{-1}) F =  \zeta_\tau(x)
$$
for all $F=(F^k,F^{d-k}) \in \Gr_k(\Kb^d) \times \Gr_{d-k}(\Kb^d)$  transverse to $\zeta_\tau(\eta^+)$.

Finally, 
$$
\lim_{j \to \infty}   \tau(\gamma)^{-n_j} \circ \xi \circ \gamma^{n_j}(x_j) = \lim_{j \to \infty}   \tau(\gamma^{-n_j} h_j^{-1})\tau(h_j) \rho(h_j)^{-1}  \xi(y_j) = \zeta_\tau(x)
$$
since $\tau(h_j) \rho(h_j)^{-1}  \xi(y_j) \to g\,\xi(y)$ and $g\,\xi(y)$ is transverse to $\zeta_\tau(\eta^+)$. Thus we have a contradiction. 

\subsection{Proof of Proposition~\ref{prop:almost homogeneous norms}} 

Let $\pi_{\partial H} \colon X \to \partial H$ be the map where $\pi_{\partial H}(p)$ is the unique point in $\partial H$ contained in the geodesic line passing through $p$ and limiting to $\eta^+$. 

\begin{lemma} There exists a smooth function $\chi \colon X \to [0,1]$ such that 
\begin{enumerate}
\item $\chi \circ \pi_{\partial H} = \chi$,
\item $\{ \chi \circ g\}_{g \in \Gamma_0}$ is a partition of unity (i.e.\  $\sum_{g \in \Gamma_0} \chi \circ g$ is a locally finite sum which equals one everywhere). 
\end{enumerate}
\end{lemma} 

\begin{proof} By Selberg's lemma there exists a finite-index torsion-free subgroup $\Gamma_0^\prime \leq \Gamma_0$. Let $n = [\Gamma_0:\Gamma_0^\prime]$.

Consider the manifold quotient $p \colon \partial H \to \Gamma_0^\prime \backslash \partial H$. Fix an open cover $\{ U_i\}_{i \in I}$ of $\Gamma_0^\prime \backslash \partial H$ such that for all $i \in I$ there is a local inverse $U_i \to \wt{U}_i \subset \partial H$ to $p$. Fix a partition of unity $\{\chi_i\}_{i \in I}$ of $\Gamma_0^\prime \backslash \partial H$ subordinate to $\{U_i\}_{i \in I}$. Then for each $i \in I$, let $\wt{\chi}_i \colon \partial H \to [0,1]$ be the lift of $\chi_i$ to $\wt{U}_i$. Finally, let 
$$
\chi:= \frac{1}{n} \sum_{ i \in I} \wt{\chi}_i\circ \pi_{\partial H}.
$$ 

By construction, 
$$
\sum_{g \in \Gamma_0^\prime} \sum_{ i \in I} \wt{\chi}_i\circ \pi_{\partial H} \circ g
$$ 
is a locally finite sum which equals one everywhere. So if $\{\Gamma_0^\prime g_1,\dots, \Gamma_0^\prime g_n \} = \Gamma_0^\prime \backslash \Gamma_0$, then 
$$
\sum_{g \in \Gamma_0} \chi \circ g = \frac{1}{n}  \sum_{k=1}^{n} \left(\sum_{g \in \Gamma_0^\prime} \sum_{ i \in I} \wt{\chi}_i\circ \pi_{\partial H} \circ g\right) \circ g_k
$$ 
is a locally finite sum which equals one everywhere.
\end{proof}

Fix $v_0 \in \Uc$ with $p_0:=\pi(v_0) \in \partial H$ and $v_0^+ = \eta^+$. By conjugating $\mathsf{K}$, we may assume that 
$$
\mathsf{K}= {\rm Stab}_{\mathsf{G}}(p_0).
$$
Since $\mathsf{K}$ is compact, there exists a $\tau(\mathsf{K})$-invariant norm $\norm{\cdot}^{(0)}$ on $\Kb^d$ which is induced by an inner product. Then 
$$
\norm{\cdot}^{(0)}_{gv_0} := \norm{\tau(g)^{-1}(\cdot)}^{(0)}
$$
defines a smooth $\tau$-equivariant family of norms indexed by $T^1X$ where each norm is induced by an inner product. 

Then given $v \in T^1 X$ define
$$
\norm{\cdot}_v = \sqrt{ \sum_{g \in \Gamma_0}  (\chi\circ g)(\pi(v))\left(\norm{\rho(g)(\cdot)}_{gv}^{(0)} \right)^2 }. 
$$
Since $\{\chi \circ g\}_{g \in \Gamma_0}$ is a partition of unity, $\norm{\cdot}_{v \in T^1 X}$ is a smooth family of norms where each $\norm{\cdot}_v$ is induced by an inner product. One can check that it is $\rho$-equivariant. We will show that this family of norms satisfies the remaining conditions in the proposition.

We start by showing some useful compactness / cocompactness properties. Let $\{ a_t\} \leq \mathsf{G}$ be a Cartan subgroup such that $a_t(v_0) = \geodflow^t(v_0)$ for all $t \in \Rb$. 

\begin{lemma}\label{lem:upgrading relative compactness} The set 
\begin{align*}
\left\{ \tau(a_{-t})\tau(g)\rho(g)^{-1}\tau(a_t) : g \in \Gamma_0, t \geq 0 \right\}
\end{align*}
is relatively compact in $\SL(d,\Kb)$. 
\end{lemma}

\begin{proof}  By~\cite{M1955} and conjugating $\tau$ and $\rho$ we may assume that 
 \begin{align*}
 \tau(a_t) = \begin{pmatrix} e^{\lambda_1 t} \id_{d_1} & & \\ & \ddots & \\ & & e^{\lambda_{m+1}t} \id_{d_{m+1}} \end{pmatrix}
 \end{align*}
 where $\lambda_1 > \cdots > \lambda_{m+1}$. Since $\tau(a_{-t})$ is conjugate to $\tau(a_t)$, notice that $d_{k} = d_{m-k}$ and $\lambda_{k} = - \lambda_{m-k}$. For $1 \leq n \leq m$, let $k_n = \sum_{j=1}^n d_j$. Then $\tau$ is $\Psf_{k_n}$-proximal for all $1 \leq n \leq m$. Consider the partial flag manifold 
 $$
 \Fc = \left\{ \left(F^{k_n}\right)_{n=1}^{m} : F^{k_1} \subset \cdots \subset F^{k_m} \text{ and } \dim F^{k_n} = k_n \text{ for } n =1,\dots, m\right\}
 $$
 and let 
 $$
 F^+ = ( \ip{e_1,\dots, e_{k_n}})_{n=1}^{m} \in \Fc.
 $$
 Since the boundary map constructed in  Proposition~\ref{prop:representations of rank one groups} is equivariant and strongly dynamics preserving, $\tau(\Gamma_0)$ fixes $F^+$ and if $(g_n)_{n \geq 1}$ is an escaping sequence in $\Gamma_0$, then 
 $$
 \lim_{n \to \infty} \tau(g_n) F = F^+
 $$
 for all $F \in \Fc$ transverse to $F^+$ and the convergence is locally uniform. 
 
 We claim that $\rho(\Gamma_0)$ fixes  $F^+$. Fix $g \in \Gamma_0$ and fix an escaping sequence $(g_n)_{n \geq 1}$ in $\Gamma_0$. Passing to a subsequence we can suppose that 
$$
\rho(g_n)^{-1}\tau(g_n) \to h_1 \quad \text{and} \quad \tau(gg_n)^{-1} \rho(gg_n)  \to h_2.
$$
Fix $F \in \Fc$ transverse to $(h_2 h_1)^{-1} F^+$ and $F^+$. Then 
\begin{align*}
\rho(g) F^+  & = \lim_{n \to \infty} \rho(g) \tau(g_n) F = \lim_{n \to \infty} \tau(g g_n) \tau(gg_n)^{-1} \rho(gg_n) \rho(g_n)^{-1} \tau(g_n) F \\
& = F^+. 
\end{align*}
Since $g \in \Gamma_0$ was arbitrary, $\rho(\Gamma_0)$ fixes  $F^+$.

Finally since $\rho(\Gamma_0)$, $\tau(\Gamma_0)$ both fix $F^+$ and 
\begin{align*}
\left\{ \tau(g)\rho(g)^{-1} : g \in \Gamma_0 \right\}
\end{align*}
is relatively compact in $\SL(d,\Kb)$, for every $1 \leq i \leq j \leq m+1$ there exist compact subsets $K_{i,j}$ of $d_i$-by-$d_j$ matrices such 
\begin{align*}
\left\{ \tau(g)\rho(g)^{-1} : g \in \Gamma_0 \right\} \subset \left\{ \begin{pmatrix} A_{1,1} &   \dots &  A_{1,m+1} \\
 & \ddots & \vdots \\
& & A_{m+1, m+1} \end{pmatrix} : A_{i,j} \in K_{i,j} \right\}. 
\end{align*}
Then
\begin{align*}
\left\{ \tau(a_{-t})\tau(g)\rho(g)^{-1}\tau(a_t) : g \in \Gamma_0, t \geq 0 \right\} 
\end{align*}
is relatively compact in $\SL(d,\Kb)$.
\end{proof} 

Since $X$ is Gromov-hyperbolic, there exists $\delta > 0$ such that every geodesic triangle, including every ideal geodesic triangle, is $\delta$-slim. 

\begin{lemma} $\Gamma_0$ acts cocompactly on $\Uc|_{\partial H}$. \end{lemma} 

\begin{proof} Fix a compact subset $K_0 \subset \Lc \smallsetminus \{\eta^+\}$ such that $\Gamma_0 \cdot K_0 =  \Lc \smallsetminus \{\eta^+\}$.  Then let 
$$
K_1 := \{ v \in \Uc|_{\partial H} : v^- \in K_0 \text{ and } v^+ = \eta^+\}
$$
and 
$$
K_2 := \{ v \in \Uc|_{\partial H} : \dist_X( \pi(v), \pi(w)) \leq 2\delta \text{ for some } w \in K_1\}.
$$
Notice that both $K_1$ and $K_2$ are compact subsets. 

We claim that $\Gamma_0 \cdot K_2 = \Uc|_{\partial H}$. Fix $v \in \Uc|_{\partial H}$. By our choice of $\delta$, the ideal geodesic triangle with vertices $\eta^+, v^+, v^-$ is $\delta$-slim. So there exists $s \in \{-,+\}$  such that $\pi(v)$ is within $\delta$ of the geodesic line joining $v^s$ and $\eta^+$. Then let $w \in \Uc$ be the vector with $\pi(w) \in \partial H$, $w^- = v^s$, and $w^+ = \eta^+$. Fix $T \in \Rb$ such that 
 $$
 \dist_X( \pi( \geodflow^T(w)), \pi( v)) \leq \delta.
 $$ 
 Since $w^+ = \eta^+$ and  $\pi(w), \pi(v) \in \partial H$, then 
$$
 \abs{T} \leq \dist_X( \pi( \geodflow^T(w)), \pi( v)) \leq \delta. 
$$
So $\dist_X(\pi(v), \pi(w)) \leq 2 \delta$. 
By our choice of $K_1$, we have $w \in \Gamma_0 \cdot K_1$ which implies that $v \in \Gamma_0 \cdot K_2$. 
\end{proof}

\begin{lemma}\label{lem:covering} There exists a compact subset $\mathcal{K} \subset \mathsf{G}$ such that 
$$
\Uc|_H \subset \Gamma_0 \cdot \{ a_t\}_{t \geq 0} \cdot  \Kc \cdot v_0. 
$$
\end{lemma}  

\begin{proof} By the previous lemma there exists $R > 0$ such that 
$$
\pi(\Uc|_{\partial H}) \subset \Gamma_0 \cdot \mathcal{B}_X(p_0, R). 
$$
Then let 
  $$
 \Kc := \{ g \in \mathsf{G} : \dist_X( g(p_0), p_0) \leq R+ \delta\}.
 $$

 Fix $v \in \Uc|_H$. By our choice of $\delta$, the ideal geodesic triangle with vertices $\eta^+, v^+, v^-$ is $\delta$-slim. So there exists $s \in \{-,+\}$  such that $\pi(v)$ is within $\delta$ of the geodesic line joining $v^s$ and $\eta^+$. Then let $w \in \Uc$ be the vector with $\pi(w) \in \partial H$, $w^- = v^s$, and $w^+ = \eta^+$. Fix $T \geq 0$ such that 
 $$
 \dist_X( \pi( \geodflow^T(w)), \pi( v)) \leq \delta
 $$ 
and fix $\beta \in \Gamma_0$ such that $\dist_X(\beta(p_0), \pi(w)) \leq R$.
 
 Then 
 \begin{align*}
 \dist_X & \Big(  p_0,  a_{-T} \beta^{-1} \pi(v) \Big)=   \dist_X\Big( \beta\pi(  \geodflow^T(v_0)), \pi(v) \Big) \\
 & \leq \dist_X\Big( \beta\pi(  \geodflow^T(v_0)), \pi(\geodflow^T(w))\Big) + \dist_X\Big( \pi(\geodflow^T(w)),  \pi(v) \Big) \leq \dist_X\Big( \beta(p_0), \pi(w)\Big) + \delta \\ 
 & \leq R + \delta. 
 \end{align*}
So we can pick $\alpha \in \Kc$ such that $\alpha(v_0) = a_{-T} \beta^{-1} v$ or equivalently  
 \begin{equation*}
 v = \beta a_T\alpha(v_0) \in \Gamma_0 \cdot \{ a_t\}_{t \geq 0} \cdot  \Kc \cdot v_0. \qedhere
 \end{equation*}
\end{proof} 

\begin{lemma}\label{lem:almost tau equivariant} There exists $C > 1$ such that: If $v \in \Uc|_H$, then 
$$
\frac{1}{C} \norm{\cdot}^{(0)}_{v} \leq \norm{\cdot}_v \leq C \norm{\cdot}_{v}^{(0)}.
$$
\end{lemma} 

\begin{proof} By Lemma~\ref{lem:covering} there exist $\beta \in \Gamma_0$, $T \geq 0$, and $\alpha \in \Kc$ such that $v= \beta a_T\alpha(v_0)$. Then 
\begin{align*}
\norm{\cdot}_v & = \sqrt{ \sum_{g \in \Gamma_0}  (\chi\circ g)(\pi(v))\left( \norm{\rho(g)( \cdot)}_{gv}^{(0)} \right)^2}\\
& =\sqrt{ \sum_{g \in \Gamma_0}  (\chi\circ g)(\pi(v)) \left(\norm{\tau(\alpha)^{-1} \tau(a_{-T}) \tau(g\beta)^{-1}\rho(g)( \cdot)}_{v_0}^{(0)} \right)^2}. 
\end{align*} 

Using the compactness of $\Kc$ and Lemma~\ref{lem:upgrading relative compactness}
\begin{align*}
& \norm{\tau(\alpha)^{-1} \tau(a_{-T}) \tau(g\beta)^{-1}\rho(g)( \cdot)}_{v_0}^{(0)}  \asymp \norm{\tau(a_{-T}) \tau(g\beta)^{-1}\rho(g)( \cdot)}_{v_0}^{(0)}\\
& \quad = \norm{\tau(a_{-T}) \tau(g\beta)^{-1}\rho(g\beta) \tau(a_T) \tau(a_{-T}) \rho(\beta)^{-1}\tau(\beta) \tau(a_T) \tau(\beta a_T)^{-1} ( \cdot)}_{v_0}^{(0)} \\
& \quad \asymp \norm{ \tau(\beta a_T)^{-1} ( \cdot)}_{v_0}^{(0)} \asymp \norm{ \tau(\beta a_T)^{-1} ( \cdot)}_{\alpha(v_0)}^{(0)} = \norm{  \cdot}_{v}^{(0)}.
\end{align*}
Thus 
\begin{equation*}
\norm{\cdot}_v \asymp\sqrt{ \sum_{g \in \Gamma_0} (\chi\circ g)(\pi(v))  \left(\norm{\cdot}_{v}^{(0)}\right)^2}  = \norm{  \cdot}_{v}^{(0)}. \qedhere
\end{equation*} 
\end{proof} 

We can now establish part (2) of the proposition.

\begin{lemma}\label{lem:uniform local bounds on the cusps} For any $r > 1$ there exists $L_r > 1$ such that: if $v,w \in  \Uc|_H$ and $\dist_X(\pi(v),\pi(w)) \leq r$, then 
 $$
 \frac{1}{L_r} \norm{\cdot}_v \leq \norm{\cdot}_{w} \leq L_r \norm{\cdot}_v.
 $$
\end{lemma} 

\begin{proof} This follows immediately from Proposition~\ref{prop:representations of rank one groups} and Lemma~\ref{lem:almost tau equivariant}, since $\norm{\cdot}^{(0)}_v$ is a $\tau$-equivariant family of norms.
\end{proof} 

We next establish part (3) of the proposition. By Proposition~\ref{prop:representations of rank one groups} there exist $C_\lambda, \lambda > 0$ such that:
 \begin{align}
 \label{eqn:homogeneous bundle contraction} 
 \frac{\norm{Y}_{\geodflow^t(v)}^{(0)}}{\norm{Z}_{\geodflow^t(v)}^{(0)}} \leq C_\lambda e^{-\lambda t} \frac{\norm{Y}_v^{(0)}}{\norm{Z}_v^{(0)}}
 \end{align}
 for all $t \geq 0$, $v \in T^1 X$, $Y \in \zeta_\tau^k(v^+)$, and non-zero $Z \in \zeta_\tau^{d-k}(v^-)$. 

\begin{lemma} There exist $C_1 > 0$ and a horoball $H^\prime \subset H$ such that: if $t \geq 0$ and $v, \geodflow^t(v) \in \Uc|_{H^\prime}$, then 
$$
\frac{\norm{Y}_{\geodflow^t(v)}}{\norm{Z}_{\geodflow^t(v)}} \leq C_1e^{-\frac{\lambda}{2}t}\frac{\norm{Y}_v}{\norm{Z}_v} 
$$
for all $Y \in \xi^k(v^+)$ and non-zero $Z \in \xi^{d-k}(v^-)$. 
\end{lemma} 

\begin{proof} The following argument is similar to the proof of ~\cite[Prop.\ 6.4]{CZZ2021}. Fix $T > 0$ such that 
 \begin{equation}
 \label{eqn:defn of T is proof almost homog contraction}
 C^4 C_\lambda  < e^{\frac{\lambda}{2} T}
 \end{equation}
 where $C$ is the constant from Lemma~\ref{lem:almost tau equivariant}.
 
We first claim that there exists a horoball $H^\prime \subset H$ such that: if $t \in [T,2T]$ and $v, \geodflow^t(v) \in \Uc|_{H^\prime}$, then 
$$
\frac{\norm{Y}_{\geodflow^t(v)}}{\norm{Z}_{\geodflow^t(v)}} \leq e^{-\frac{\lambda}{2}t}\frac{\norm{Y}_v}{\norm{Z}_v} 
$$
for all $Y \in \xi^k(v^+)$ and non-zero $Z \in \xi^{d-k}(v^-)$. 

Suppose not. Then there exist sequences $(v_n)_{n \geq 1}$ in $\Uc$, $( t_n)_{n \geq 1}$ in $[T,2T]$, and $(Y_n)_{n \geq 1}, (Z_n)_{n \geq 1}$ in $\Kb^d$ such that 
$\dist_X(\pi(v_n), \partial H) \to \infty$, $Y_n \in \xi^k(v_n^+)$, $Z_n \in \xi^{d-k}(v_n^-)\smallsetminus\{0\}$, and
$$
\frac{\norm{Y_n}_{\geodflow^{t_n}(v_n)}}{\norm{Z_n}_{\geodflow^{t_n}(v_n)}} > e^{-\frac{\lambda}{2}t_n}\frac{\norm{Y_n}_{v_n}}{\norm{Z_n}_{v_n}}.
$$
By scaling we may assume that 
\begin{equation}\label{eqn: we can assume Yn and Zn are 1}
\norm{Y_n}_{v_n} =\norm{Z_n}_{v_n}=1.
\end{equation} 

Using Lemma~\ref{lem:covering} and possibly replacing each $v_n$ with an $\Gamma_0$-translate, we can find a sequence $m_n \to \infty$ and a relatively compact sequence $( \alpha_n)_{n \geq 1}$ in $\mathsf{G}$ such that 
$$
v_n = a_1^{m_n} \alpha_n(v_0).
$$
 Let $Y^\prime_n := \tau(a_1^{m_n} \alpha_n)^{-1} Y_n$ and $Z^\prime_n := \tau(a_1^{m_n} \alpha_n)^{-1} Z_n$. Then by Lemma~\ref{lem:almost tau equivariant} and Equation~\eqref{eqn: we can assume Yn and Zn are 1}
 $$
 \norm{Y^\prime_n}_{v_0}^{(0)} = \norm{ Y_n}_{ v_n}^{(0)} \in [C^{-1}, C].
 $$
Likewise $\norm{Z^\prime_n}_{v_0}^{(0)} \in [C^{-1}, C]$.
 
Passing to a subsequence we can suppose that $t_n \to t \in [T,2T]$, $\alpha_n \to \alpha \in \mathsf{G}$, $Y^\prime_n \to Y^\prime$, and $Z^\prime_n \to Z^\prime$. Proposition~\ref{prop:local homogeneity} implies that
 \begin{align*}
 Y^\prime &  = \tau(\alpha)^{-1}  \lim_{n \to \infty} \tau(a_1)^{-m_n} Y_n  \in \tau(\alpha)^{-1}  \lim_{n \to \infty}  \tau(a_1)^{-m_n} \xi^k( v_n^+) \\
 & = \tau(\alpha)^{-1}  \lim_{n \to \infty}  \tau(a_1)^{-m_n} \circ \xi^k \circ a_1^{m_n} ( \alpha_n(v_0^+)) = \tau(\alpha)^{-1} \circ \zeta^k_{\tau} \circ \alpha(v_0^+) \\
 & = \zeta_{\tau}^k(v_0^+).
 \end{align*}
Likewise, $Z^\prime \in  \zeta_{\tau}^{d-k}(v_0^-)$. 

Then Lemma~\ref{lem:almost tau equivariant} and Equation~\eqref{eqn:homogeneous bundle contraction} imply that
 \begin{align*}
e^{-\frac{\lambda}{2} t} & \leq \liminf_{n \to \infty}  \frac{\norm{Y_n}_{\geodflow^{t_n}(v_n)}}{\norm{Z_n}_{\geodflow^{t_n}(v_n)}} \leq C^2 \liminf_{n \to \infty}  \frac{\norm{Y_n}_{\geodflow^{t_n}(v_n)}^{(0)}}{\norm{Z_n}_{\geodflow^{t_n}(v_n)}^{(0)}} = C^2 \liminf_{n \to \infty}  \frac{\norm{Y^\prime_n}_{\geodflow^{t_n}(v_0)}^{(0)}}{\norm{Z^\prime_n}_{\geodflow^{t_n}(v_0)}^{(0)}} \\
& = C^2\frac{\norm{Y^\prime}_{\geodflow^{t}(v_0)}^{(0)}}{\norm{Z^\prime}_{\geodflow^{t}(v_0)}^{(0)}}  \leq C^2 C_\lambda e^{-\lambda t} \frac{\norm{Y^\prime}_{v_0}^{(0)}}{\norm{Z^\prime}_{v_0}^{(0)}} \leq C^4 C_\lambda e^{-\lambda t}.
 \end{align*}
Then
 $
e^{\frac{\lambda}{2} T} \leq e^{\frac{\lambda}{2} t} \leq C^4 C_\lambda
 $
 and we have a contradiction with Equation~\eqref{eqn:defn of T is proof almost homog contraction}. So the claim is true. 
 
 Now suppose that $t \geq 0$,  $v, \geodflow^t(v) \in \Uc|_{H^\prime}$, $Y \in \xi^k(v^+)$, and $Z \in \xi^{d-k}(v^-)\smallsetminus \{0\}$. If $t \leq T$, then 
 $$
\frac{\norm{Y}_{\geodflow^t(v)}}{\norm{Z}_{\geodflow^t(v)}} \leq L_T e^{\frac{\lambda}{2}T}e^{-\frac{\lambda}{2}t}\frac{\norm{Y}_v}{\norm{Z}_v} 
$$
 by Lemma~\ref{lem:uniform local bounds on the cusps}. If $t \geq T$, then we can break $[0,t]$ into subintervals each with length between $T$ and $2T$, then apply the claim on each subinterval to obtain
 $$
\frac{\norm{Y}_{\geodflow^t(v)}}{\norm{Z}_{\geodflow^t(v)}} \leq e^{-\frac{\lambda}{2}t}\frac{\norm{Y}_v}{\norm{Z}_v}.
$$
So $C_1 := L_T e^{\frac{\lambda}{2}T}$ suffices. 
\end{proof}

\section{Representations with almost homogeneous cusps} 

Let $\mathsf{G}$, $\mathsf{K}$, and $X=\mathsf{G}/\mathsf{K}$ be as in Section~\ref{sec:reminders negatively curved symspaces}. In this section we prove Theorem~\ref{thm:repn with almost homog cusps in intro}, restated in the following form. 

\begin{theorem}\label{thm:repn with almost homog cusps} Suppose that 
\begin{itemize}
\item $\Gamma \leq \mathsf{G}$ is geometrically finite and $\peripherals$ is a set of representatives of the conjugacy classes of maximal parabolic subgroups of $\Gamma$,
\item $\rho \colon \Gamma \to \SL(d,\Kb)$ is $\Psf_k$-Anosov relative to $\peripherals$, and 
\item for each $P \in \peripherals$ there exists a representation $\tau_P\colon \mathsf{G} \to \SL(d,\Kb)$ such that 
\begin{align*}
\left\{ \tau_P(g)\rho(g)^{-1} : g \in P\right\}
\end{align*}
is relatively compact in $\SL(d,\Kb)$.
\end{itemize}
Then $\rho$ is uniformly $\Psf_k$-Anosov relative to $\Cc_{X}(\Gamma)$.
\end{theorem} 

The rest of the section is devoted to the proof of theorem, so fix $\Gamma$, $\peripherals$, $\rho$, and representations $\{ \tau_P : P \in \peripherals\}$ as in the statement. Let $E_\rho: = \Uc(\Gamma) \times \Kb^d$ and $\wh{E}_\rho: = \Gamma \backslash (\Uc(\Gamma) \times \Kb^d)$. 

For each $P \in \peripherals$, fix an open horoball $H_P$ centered at the fixed point of $P$ such that: if $\gamma \in \Gamma$, then $\gamma \overline{H}_P \cap \overline{H}_P \neq \varnothing$ if and only if $\gamma \in P$. This is possible by the ``F1'' definition of geometrically finite subgroups in~\cite{Bowditch_GF_Riem}. Let 
$$
\wh{\Uc}_P  :=\Gamma \backslash \left\{ v \in \Uc(\Gamma) : \pi(v) \in H_P \right\},
$$
$\wh{\Uc}_{thin} := \bigcup_{P \in \peripherals} \wh{\Uc}_P$, and $\wh{\Uc}_{thick} := \wh{\Uc}(\Gamma) \smallsetminus \wh{\Uc}_{thin}$. Then  $\wh{\Uc}_{thick}$ is compact by the ``F1'' definition of geometrically finite subgroups in~\cite{Bowditch_GF_Riem}.

\begin{lemma}\label{lem:building the norms in almost homog examples} After possibly replacing each $H_P$ with a smaller horoball, there exist $C_0,c_0>0$ and a metric $\norm{\cdot}_{v \in \wh{\Uc}(\Gamma)}$ on the vector bundle $\wh{E}_\rho \to \wh{\Uc}(\Gamma)$ such that: 
\begin{enumerate}
\item $\norm{\cdot}_{v \in \wh{\Uc}(\Gamma)}$ is locally uniform, 
\item if $t \geq 0$, $v \in \wh{\Uc}(\Gamma)$, and $\geodflow^s(v) \in \wh{\Uc}_{thin}$ for all $s \in [0,t]$, then 
\begin{align*}
\frac{\norm{\flatflow^t(Y)}_{\geodflow^t(v)}}{\norm{\flatflow^t(Z)}_{\geodflow^t(v)}} \leq C_0 e^{-c_0t} \frac{\norm{Y}_v}{\norm{Z}_v}
\end{align*}
for all $Y \in \wh{\Theta}^k(v)$ and non-zero $Z \in \wh{\Xi}^{d-k}(v)$. 
\end{enumerate}
\end{lemma} 

\begin{proof} Fix a partition of unity $\{ \chi_0\} \cup \{ \chi_P : P \in \peripherals\}$ of $\wh{\Uc}(\Gamma)$ such that ${\rm supp}(\chi_0)$ is compact and ${\rm supp}(\chi_P) \subset \wh{\Uc}_P$ for all $P \in \peripherals$. 

Let $\norm{\cdot}^{(0)}_{v \in \wh{\Uc}(\Gamma)}$ be any metric on $\wh{E}_\rho \to \wh{\Uc}(\Gamma)$. For each $P \in \peripherals$ let $\norm{\cdot}_{v \in T^1 X}^P$ be a family of $\rho|_P$-equivariant norms satisfying Proposition~\ref{prop:almost homogeneous norms}. Then $\norm{\cdot}_{v \in T^1 X}^P$ descends to a metric on the fibers of $\wh{E}_\rho$ above $\wh{\Uc}_P$ which we denote by $\norm{\cdot}_{v \in \wh{\Uc}_P}^P$. Then 
$$
\norm{\cdot}_v = \sqrt{\chi_0(v)  \left( \norm{\cdot}^{(0)}_v \right)^2+ \sum_{P \in \peripherals} \chi_P(v)  \left(\norm{\cdot}^P_v\right)^2 }
$$
defines a metric with the desired properties. 
\end{proof} 

\begin{lemma}\label{lem:contraction between the thick parts} There exists $T_0 > 0$ such that: if $t \geq T_0$ and $v, \geodflow^t(v) \in \wh{\Uc}_{thick}$, then 
\begin{align*}
\frac{\norm{\flatflow^t(Y)}_{\geodflow^t(v)}}{\norm{\flatflow^t(Z)}_{\geodflow^t(v)}} \leq \frac{1}{2C_0^2} \frac{\norm{Y}_v}{\norm{Z}_v}
\end{align*}
for all $Y \in \wh{\Theta}^k(v)$ and non-zero $Z \in \wh{\Xi}^{d-k}(v)$. 
\end{lemma}

\begin{proof} Lift $\norm{\cdot}_{v \in \wh{\Uc}(\Gamma)}$ to a $\rho$-equivariant family of norms $\norm{\cdot}_{v \in \Uc(\Gamma)}$. Let 
$$
\Uc_{thick} := \Uc(\Gamma) \cap \pi^{-1}\left( \wh{\Uc}_{thick} \right). 
$$ 
Then fix a compact set $K \subset \Uc_{thick}$ such that $\Gamma \cdot K = \Uc_{thick}$. Finally, fix some $p_0 \in \pi(K)$ and let $R := \diam_X(\pi(K))$.

Arguing as in the proof of ~\cite[Lem.\ 9.4]{ZZ2022a}, there exists $C > 1$ such that: if $v \in K$, $t \geq 0$, and $\geodflow^t(v) \in g(K)$ for some $g \in \Gamma$, then 
\begin{align*}
\frac{\norm{Y}_{\geodflow^t(v)}}{\norm{Z}_{\geodflow^t(v)}} \leq C \, \frac{\mu_{k+1}}{\mu_k}(\rho(g)) \, \frac{\norm{Y}_v}{\norm{Z}_v}
\end{align*}
for all $Y \in \xi^k(v^+)$ and non-zero $Z \in \xi^{d-k}(v^-)$. Notice that in this case
$$
\dist_X( p_0, g(p_0)) \geq t - 2R. 
$$

Also, by the strongly dynamics preserving property and Observation \ref{obs:strongly_dynamics_pres_div_cartan}, there exists $T_0^\prime > 0$ such that: if $g \in \Gamma$ and 
$\dist_X( p_0, g(p_0)) \geq T_0^\prime$, then 
$$
\frac{\mu_{k+1}}{\mu_k}(\rho(g))  \leq \frac{1}{2CC_0^2}.
$$
So $T_0 : = T_0^\prime + 2R$ suffices. 
\end{proof} 

\begin{lemma} \label{lem:contraction everywhere}
There exists $T > 1$ such that: if $t \geq T$ and $v \in \wh{\Uc}(\Gamma)$, then 
\begin{align*}
\frac{\norm{\flatflow^t(Y)}_{\geodflow^t(v)}}{\norm{\flatflow^t(Z)}_{\geodflow^t(v)}} \leq \frac{1}{2} \frac{\norm{Y}_v}{\norm{Z}_v}
\end{align*}
for all $Y \in \wh{\Theta}^k(v)$ and non-zero $Z \in \wh{\Xi}^{d-k}(v)$. 
\end{lemma}

\begin{proof} The following argument is similar to an argument in~\cite[pp.\ 33-35]{CZZ2021}. From Lemma~\ref{lem:building the norms in almost homog examples}(1), there exists $C_2 > 1$ such that: if $v \in \wh{\Uc}(\Gamma)$ and $t \in [0,T_0]$, then 
\begin{align}
\label{eqn:dumb bound}
\frac{\norm{\flatflow^t(Y)}_{\geodflow^t(v)}}{\norm{\flatflow^t(Z)}_{\geodflow^t(v)}} \leq C_2 \frac{\norm{Y}_v}{\norm{Z}_v}
\end{align}
for all $Y \in \wh{\Theta}^k(v)$ and non-zero $Z \in \wh{\Xi}^{d-k}(v)$.

Fix $T > 1$ so that 
$$
C_0 e^{-c_0 T} \leq \frac{1}{2} \quad \text{and} \quad C_0^2 C_2 e^{-c_0(T-T_0)} \leq \frac{1}{2}. 
$$

Suppose $t \geq T$ and $v \in \wh{\Uc}(\Gamma)$. If $\geodflow^s(v) \in \wh{\Uc}_{thin}$ for all $s \in [0,t]$, then Lemma~\ref{lem:building the norms in almost homog examples}(2) implies that
\begin{align*}
\frac{\norm{\flatflow^t(Y)}_{\geodflow^t(v)}}{\norm{\flatflow^t(Z)}_{\geodflow^t(v)}} \leq C_0e^{-c_0 t} \frac{\norm{Y}_v}{\norm{Z}_v} \leq  \frac{1}{2} \frac{\norm{Y}_v}{\norm{Z}_v}
\end{align*}
for all $Y \in \wh{\Theta}^k(v)$ and non-zero $Z \in \wh{\Xi}^{d-k}(v)$. Otherwise, the set $\mathcal{R} := \{ s \in [0,t] : \geodflow^s(v) \in \wh{\Uc}_{thick} \}$ is non-empty. Let $s_1 := \min \mathcal{R}$ and $s_2 := \max \mathcal{R}$. If $s_2 - s_1 \geq T_0$, then applying Lemma~\ref{lem:building the norms in almost homog examples}(2) to the intervals $[0,s_1]$, $[s_2, t]$ and Lemma~\ref{lem:contraction between the thick parts} to the interval $[s_1, s_2]$ yields 
\begin{align*}
\frac{\norm{\flatflow^t(Y)}_{\geodflow^t(v)}}{\norm{\flatflow^t(Z)}_{\geodflow^t(v)}} \leq C_0e^{-c_0 (t-s_2)}\frac{1}{2C_0^2} C_0e^{-c_0s_1}  \frac{\norm{Y}_v}{\norm{Z}_v} \leq  \frac{1}{2} \frac{\norm{Y}_v}{\norm{Z}_v}
\end{align*}
for all $Y \in \wh{\Theta}^k(v)$ and non-zero $Z \in \wh{\Xi}^{d-k}(v)$. Otherwise, if $s_2 - s_1 \leq T_0$, then applying Lemma~\ref{lem:building the norms in almost homog examples}(2) to the intervals $[0,s_1]$, $[s_2, t]$ and Equation~\eqref{eqn:dumb bound} to the interval $[s_1,s_2]$ yields 
\begin{align*}
\frac{\norm{\flatflow^t(Y)}_{\geodflow^t(v)}}{\norm{\flatflow^t(Z)}_{\geodflow^t(v)}} \leq C_0e^{-c_0 (t-s_2)}C_2 C_0e^{-c_0s_1}  \frac{\norm{Y}_v}{\norm{Z}_v}  \leq C_0^2 C_2 e^{-c_0(T-T_0)} \frac{\norm{Y}_v}{\norm{Z}_v} \leq \frac{1}{2} \frac{\norm{Y}_v}{\norm{Z}_v}
\end{align*}
for all $Y \in \wh{\Theta}^k(v)$ and non-zero $Z \in \wh{\Xi}^{d-k}(v)$. 
\end{proof}

\begin{proof}[Proof of Theorem~\ref{thm:repn with almost homog cusps}] By Lemma~\ref{lem:building the norms in almost homog examples}(1), we have locally uniform norms, and it remains only to verify the dominated splitting condition in Proposition~\ref{prop: Hom bundles contraction/expansions}. Already, from Lemma~\ref{lem:building the norms in almost homog examples}(1) there exists $C_3 > 1$ such that: if $v \in \wh{\Uc}(\Gamma)$ and $t \in [0,T]$, then 
\begin{align*}
\frac{\norm{\flatflow^t(Y)}_{\geodflow^t(v)}}{\norm{\flatflow^t(Z)}_{\geodflow^t(v)}} \leq C_3 \frac{\norm{Y}_v}{\norm{Z}_v}
\end{align*}
for all $Y \in \wh{\Theta}^k(v)$ and non-zero $Z \in \wh{\Xi}^{d-k}(v)$. Lemma \ref{lem:contraction everywhere} then implies that: 
\begin{align*}
\frac{\norm{\flatflow^t(Y)}_{\geodflow^t(v)}}{\norm{\flatflow^t(Z)}_{\geodflow^t(v)}} \leq 2C_3 e^{-\frac{\ln(2)}{T} t} \frac{\norm{Y}_v}{\norm{Z}_v}
\end{align*}
for all $v \in \wh{\Uc}(\Gamma)$, $t \geq 0$, $Y \in \wh{\Theta}^k(v)$, and non-zero $Z \in \wh{\Xi}^{d-k}(v)$. 
\end{proof}

\section{Not uniform relative to the Groves--Manning cusp space} \label{sec:not uniform rel GM cusp space}

In this section we construct the representation described in Example~\ref{ex:non uniform to GM cusp space} above. In particular, we construct a relatively $\Psf_1$-Anosov representation which is uniform relative to some weak cusp space, but  is not uniformly $\Psf_1$-Anosov relative to any Groves--Manning cusp space. 

We consider the Siegel model of complex hyperbolic 2-space
$$
\Hb_{\Cb}^2 = \left\{ [z_1 : z_2:1] : {\rm Im}(z_1) > \abs{z_2}^2\right\} \subset \proj(\Cb^3). 
$$
Then $\Isom_0(\Hb_{\Cb}^2)$ coincides with the subgroup of $\PSL(3,\Cb)$ which preserves $\Hb_{\Cb}^2$. Let $\mathsf{G} \to \Isom_0(\Hb^2_{\Cb})$ denote the preimage in $\SL(3,\Cb)$. 

For $m,n \in \Zb$ define
$$
u(m,n) := \begin{pmatrix} 1 & m & \frac{1}{2}m^2 + in \\ 0 & 1 & m \\ 0 & 0 & 1 \end{pmatrix}  \in \SL(3,\Cb).
$$
Then let $P:= \{ u(m,n)  : m,n \in \Zb\} \leq \mathsf{G}$. Notice that 
$$
(m,n) \in \Zb^2 \to u(m,n) \in P
$$
is a group isomorphism. Using ping-pong we can find a hyperbolic element $h \in \mathsf{G}$ such that $\Gamma: = \ip{h} * P$ is a  geometrically finite subgroup of $\mathsf{G}$ isomorphic to $\Zb * \Zb^2$. 

Let $\Lambda(\Gamma) \subset \partial_\infty \Hb^2_{\Cb}$ denote the limit set of $\Gamma$ and let $\Cc(\Gamma)$ denote the convex hull of $\Lambda(\Gamma)$ in $\Hb^2_{\Cb}$. Then by Proposition~\ref{prop:representations of rank one groups}, the inclusion representation $\rho \colon \Gamma \into \SL(3,\Cb)$ is uniformly $\Psf_1$-Anosov relative to $\Cc(\Gamma)$. 

Let $\peripherals := \{ P\}$ and $S :=\{ h, h^{-1}, u(1,0), u(-1,0), u(0,1), u(0,-1) \}$. Then consider the associated Groves--Manning cusp space $X := \Cc_{GM}(\Gamma, \peripherals, S)$.

The main result of this section is the following. 

\begin{proposition} There does not exist a $\rho$-equivariant quasi-isometric embedding of $X$ into $M:=\SL(3,\Cb) / \SU(3,\Cb)$. \end{proposition}

When combined with results in~\cite{ZZ2022a} this yields the following corollary. 
 
\begin{corollary} $\rho$ is not uniformly $\Psf_1$-Anosov relative to any Groves--Manning cusp space. \end{corollary} 
 
\begin{proof}[Proof of Corollary] Suppose for a contradiction that $\rho$ is uniformly $\Psf_1$-Anosov relative to some Groves--Manning cusp space $Y$. By~\cite[Th.\ 1.12]{ZZ2022a} there exists a $\rho$-equivariant quasi-isometric embedding of $F\colon Y \to M$. However, the identity map on vertices extends to a $\Gamma$-equivariant quasi-isometry $G\colon X \to Y$, see \cite[Th.\ 1.1]{HealyHruska}, and so we obtain a $\rho$-equivariant quasi-isometric embedding $F \circ G \colon X \to M$. Hence we have a contradiction. 
 \end{proof}

The rest of the section is devoted to the proof of the proposition. Suppose for a contradiction that  there exists a $\rho$-equivariant quasi-isometric embedding $F \colon X \to M$. Let $\dist_M$ denote the standard symmetric distance on $M$ defined in Equation~\eqref{eqn:symmetric distance in prelims} and let $\mathsf{K}:=\SU(3,\Cb)$. Then
$$
\dist_M(g \mathsf{K}, \mathsf{K}) \asymp \log \frac{\mu_1}{\mu_3}(g)
$$
for all $g \in \SL(3,\Cb)$. 

Using the Iwasawa decomposition, for every $n \in \Nb$ we can write 
$$
F\left( (\id_P, n) \right) = \mathscr{w}_n \mathscr{a}_n \mathsf{K}
$$
where $\mathscr{a}_n$ is a diagonal matrix with positive diagonal entries and $\mathscr{w}_n$ is upper triangular matrix with ones on the diagonal. Then for all $g \in P$ and $n \in \Nb$, we have
\begin{align*}
\dist_M &\Big( F( (g, n) ), F((\id_P, n)) \Big) = \dist_M \Big(  \rho(g) \mathscr{w}_n \mathscr{a}_n\mathsf{K},  \mathscr{w}_n \mathscr{a}_n\mathsf{K} \Big) \\
&= \dist_M \Big( \mathscr{a}_n^{-1} \mathscr{w}_n^{-1} \rho(g) \mathscr{w}_n \mathscr{a}_n\mathsf{K}, \mathsf{K} \Big)  \asymp \log  \frac{\mu_1}{\mu_3}\left( \mathscr{a}_n^{-1} \mathscr{w}_n^{-1} \rho(g) \mathscr{w}_n \mathscr{a}_n \right).
\end{align*}
Further, since $F \colon X \to M$ is a quasi-isometric embedding, there exist $\alpha > 1, \beta > 0$ such that: if $g \in P$ and $n \in \Nb$, then 
\begin{equation}\label{eqn:distance in X estimate}
\begin{aligned}
\frac{1}{\alpha} \dist_X\Big( (g, n),& (\id_{P}, n)\Big) - \beta  \leq \log  \frac{\mu_1}{\mu_3}\left( \mathscr{a}_n^{-1} \mathscr{w}_n^{-1} \rho(g) \mathscr{w}_n \mathscr{a}_n \right)   \\
& \leq \alpha \dist_X\Big( (g, n), (\id_{P}, n)\Big) + \beta. 
\end{aligned}
\end{equation}

Suppose
$$
\mathscr{a}_n = \begin{pmatrix} \lambda_{n,1} & 0 & 0 \\ 0 & \lambda_{n,2} & 0 \\ 0 & 0 & \lambda_{n,3} \end{pmatrix} \quad \text{and} \quad \mathscr{w}_n = \begin{pmatrix} 1 & s_n & r_n \\ 0 & 1 & t_n \\ 0 & 0 & 1 \end{pmatrix}.
$$
We will obtain a contradiction by estimating $\lambda_{n,1}^{-1} \lambda_{n,3}$ in two ways. 

We start with the following distance estimate in the Groves--Manning cusp space. 

\begin{lemma}\label{lem:GM dist estimate} There exists $n_0 > 0$ such that: if $k \geq n \geq n_0$, then 
$$
2k-2n - 2 \leq \dist_X\Big( (u(0,2^k), n), (\id_{P}, n)\Big).
$$
\end{lemma}

\begin{proof} For $L \geq 1$, let $\Hc(L)\subset X$ denote the induced subgraph of $X$ with vertex set
$$
\{ (g,n) : g \in P, n \geq L\}. 
$$
By~\cite[Lem.\ 3.26]{GrovesManning} there exists $\delta \geq 1$ such that $\Hc(\delta)$ is geodesically convex in $X$. 

Fix $k \geq n \geq \delta$. By~\cite[Lem.\ 3.10]{GrovesManning}, there exists a geodesic in $\Hc(\delta)$ joining  $(u(0,2^k), n)$ to $(\id_{P}, n)$ which consists of $m$ vertical edges, followed by no more than three horizontal edges, followed by $m$ vertical edges. Then 
$$
2^k=\abs{u(0,2^k)}_{S \cap P} \leq 3 \cdot 2^{n+m-1} \leq 2^{n+m+1}
$$ 
and since $\Hc(\delta)$ is geodesically convex
\begin{align*}
 \dist_X\Big( (u(0,2^k), n), (\id_{P}, n)\Big) &= \dist_{\Hc(\delta)}\Big( (u(0,2^k), n), (\id_{P}, n)\Big) \geq 2m  \\
& \geq 2k-2n-2.
\end{align*}
So $n_0:=\delta$ suffices. 
\end{proof} 

In the arguments that follow, given a matrix $g \in \GL(d,\Cb)$ let
$$
\norm{g}_\infty := \max_{1 \leq i,j \leq d} \abs{g_{i,j}}.
$$
Then 
\begin{equation}
\label{eqn:matrix norms Linfty versus operator norm}
\norm{g}_\infty \leq \mu_1(g) \leq d \norm{g}_\infty
\end{equation}
for all $g \in \GL(d,\Cb)$. 

\begin{lemma}\label{lem:product of diagonal elements, lower bound} $\lambda_{n,1}^{-1} \lambda_{n,3} \gtrsim 2^{-n}$. \end{lemma} 

\begin{proof} For every $n \geq n_0$, let 
$$
k_n: = \left\lceil \frac12 \alpha (\beta+6) + n + 1 \right\rceil
$$
and let $g_n:= u(0, 2^{k_n})$. Then
$$
\mathscr{a}_n^{-1} \mathscr{w}_n^{-1} \rho(g_n)^{\pm 1} \mathscr{w}_n \mathscr{a}_n = \begin{pmatrix} 1 & 0 &\pm i 2^{k_n} \lambda_{n,1}^{-1} \lambda_{n,3} \\ 0 & 1 & 0 \\ 0 & 0 & 1 \end{pmatrix}. 
$$
Hence, by Equation~\eqref{eqn:matrix norms Linfty versus operator norm},
\begin{align*}
\log \frac{\mu_1}{\mu_3}& \left( \mathscr{a}_n^{-1} \mathscr{w}_n^{-1} \rho(g_n) \mathscr{w}_n \mathscr{a}_n\right) = \log \left( \mu_1\left( \mathscr{a}_n^{-1} \mathscr{w}_n^{-1} \rho(g_n) \mathscr{w}_n \mathscr{a}_n\right) \mu_1\left( \mathscr{a}_n^{-1} \mathscr{w}_n^{-1} \rho(g_n)^{-1} \mathscr{w}_n \mathscr{a}_n\right) \right) \\
& \leq \max\left\{ 0,6\log \left( 2^{k_n} \lambda_{n,1}^{-1} \lambda_{n,3} \right) \right\}.
\end{align*}
So by Lemma~\ref{lem:GM dist estimate} and Equation~\eqref{eqn:distance in X estimate},
\begin{align*}
6 \leq \frac{1}{\alpha}(2k_n - 2n -2) - \beta \leq \frac{1}{\alpha}\dist_X\Big( (g_n, n), (\id_{P}, n)\Big) - \beta \leq \max\{ 0, 6\log \left( 2^{k_n} \lambda_{n,1}^{-1} \lambda_{n,3} \right)\}.
\end{align*}
Then
$$
1\leq \log \left( 2^{k_n} \lambda_{n,1}^{-1} \lambda_{n,3} \right) \leq \log \left( 2^{\frac12\alpha (\beta+6) + n + 1} \lambda_{n,1}^{-1} \lambda_{n,3} \right),
$$
or equivalently 
\begin{equation*}
\frac{e}{2^{\frac12 \alpha(\beta+6) + 1}}2^{-n} \leq \lambda_{n,1}^{-1} \lambda_{n,3}. \qedhere
\end{equation*}

\end{proof} 

\begin{lemma}\label{lem:product of diagonal elements, upper bound} $\lambda_{n,1}^{-1} \lambda_{n,3} \lesssim 4^{-n}$. 
\end{lemma} 

\begin{proof} Let $g_n := u(2^{n-1},0)$. Then 
$$
 \dist_X\Big( (g_n, n), (\id_{P}, n)\Big) = 1. 
$$
Further, 
\begin{align*}
\mathscr{a}_n^{-1}\mathscr{w}_n^{-1} \rho(g_n) \mathscr{w}_n\mathscr{a}_n & =\begin{pmatrix} 1 & \lambda_{n,1}^{-1} \lambda_{n,2}2^{n-1} & * \\ 0 & 1 & \lambda_{n,2}^{-1} \lambda_{n,3}2^{n-1} \\ 0 & 0 & 1 \end{pmatrix}. 
\end{align*}
So by Equations~\eqref{eqn:matrix norms Linfty versus operator norm} and~\eqref{eqn:distance in X estimate},
\begin{align*}
& \max\left\{ \log\left( \lambda_{n,1}^{-1} \lambda_{n,2} 2^{n-1}\right) ,\log \left(\lambda_{n,2}^{-1} \lambda_{n,3}2^{n-1}\right) \right\} \leq \log \mu_1(\mathscr{a}_n^{-1}\mathscr{w}_n^{-1} \rho(g_n) \mathscr{w}_n\mathscr{a}_n) \\
& \leq \alpha \dist_X\Big( (g_n, n), (\id_{P}, n)\Big) +\beta = \alpha + \beta
\end{align*}
which implies that 
\begin{equation*}
\lambda_{n,1}^{-1} \lambda_{n,3}=\lambda_{n,1}^{-1} \lambda_{n,2}\lambda_{n,2}^{-1} \lambda_{n,3} \lesssim 4^{-n}. \qedhere
\end{equation*}
\end{proof} 

Then by Lemmas~\ref{lem:product of diagonal elements, lower bound} and ~\ref{lem:product of diagonal elements, upper bound} we obtain the estimate $2^{-n} \lesssim 4^{-n}$
which is impossible. Hence there does not exist a $\rho$-equivariant quasi-isometric embedding of $X$ into $M$.

\part{Geometrically finite groups in convex real projective geometry} 

\section{Convex real projective geometry} \label{sec:convex projective geometry}

In this expository section we recall the definitions and results in convex real projective geometry that we will need in Sections~\ref{sec:rel Anosov repn preserving a properly convex domain}, \ref{sec:stability of the lifting property}, and~\ref{sec:when a representation preserves a properly convex domain}. We also briefly discuss relatively Anosov representations into the projective linear group.

\subsection{Convexity and the Hilbert metric}  \label{subsec:hilbert metric}
A subset of $\proj(\Rb^d)$ is called \emph{convex} if it is a convex subset of some affine chart of $\proj(\Rb^d)$ and called \emph{properly convex} if it is a bounded convex subset of some affine chart $\proj(\Rb^d)$. A \emph{properly convex domain} is an open properly convex subset of $\proj(\Rb^d)$. 

A subset $H \subset \proj(\Rb^d)$ is called a \emph{projective hyperplane} if it is the image of some codimension-one linear subspace $W \subset \Rb^d$ under the map $\Rb^d \smallsetminus\{0\} \to \proj(\Rb^d)$. Given a properly convex domain $\Omega \subset \proj(\Rb^d)$ and $x \in \partial \Omega$ there always exists at least one projective hyperplane $H \subset \proj(\Rb^d)$ with $x \in H$ and $H \cap \Omega = \varnothing$. In this case, $H$ is called a \emph{supporting hyperplane of $\partial \Omega$ at $x$}. 
When a boundary point $x \in \partial \Omega$ has a unique supporting hyperplane we say that $x$ is a \emph{$\Cc^1$-smooth point} of $\partial \Omega$ and let $T_x \partial \Omega$ denote this unique supporting hyperplane. 

Given a properly convex domain $\Omega \subset \proj(\Rb^d)$ and $p,q \in \overline{\Omega}$ we will let $[p,q]_\Omega$ denote the closed projective line segment in $\overline{\Omega}$ which contains $p$ and $q$. Then define $[p,q)_\Omega := [p,q]_\Omega \smallsetminus \{q\}$, $(p,q]_\Omega := [p,q]_\Omega \smallsetminus \{p\}$, and $(p,q)_\Omega := [p,q]_\Omega \smallsetminus \{p,q\}$.

The \emph{automorphism group} of a subset $S \subset \proj(\Rb^d)$ is the group 
$$
\Aut(S) := \{ g \in \PGL(d,\Rb) : g\cdot S = S\}.
$$
Given a properly convex domain $\Omega \subset \proj(\Rb^d)$ and a subgroup $\Gamma \leq \Aut(\Omega)$, the \emph{limit set of $\Gamma$} is 
$$
\Lambda_\Omega(\Gamma) := \partial \Omega \cap \bigcup_{p \in \Omega} \overline{\Gamma \cdot p},
$$
where the closure is taken in $\proj(\Rb^d)$. Equivalently, $\Lambda_\Omega(\Gamma)$ is the set of boundary points $x \in \partial \Omega$ where there exist $p \in \Omega$ and a sequence $(\gamma_n)_{n \geq 1}$ in $\Gamma$ such that $\gamma_n(p) \to x$. The \emph{convex hull of $\Gamma$}, denoted $\Cc_\Omega(\Gamma)$, is the closed convex hull of $\Lambda_\Omega(\Gamma)$ in $\Omega$. 

Given a properly convex convex domain $\Omega \subset \proj(\Rb^d)$, the \emph{dual domain} is 
$$
\Omega^*: = \left\{ f \in \proj(\Rb^{d*}) : f(x) \neq 0 \text{ for all } x \in \overline{\Omega}\right\}. 
$$
It is straightforward to show that $\Omega^*$ is a properly convex domain of $\proj(\Rb^{d*})$ and under the natural identification $\PGL(d,\Rb) = \PGL(\Rb^{d*})$, we have $\Aut(\Omega) = \Aut(\Omega^*)$. 

 A properly convex domain $\Omega \subset \proj(\Rb^d)$ has a natural distance, called the \emph{Hilbert distance}, which is defined by
 $$
 \dist_\Omega(p,q) = \frac{1}{2} \log [a,p,q,b]
 $$
 where $L$ is a projective line containing $p,q$, $\{a,b\} = L \cap \partial \Omega$ with the ordering $a,p,q,b$ along $L$, and $[a,p,q,b]$ is the standard projective cross ratio.  
Then $(\Omega, \dist_\Omega)$ is a proper geodesic metric space and $\Aut(\Omega)$ acts on $(\Omega, \dist_\Omega)$ by isometries. Further, the line segment $[p,q]_\Omega$ joining  $p,q\in \Omega$ can be parametrized to be a geodesic in $(\Omega, \dist_\Omega)$.
 
We recall that given two subsets $A, B \subset \Omega$, the \emph{Hausdorff distance with respect to $\dist_\Omega$} between $A$ and $B$ is defined as
$$
\dist^{\Haus}_\Omega(A,B) := \max \left\{ 
\sup_{a \in A} \dist_\Omega (a,B), 
\sup_{b \in B} \dist_\Omega (b,A) \right\}.
$$
We will use the following well-known estimate on the Hausdorff distance between two line segments with respect to the Hilbert metric $\dist_\Omega$.

 \begin{observation}\label{obs:Haus dist between lines} Suppose that $\Omega \subset \proj(\Rb^d)$ is properly convex. If $p_1,p_2, q_1, q_2 \in \Omega$, then 
 $$
 \dist_\Omega^{\Haus}\left( [p_1, q_1]_\Omega, [p_2, q_2]_\Omega \right) \leq \max\left\{ \dist_\Omega(p_1,p_2), \dist_\Omega(q_1, q_2) \right\}. 
 $$
 \end{observation}  
 
 \begin{proof} See for instance~\cite[Prop.\ 5.3]{IZ2021}. \end{proof}

 \subsection{Convex hulls} A general subset of $\proj(\Rb^d)$ has no well-defined convex hull, for instance if $X = \{x_1,x_2\}$, then there is no natural way to choose between the two line projective line segments joining $x_1$ and $x_2$. However, it was observed in~\cite{IZ2019} that for certain types of subsets one can define a convex hull. We recall   these observations here. 
 
 Given a subset $X \subset \proj(\Rb^d)$ which is contained in some affine chart $\mathbb{A} \subset \proj(\Rb^d)$, let ${\rm ConvHull}_{\mathbb{A}}(X) \subset \mathbb{A}$ denote the convex hull of $X$ in $\mathbb{A}$.  For a general set (e.g. two points), this convex hull depends on the choice of $\mathbb{A}$ but when $X$ is connected we have the following. 
 
 \begin{observation}\cite[Lem.\ 5.9]{IZ2019} Suppose that $X \subset \proj(\Rb^d)$ is connected. If $\mathbb{A}_1$ and $\mathbb{A}_2$ are two affine charts which contain $X$, then 
 $$
 {\rm ConvHull}_{\mathbb{A}_1}(X)={\rm ConvHull}_{\mathbb{A}_2}(X).
 $$
 \end{observation} 
 
 This leads to the following definition. 
 
 \begin{definition} If  $X \subset \proj(\Rb^d)$ is connected and contained in some affine chart, then let ${\rm ConvHull}(X)$ denote the convex hull of $X$ in some (any) affine chart which contains $X$. 
 \end{definition} 
 
 As a consequence of the definition we have the following. 
 
 \begin{observation}\label{obs:equivariance of convex hulls} Suppose that $X \subset \proj(\Rb^d)$ is connected and contained in some affine chart. If $g \in \PGL(d,\Rb)$, then 
$$
{\rm ConvHull}(gX) = g\cdot {\rm ConvHull}(X).
$$
\end{observation} 

\subsection{Relatively Anosov representations into the projective linear group}  In the context of convex real projective geometry, it is more natural to consider representations into $\PGL(d,\Rb)$. It is also helpful to identify $\Gr_1(\Rb^d) = \proj(\Rb^d)$ and $\Gr_{d-1}(\Rb^d)=\proj(\Rb^{d*})$ and assume that the boundary map of a relatively $\Psf_1$-Anosov representation has image in $\proj(\Rb^d) \times \proj(\Rb^{d*})$. This leads to the following analogue of Definition~\ref{defn:Pk Anosov}.

\begin{definition} Suppose that $(\Gamma,\peripherals)$ is relatively hyperbolic with Bowditch boundary $\partial(\Gamma, \peripherals)$. A representation $\rho\colon \Gamma \to \PGL(d,\Rb)$ is \emph{$\Psf_1$-Anosov relative to $\peripherals$} if there exists a continuous map 
$$
\xi = (\xi^1, \xi^{d-1}) \colon \partial(\Gamma, \peripherals) \to\proj(\Rb^d) \times \proj(\Rb^{d*})
$$
which is 
\begin{enumerate} 
\item \emph{$\rho$-equivariant}: if $\gamma \in \Gamma$, then $\rho(\gamma) \circ \xi = \xi \circ \gamma$,
\item \emph{transverse}: if $x,y \in \partial(\Gamma, \peripherals)$ are distinct, then $\xi^1(x) \oplus \ker \xi^{d-1}(y) = \Rb^d$,
\item \emph{strongly dynamics preserving}: if $(\gamma_n)_{n \geq 1}$ is a sequence of elements in $\Gamma$ where $\gamma_n \to x \in \partial(\Gamma, \peripherals)$ and $\gamma_n^{-1} \to y \in \partial(\Gamma, \peripherals)$, then 
$$
\lim_{n \to \infty} \rho(\gamma_n)v = \xi^1(x)
$$
for all $v \in \proj(\Rb^d) \setminus \proj(\ker \xi^{d-1}(y))$. 
\end{enumerate}
\end{definition}

\subsection{Relatively Anosov representations from visible subgroups}
As mentioned in the introduction, a projectively visible subgroup (see Section~\ref{introsec:convex projective geometry} for the definition) acts as a convergence group on its limit set \cite[Prop.\ 3.5]{CZZ2022}. Further, if the action on the limit set is geometrically finite, then the inclusion representation is relatively $\Psf_1$-Anosov. 

\begin{proposition} \label{prop:projvis+geomfin=relAnosov}
Suppose that $\Omega \subset \proj(\Rb^d)$ is a properly convex domain and $\Gamma \leq \Aut(\Omega)$ is a projectively visible subgroup. If $\Gamma$ acts on $\Lambda_\Omega(\Gamma)$ as a geometrically finite convergence group and $\peripherals$ is a set of conjugacy representatives of the stabilizers of bounded parabolic points in $\Lambda_\Omega(\Gamma)$, then the inclusion representation $\Gamma \into \PGL(d,\Rb)$ is $\Psf_1$-Anosov relative to $\peripherals$.
\end{proposition}

\begin{proof}
By definition there exists a equivariant homeomorphism $\xi^1\colon \partial(\Gamma,\peripherals) \to \Lambda_\Omega(\Gamma)$, see~\cite{Yaman}. By the visibility property each point in $\Lambda_\Omega(\Gamma)$ is a $\Cc^1$-smooth point of $\partial \Omega$. So for every $x \in \partial(\Gamma,\peripherals)$ there exists a unique $\xi^{d-1}(x) \in \proj(\Rb^{d*})$ such that 
$$
\proj(\ker \xi^{d-1}(x)) = T_{\xi^1(x)} \partial \Omega. 
$$
Then let $\xi := (\xi^1, \xi^{d-1})$. Then $\xi$ is continuous and equivariant. By the visibility property, if $x,y \in \partial(\Gamma,\peripherals)$ are distinct, then the open line segment in $\overline{\Omega}$ joining $\xi^1(x)$ to $\xi^1(y)$ is in $\Omega$. Since $\proj(\ker \xi^{d-1}(y)) \cap \Omega = \varnothing$, we must have $\xi^1(x) \notin \proj(\ker \xi^{d-1}(y))$ and so 
$$
\xi^1(x) \oplus \ker \xi^{d-1}(y) = \Rb^d. 
$$
Thus $\xi$ is transverse. Finally, by \cite[Prop.\ 3.5]{CZZ2022}, $\xi$ is strongly dynamics preserving.
\end{proof}

\section{Relatively Anosov representations whose images preserve a properly convex domain}\label{sec:rel Anosov repn preserving a properly convex domain}

In this section we prove a converse to Proposition~\ref{prop:projvis+geomfin=relAnosov} and characterize the relatively $\mathsf{P}_1$-Anosov representations that preserve a properly convex domain. This builds upon work in ~\cite{CZZ2022} and extends results in~\cite{DGK2017,Z2017} from the classical Anosov case to the relative one. 

Let $\norm{\cdot}_2$ denote both the Euclidean norm on $\Rb^d$ and the associated dual norm on $\Rb^{d*}$. Then let $\Sb \subset \Rb^{d}$ and $\Sb^* \subset \Rb^{d*}$ denote the unit balls relative to these norms. Also let $\SL^\pm(d,\Rb) = \{ g \in \GL(d,\Rb) : \det g = \pm 1\}$. The group $\SL^\pm(d,\Rb)$ acts on $\Sb$ and $\Sb^*$ by 
$$
g \cdot v = \frac{1}{\norm{g(v)}_2} g(v) \quad \text{and} \quad g \cdot f = \frac{1}{\norm{ f \circ g^{-1}}_2} f \circ g^{-1}. 
$$

\begin{definition}\label{defn:the lifting property} Suppose that $(\Gamma, \peripherals)$ is relatively hyperbolic, $\rho \colon \Gamma \to \PGL(d,\Rb)$ is $\Psf_1$-Anosov relative to $\peripherals$, and $\xi \colon \partial(\Gamma, \peripherals) \to \proj(\Rb^d) \times \proj(\Rb^{d*})$ is the Anosov boundary map, then we say that $\rho$ has the \emph{lifting property} if there exist lifts $\tilde{\xi}=(\tilde{\xi}^1, \tilde{\xi}^{d-1}) \colon \partial(\Gamma, \peripherals) \to \Sb \times \Sb^*$ and $\tilde{\rho} \colon \Gamma \to \SL^{\pm}(d,\Rb)$ of $\xi$ and $\rho$ with the following properties: 
\begin{enumerate}
\item $\tilde{\xi}$ is continuous and $\tilde{\rho}$-equivariant,
\item $\tilde{\xi}$ is \emph{positive} in the following sense: if $x,y \in \partial(\Gamma, \peripherals)$ are distinct, then 
$$
\tilde{\xi}^{d-1}(y)(\tilde{\xi}^1(x)) > 0.
$$
\end{enumerate} 
\end{definition} 

\begin{proposition}\label{prop:lifting property one} Suppose that $(\Gamma, \peripherals)$ is relatively hyperbolic and $\rho \colon \Gamma \to \PGL(d,\Rb)$ is $\Psf_1$-Anosov relative to $\peripherals$. Then the following are equivalent: 
\begin{enumerate}
\item $\rho$ has the lifting property, 
\item there exists a properly convex domain $\Omega_0 \subset \proj(\Rb^d)$ where $\rho(\Gamma) \leq \Aut(\Omega_0)$,
\item there exists a properly convex domain $\Omega \subset \proj(\Rb^d)$ where $\rho(\Gamma) \leq \Aut(\Omega)$ is a projectively visible subgroup. 
\end{enumerate}
\end{proposition} 

\begin{remark} The equivalence (2) $\iff$ (3) follows from general results in~\cite{CZZ2022} and the implication (2) $\implies$ (1) is elementary. So the new content of Proposition~\ref{prop:lifting property one} is the implication (1) $\implies$ (2). \end{remark} 

The rest of the section is devoted to the proof of Proposition~\ref{prop:lifting property one}. So fix $(\Gamma, \peripherals)$ and $\rho$ as in the proposition, and let $\xi \colon \partial(\Gamma, \peripherals) \to \proj(\Rb^d) \times \proj(\Rb^{d*})$ denote the Anosov boundary map of $\rho$. 

\begin{lemma} (2) $\iff$ (3). \end{lemma}

\begin{proof} Using the language in~\cite{CZZ2022}, ~\cite[Prop.\ 4.4]{ZZ2022a} implies that $\rho(\Gamma)$ is a $\Psf_{k,d-k}$-transverse group. Then the equivalence of (2) and (3) follows from~\cite[Prop.\ 4.4]{CZZ2022}.
\end{proof} 

\begin{lemma}[(2) $\implies$ (1)]\label{lem:2=>1 in lifting property} If there exists a properly convex domain $\Omega_0 \subset \proj(\Rb^d)$ where $\rho(\Gamma) \leq \Aut(\Omega_0)$, then $\rho$ has the lifting property. 
\end{lemma} 

\begin{proof} We first observe that the strongly dynamics preserving property implies that $\xi^1$ has image in $\partial \Omega_0$. Fix $x \in \partial(\Gamma, \peripherals)$ and a sequence $(\gamma_n)_{n \geq 1}$ in $\Gamma$ with $\gamma_n \to x$. Passing to a subsequence we can assume that $\gamma_n^{-1} \to y \in  \partial(\Gamma, \peripherals)$. Then 
$$
\rho(\gamma_n) v \to \xi^1(x)
$$
for all $v \in \proj(\Rb^d) \smallsetminus \proj(\ker \xi^{d-1}(y))$. Since $\Omega_0$ is open, there exists $v \in \Omega_0 \smallsetminus \proj(\ker \xi^{d-1}(y))$ and hence $\xi^1(x) \in \overline{\Omega}_0$. Since $\rho(\Gamma)$ acts properly on $\Omega_0$, we must have $\xi^1(x) \in \partial\Omega_0$. So  $\xi^1$ has image in $\partial \Omega_0$. The same argument shows that $\xi^{d-1}$ has image in $\partial \Omega_0^*$. 

The rest of the argument is identical to the proof of Case 1 in \cite[Th.\ 3.1]{Z2017}. Let $\pi \colon \Rb^d \smallsetminus \{0\} \to \proj(\Rb^d)$ denote the projection map. Since $\Omega_0$ is properly convex, $\pi^{-1}(\Omega_0)$ has two connected components $C_1$ and $C_2$. Moreover, both components are properly convex cones in $\Rb^d$ and $C_2 = -C_1$. 

For $x \in \partial(\Gamma, \Pc)$ let $\tilde{\xi}^1(x) \in \Sb$ denote the unique lift of $\xi^1(x)$ in $\overline{C}_1 \cap \Sb$ and let $\tilde{\xi}^{d-1}(x)$ denote the unique lift of $\xi^{d-1}(x)$ such that $\tilde{\xi}^{d-1}(x) \in \Sb^*$ and $\tilde{\xi}^{d-1}(x)|_{C_1} > 0$. For $\gamma \in \Gamma$, let $\tilde{\rho}(\gamma) \in \SL^\pm(d,\Rb)$ denote the unique lift of $\rho(\gamma)$ which preserves $C_1$. Then $\tilde{\rho}$ is a homomorphism and $\tilde{\xi}=(\tilde{\xi}^1, \tilde{\xi}^{d-1})$ is continuous, $\tilde{\rho}$-equivariant, and positive. So $\rho$ has the lifting property. 
\end{proof} 

For the other direction we closely follow the arguments in Section 5 of~\cite{IZ2019}. 

\begin{lemma}[(1) $\implies$ (2)]\label{lem:1=>2 in lifting property} If $\rho$ has the lifting property, then there exists a properly convex domain $\Omega_0 \subset \proj(\Rb^d)$ where $\rho(\Gamma) \leq \Aut(\Omega_0)$. \end{lemma} 

\begin{proof} Let $\tilde{\xi}$, $\tilde{\rho}$ denote lifts of $\xi$, $\rho$ satisfying the lifting property. Then define
$$
C_0 := \left\{ \left[ \sum_{j=1}^N \lambda_j \tilde{\xi}^1(x_j)\right] : N \geq 2; \, \lambda_1, \dots, \lambda_N > 0; \, x_1,\dots, x_N \in \partial(\Gamma, \peripherals) \text{ distinct}\right\}.
$$
Since $\tilde{\xi}$ is $\tilde{\rho}$-equivariant, $\tilde\rho(\gamma) C_0=C_0$ for every $\gamma \in \Gamma$. Since $\tilde{\xi}$ is positive, 
\begin{equation}
\label{eqn:C0 intersect kernels} 
C_0 \cap \bigcup_{y \in \partial(\Gamma, \peripherals)} \proj\left(\ker \xi^{d-1}(y)\right) = \varnothing.
\end{equation} 
Also, if we fix $x_1, x_2 \in \partial(\Gamma, \peripherals)$ distinct, then the positivity of $\tilde{\xi}$ implies that $C_0$ is bounded in the affine chart 
$$
\mathbb{A} := \left\{ [v] \in \proj(\Rb^d) : (\tilde{\xi}^{d-1}(x_1)+\tilde{\xi}^{d-1}(x_2))(v) \neq 0\right\}.
$$

Fix $p \in C_0$. We claim that there exists a connected neighborhood $U$ of $p$ in $\proj(\Rb^d)$ such that 
$$
\rho(\Gamma)U = \bigcup_{\gamma \in \Gamma} \rho(\gamma) U
$$
is bounded in $\mathbb{A}$. Suppose not. Then there exist sequences $(p_n)_{n \geq 1}$ in $\proj(\Rb^d)$ and $(\gamma_n)_{n \geq 1}$ in $\Gamma$ such that $p_n \to p$ and $\rho(\gamma_n)p_n$ leaves every compact subset of $\mathbb{A}$. Passing to a subsequence we can suppose that $\gamma_n \to x \in \partial(\Gamma, \peripherals)$ and $\gamma_n^{-1} \to y \in \partial(\Gamma, \peripherals)$. Then, by the strongly dynamics preserving property, 
$$
\rho(\gamma_n)q \to \xi^1(x)
$$
for all $q \in \proj(\Rb^d) \smallsetminus \proj(\ker \xi^{d-1}(y))$ and the convergence is locally uniform. Equation~\eqref{eqn:C0 intersect kernels}  implies that  $p \in  \proj(\Rb^d) \smallsetminus \proj(\ker \xi^{d-1}(y))$ and so $\rho(\gamma_n)p_n \to \xi^1(x)$. However $\xi^1(x)$ lies in the closure of $C_0$ and $C_0$ is bounded in $\mathbb{A}$. This contradicts our assumption and hence such a set $U$ exists. 

 Finally the set 
 $$
 X := C_0 \cup \bigcup_{\gamma \in \Gamma} \rho(\gamma)U
 $$ is connected (since each of the sets in the union is path-connected, and $\rho(\gamma) U \cap C_0 \neq \varnothing$ for each $\gamma \in \Gamma$), bounded in $\mathbb{A}$, and preserved by $\rho(\Gamma)$. So Observation~\ref{obs:equivariance of convex hulls} implies that
$$
\Omega_0 := {\rm ConvHull}(X)
$$
is a properly convex domain where $\rho(\Gamma) \leq \Aut(\Omega_0)$. 
\end{proof} 

\section{Stability of the lifting property}\label{sec:stability of the lifting property} 

In this section we prove Proposition~\ref{prop:lifting property two in intro} which we restate here.

\begin{proposition}\label{prop:lifting property two} Suppose that $(\Gamma, \peripherals)$ is relatively hyperbolic and $\rho_0 \colon \Gamma \to \PGL(d,\Rb)$ is a representation. Let $\Ac_1(\rho_0)$ denote the set of representations in $ \Hom_{\rho_0}(\Gamma, \PGL(d,\Rb))$  which are $\Psf_1$-Anosov relative to $\peripherals$. Then the subset $\Ac_1^+(\rho_0) \subset \Ac_1(\rho_0)$ of representations with the lifting property is open and closed in $\Ac_1(\rho_0)$. 
\end{proposition}

\subsection{Lifting maps} In this subsection we record some basic observations about lifting maps to covering spaces. Suppose that $M$ is a compact Riemannian manifold and $\pi\colon \wt{M} \to M$ is a Riemannian cover (i.e. $\wt{M}$ is a Riemannian manifold and $\pi$ is a covering map which is a local isometry). Fix $\epsilon > 0$ so that every metric ball of radius $\epsilon$ in $M$ is normal. 

\begin{observation}\label{obs:lifting_one} If $p \in \wt{M}$, then
\begin{enumerate}
\item $\pi$ induces a diffeomorphism between metric balls $\mathcal{B}_{\wt{M}}(p,\epsilon) \to \mathcal{B}_M(\pi(p), \epsilon)$,
\item $\pi^{-1}(q) \cap \mathcal{B}_{\wt{M}}(p,\epsilon)$ is a single point for any $q \in \mathcal{B}_M(\pi(p), \epsilon)$. 
\end{enumerate}
\end{observation} 

\begin{proof} For part (1) see for instance the proof of ~\cite[Lem.\ 1.38]{CE2008}. Part (2) follows immediately from part (1). \end{proof} 

\begin{observation}\label{obs:lifting_two} Suppose that $N$ is a compact topological space and $f,g \colon N \to M$ are continuous maps. If 
$$
\max_{x \in N} \dist_M(f(x),g(x)) < \epsilon
$$
and $f$ admits a continuous lift $\tilde{f} \colon N \to \wt{M}$, then $g$ admits a unique continuous lift $\tilde{g} \colon N \to \wt{M}$ with 
$$
\max_{x \in N} \dist_{\wt{M}}(\tilde{f}(x),\tilde{g}(x)) < \epsilon.
$$
\end{observation} 

\begin{proof} By Observation~\ref{obs:lifting_one}, for each $x \in N$ there is a unique $\tilde{g}(x) \in \pi^{-1}(g(x))$ such that $\dist_{\wt{M}}\left( \tilde{f}(x), \tilde{g}(x)\right) < \epsilon$. By uniqueness, $\tilde{g}$ is continuous. 
\end{proof}

\subsection{Proof of Proposition~\ref{prop:lifting property two}} Suppose that $(\Gamma, \peripherals)$ is relatively hyperbolic and $\rho_0 \colon \Gamma \to \PGL(d,\Rb)$ is a representation. 

For $\rho \in \Ac_1(\rho_0)$, let $\xi_\rho$ denote the Anosov boundary map. We will use the following stability result from~\cite{ZZ2022a}. 

\begin{theorem}[{\cite[Cor.\ 13.6]{ZZ2022a}}]\label{thm:continuity of boundary maps} The map 
$$
\Ac_1(\rho_0) \times \partial(\Gamma, \peripherals) \ni (\rho, x) \mapsto \xi_\rho(x) \in \proj(\Rb^d) \times \proj(\Rb^{d*})
$$
is continuous. 
\end{theorem} 

Fix Riemannian metrics on $\Sb \times \Sb^*$ and $\proj(\Rb^d) \times \proj(\Rb^{d*})$ so that $\Sb \times \Sb^* \to \proj(\Rb^d) \times \proj(\Rb^{d*})$ is a Riemannian cover. We will let $\dist$ denote the associated distance on both spaces. Then fix $\epsilon > 0$ satisfying Observation~\ref{obs:lifting_two} with $M = \proj(\Rb^d) \times \proj(\Rb^{d*})$. 

\begin{lemma} $\Ac_1^+(\rho_0)$ is closed in $\Ac_1(\rho_0)$. \end{lemma}

\begin{proof} Suppose that $\rho_n \to \rho$ in $\Ac_1(\rho_0)$ where $\{\rho_n\} \subset \Ac_1^+(\rho_0)$. Let $\xi_n$ (respectively, $\xi$) denote the Anosov boundary map of $\rho_n$ (respectively $\rho$) and let $\tilde{\xi}_n$, $\tilde{\rho}_n$ denote lifts of $\xi_n$, $\rho_n$ satisfying the lifting property. 

Theorem~\ref{thm:continuity of boundary maps} implies that $\xi_n \to \xi$ uniformly. So for $n$ sufficiently large, we have 
$$
\max_{x \in \partial(\Gamma, \peripherals)} \dist( \xi_n(x), \xi(x)) < \epsilon. 
$$
So by our choice of $\epsilon > 0$ there exists a unique continuous lift $\tilde{\xi}$ of $\xi$ such that 
$$
\max_{x \in \partial(\Gamma, \peripherals)} \dist\left( \tilde{\xi}_n(x), \tilde{\xi}(x)\right) < \epsilon
$$
for $n$ sufficiently large.  Further, $\tilde{\xi}_n$ converges pointwise to $\tilde{\xi}$. Then
$$
\tilde{\xi}^{d-1}(y)(\tilde{\xi}^1(x)) = \lim_{n \to \infty} \tilde{\xi}_n^{d-1}(y)(\tilde{\xi}_n^1(x)) \geq 0
$$
for all $x,y \in \partial(\Gamma, \peripherals)$. So by transversality, we see that 
\begin{equation}
\label{eqn:positivity in the limit}
\tilde{\xi}^{d-1}(y)(\tilde{\xi}^1(x)) >0
\end{equation}
for all distinct $x,y \in \partial(\Gamma, \peripherals)$.  

Finally, we construct the lift of $\rho$. Since $\Gamma$ is finitely generated and $\SL^\pm(d,\Rb) \to \PGL(d,\Rb)$ is a finite cover, by passing to a further subsequence we can suppose that 
$$
\tilde{\rho}(\gamma) := \lim_{n \to \infty} \tilde{\rho}_n(\gamma)
$$
exists for all $\gamma \in \Gamma$. Since $\tilde{\xi}_n$ converges pointwise to $\tilde{\xi}$, we see that $\tilde{\xi}$ is $\tilde{\rho}$-equivariant. Hence $\rho$ has the lifting property. 
\end{proof} 

\begin{lemma} $\Ac_1^+(\rho_0)$ is open in $\Ac_1(\rho_0)$. \end{lemma}

\begin{proof} It suffices to assume that $\rho_0 \in \Ac_1^+(\rho_0)$ and show that there exists an open neighborhood of $\rho_0$ in $\Ac_1(\rho_0)$ which is contained in $\Ac_1^+(\rho_0)$. Let $\xi_{\rho_0}$ denote the Anosov boundary map of $\rho_0$ and let $\tilde{\xi}_{\rho_0}$, $\tilde{\rho}_0$ denote lifts of $\xi_{\rho_0}$, $\rho_0$ satisfying the lifting property. 

By Theorem~\ref{thm:continuity of boundary maps} and our choice of $\epsilon > 0$, we can find a neighborhood $\Oc$ of $\rho_0$ in $\Ac_1(\rho_0)$ such that if $\rho \in \Oc$, then the associated boundary map $\xi_\rho$ admits a unique continuous lift $\tilde{\xi}_\rho \colon \partial(\Gamma, \peripherals) \to \Sb \times \Sb^*$ with 
\begin{equation}\label{eqn:max is less than epsilon defines lift}
\dist_{\max}(\tilde{\xi}_{\rho_0}, \tilde{\xi}_{\rho}) : = \max_{x \in \partial(\Gamma, \peripherals)} \dist(\tilde{\xi}_{\rho_0}(x), \tilde{\xi}_{\rho}(x)) < \epsilon. 
\end{equation}

Fix a finite generating set $S \subset \Gamma$. Then we can find a sub-neighborhood $\Oc^\prime \subset \Oc$ where for each $\gamma \in S$ and $\rho \in \Oc^\prime$ there exists a lift $\tilde{\rho}(\gamma)$ of $\rho(\gamma)$ such that 
$$
\max_{v \in \Sb^* \times \Sb} \dist( \tilde{\rho}_0(\gamma)v, \tilde{\rho}(\gamma)v) < \epsilon/2.
$$
By replacing $\Oc^\prime$ with a  relatively compact subset, we can also assume that there exists $C > 1$ such that: if  $\gamma \in S$ and $\rho \in \Oc^\prime$, then $\tilde{\rho}(\gamma)$ acts as a $C$-Lipschitz map on $\Sb \times \Sb^*$. Finally, by possibly replacing $\Oc^\prime$ with a smaller neighborhood and using Theorem~\ref{thm:continuity of boundary maps}, we can assume that 
$$
\dist_{\max}(\tilde{\xi}_{\rho_0}, \tilde{\xi}_{\rho}) < \frac{\epsilon}{2C}
$$
for all $\rho \in \Oc^\prime$.

Now, if $\rho \in \Oc^\prime$ and $\gamma \in S$, then (since $\tilde\xi_{\rho_0}$ is $\tilde{\rho}_0$-equivariant)
\begin{align*}
\dist_{\max}& \left(\tilde{\xi}_{\rho_0}, \tilde{\rho}(\gamma)\circ \tilde{\xi}_{\rho} \circ \gamma^{-1} \right) =\max_{x \in \partial(\Gamma, \peripherals)}  \dist\left(\tilde{\rho}_0(\gamma)\circ \tilde{\xi}_{\rho_0} \circ \gamma^{-1}(x), \tilde{\rho}(\gamma)\circ \tilde{\xi}_{\rho} \circ \gamma^{-1}(x) \right) \\
& < \epsilon/2 + \max_{x \in \partial(\Gamma, \peripherals)}  \dist\left(\tilde{\rho}(\gamma)\circ \tilde{\xi}_{\rho_0} \circ \gamma^{-1}(x), \tilde{\rho}(\gamma)\circ \tilde{\xi}_{\rho} \circ \gamma^{-1}(x) \right) \\
& < \epsilon/2 + C \dist_{\max}(\tilde{\xi}_{\rho_0} \circ \gamma^{-1},\tilde{\xi}_{\rho} \circ \gamma^{-1}) =\epsilon/2 + C \dist_{\max}(\tilde{\xi}_{\rho_0},\tilde{\xi}_{\rho}) < \epsilon. 
\end{align*}
So by uniqueness of the lift $\tilde{\xi}_\rho$ satisfying Equation~\eqref{eqn:max is less than epsilon defines lift}, we have $\tilde{\rho}(\gamma)\circ \tilde{\xi}_{\rho} \circ \gamma^{-1} = \tilde{\xi}_{\rho}$. Since at most one lift $\tilde{g} \in \SL^\pm(d,\Rb)$ of an element $\rho(\gamma) \in \rho(\Gamma)$ can  satisfy the equation $\wt{g} \circ \tilde{\xi}_{\rho} \circ \gamma^{-1} = \tilde{\xi}_\rho$, we then see that $\tilde{\rho}$ extends to a homomorphism of $\Gamma$ and $\tilde{\xi}_{\rho}$ is $\tilde{\rho}$-equivariant. 

It remains to verify positivity. Fix a compact set $K \subset \{ (x,y) \in \partial(\Gamma, \peripherals)^2 : x \neq y\}$ such that 
\begin{equation}
\label{eqn:co-compact action} 
\Gamma \cdot K =\left\{ (x,y) \in \partial(\Gamma, \peripherals)^2 : x \neq y\right\}
\end{equation}
(such a compact set exists by~\cite[Th.\ 1C]{Tukia}). By shrinking $\Oc^\prime$, we may assume that 
\begin{equation} \label{eqn:positivity of lift}
\tilde{\xi}^{d-1}_\rho(y)\left(\tilde{\xi}^1_\rho(x)\right) >0
\end{equation}
for all $\rho \in \Oc^\prime$ and $(x,y) \in K$. Fix $\rho \in \Oc^\prime$. Since $\tilde{\xi}_{\rho}$ is $\tilde{\rho}$-equivariant, Equation~\eqref{eqn:co-compact action} implies that 
Equation~\eqref{eqn:positivity of lift} holds for all distinct $x,y \in \partial(\Gamma, \peripherals)$. Hence we see that $\rho \in \Ac_1^+(\rho_0)$. 
\end{proof}

\section{Representations of rank one groups revisited}\label{sec:when a representation preserves a properly convex domain} 

For the rest of the section let $\mathsf{G}$, $\mathsf{K}$, and $X=\mathsf{G}/\mathsf{K}$ be as in Section~\ref{sec:reminders negatively curved symspaces}. Then suppose that $\tau \colon \mathsf{G} \to \PGL(d,\Rb)$ is a $\mathsf{P}_1$-proximal representation. 

In this section we prove three propositions. The first two characterize exactly when $\tau(\mathsf{G})$ preserves a properly convex domain and the third proposition establishes a structure theorem in the case it does. The first and third propositions imply Proposition~\ref{prop: repn of rank one groups in convex setting intro}. 

\begin{proposition}\label{prop:not real hyperbolic plane}  If $X$ is not isometric to real hyperbolic 2-space (equivalently, $\mathsf{G}$ is not locally isomorphic to $\SL(2,\Rb)$), then $\tau(\mathsf{G})$ preserves a properly convex domain. 
\end{proposition} 

\begin{proposition}\label{prop:real hyperbolic plane case}  Suppose that $X$ is isometric to real hyperbolic 2-space and 
$$
\Rb^d = \bigoplus_{j=1}^m V_j
$$ 
is a decomposition into $\tau(\mathsf{G})$-irreducible subspaces. Then $\tau(\mathsf{G})$ preserves a properly convex domain if and only if $\max_{1 \leq j \leq m} \dim V_j$ is odd. 
\end{proposition}

\begin{proposition}\label{prop: if tau preserves a properly convex domain} Suppose that $\tau(\mathsf{G})$ preserves some properly convex domain. Then there exists a $\tau(\mathsf{G})$-invariant properly convex domain $\Omega \subset \proj(\Rb^d)$ such that: if $\Gamma \leq \mathsf{G}$ is a geometrically finite subgroup, then 
\begin{enumerate}
\item $\tau(\Gamma)$ is a projectively visible subgroup of $\Aut(\Omega)$ and acts geometrically finitely on its limit set. 
\item If $\Cc_\Gamma: = \Cc_\Omega(\tau(\Gamma))$, then $(\Cc_\Gamma, \dist_\Omega)$ is Gromov-hyperbolic. 
\end{enumerate}
\end{proposition} 

Arguing exactly as in the proof of Proposition~\ref{prop:representations of rank one groups} there exists a continuous $\tau$-equivariant, transverse, strongly dynamics preserving map
$$
\zeta=(\zeta^1, \zeta^{d-1})\colon \partial_\infty X \to \proj(\Rb^d) \times \proj(\Rb^{d*}).
$$

Arguing as in the first step of the proof of Lemma~\ref{lem:2=>1 in lifting property}, we obtain the following.

\begin{observation}\label{obs: bounday inclusions in homog case} If $\tau(\mathsf{G})$ preserves a properly convex domain $\Omega \subset \proj(\Rb^d)$, then 
$$
\zeta^1(\partial_\infty X) \subset \partial \Omega \quad \text{and} \quad \zeta^{d-1}(\partial X) \subset \partial \Omega^*. 
$$
\end{observation}

\subsection{Proof of Proposition~\ref{prop:not real hyperbolic plane} } Suppose that $X$ is not isometric to hyperbolic 2-space. Then $\partial_\infty X$ is a sphere with dimension at least two and in particular is simply connected. 

As in Section~\ref{sec:rel Anosov repn preserving a properly convex domain}, let $\Sb \subset \Rb^{d}$ and $\Sb^* \subset \Rb^{d*}$ denote the unit spheres relative to the Euclidean norms. Then, since $\Sb \to \proj(\Rb^d)$ is a covering map and $\partial_\infty X$ is simply connected, we can lift $\zeta^1$ to a continuous map $\tilde{\zeta}^1 \colon \partial_\infty X \to \Sb$. For the same reasons, we can lift $\zeta^{d-1}$ to a continuous map $\tilde{\zeta}^{d-1} \colon \partial_\infty X \to \Sb^*$. By transversality, 
\begin{equation}
\label{eqn:value off diagonal} 
\tilde{\zeta}^{d-1}(x) \Big( \tilde{\zeta}^1(y) \Big) \neq 0
\end{equation}
for all distinct $x,y \in \partial_\infty X$. Since $\partial_\infty X$  minus any point is connected, Equation~\eqref{eqn:value off diagonal} has the same sign for all distinct $x,y \in \partial_\infty X$. So by possibly replacing $\tilde{\zeta}^{d-1}$ by $-\tilde{\zeta}^{d-1}$, we may assume that Equation~\eqref{eqn:value off diagonal} is positive for all distinct $x,y \in \partial_\infty X$. 

Since $\partial_\infty X$ is connected, $\zeta^1$ has exactly two continuous lifts to $\Sb$. So if $g \in \mathsf{G}$ and $\tilde{h} \in \SL^\pm(d,\Rb)$ is a lift of $\tau(g)$, then either $\tilde{h} \circ \tilde{\zeta}^1 \circ g^{-1} = \tilde{\zeta}^1$ or $\tilde{h} \circ \tilde{\zeta}^1 \circ g^{-1}= -\tilde{\zeta}^1$. So for every $g \in \mathsf{G}$ there exists a unique lift $\tilde{\tau}(g) \in \SL^\pm(d,\Rb)$ of $\tau(g)$ such that $\tilde{\tau}(g) \circ \tilde{\zeta}^1 \circ g^{-1} = \tilde{\zeta}^1$. By uniqueness, $\tilde{\tau}$ is a representation. 

Then arguing as in the proof of Lemma~\ref{lem:1=>2 in lifting property} we see that $\tau(\mathsf{G})$ preserves a properly convex domain. 

\subsection{Proof of Proposition~\ref{prop:real hyperbolic plane case}} Suppose that $X$ is isometric to real hyperbolic 2-space and $\Rb^d = \bigoplus_{j =1}^m V_j$ is a decomposition into $\tau(\mathsf{G})$-irreducible subspaces. Then $\mathsf{G}$ is locally isomorphic to $\SL(2,\Rb)$ and hence $\tau$ induces a Lie algebra representation $d\tau \colon \mathfrak{sl}(2,\Rb) \to \mathfrak{sl}(d,\Rb)$. Since every such Lie algebra representation integrates to a representation $\SL(2,\Rb) \to \SL(d,\Rb)$ and $\mathsf{G}$ is connected, there exists a representation $\hat{\tau} \colon \SL(2,\Rb) \to \PSL(d,\Rb)$ with the same image as $\tau$. So by possibly replacing $\tau$ with $\hat{\tau}$, we can assume that $\mathsf{G} = \SL(2,\Rb)$.

Let $d_j := \dim V_j$ and let $\tau_j \colon \SL(2,\Rb) \to \PSL(V_j)$ be the restriction of $\tau$ to $V_j$. By possibly relabeling we can assume $d_1 \geq d_2 \geq \dots \geq d_m$. Recall that $V_j$ is isomorphic to the vector space of homogeneous polynomials in two variables with degree $d_j-1$ where $\tau_j$ acts by $\tau_j(g) f = f \circ g^{-1}$. Then one can check that
$$
\lambda_k( \tau_j(g)) = \lambda_1(g)^{d_j+1-2k}
$$
for all $g \in \SL(2,\Rb)$. Then, since $\tau$ is $\mathsf{P}_1$-proximal, we must have $d_1 > d_2$. 

Let $\iota_j \colon V_j \into \Rb^d$ be the inclusion map and let $\pi_j \colon \Rb^d \to V_j$ be the projection relative to the decomposition $\Rb^d = \oplus_{j =1}^m V_j$. Then the adjoint $\pi_j^* \colon V_j^* \to \Rb^{d*}$ of $\pi_j$, which is given by 
$$
\pi_j^*(f) = f \circ \pi_j, 
$$
defines an inclusion. Since $\tau_1$ is $\mathsf{P}_1$-proximal, the proof of Proposition~\ref{prop:representations of rank one groups} implies that there exists a boundary map $\zeta_1 \colon \partial_\infty X \to \proj(V_1) \times \proj(V_1^*)$  associated to $\tau_1$. Then by the strongly dynamics preserving property
\begin{equation}
\label{eqn:zeta versus zeta1}
\zeta = (\iota_1, \pi_1^*) \circ \zeta_1.
\end{equation} 

\begin{lemma} $\tau(\SL(2,\Rb))$ preserves a properly convex domain in $\proj(\Rb^d)$ if and only if $\tau_1(\SL(2,\Rb))$ preserves a properly convex domain in $\proj(V_1)$. \end{lemma} 

\begin{proof} First suppose $\tau(\SL(2,\Rb))$ preserves a properly convex domain $\Omega \subset \proj(\Rb^d)$. By Observation~\ref{obs: bounday inclusions in homog case} and Equation~\eqref{eqn:zeta versus zeta1} 
$$
\zeta_1^1( \partial_\infty X)=\zeta^1( \partial_\infty X) \subset \partial \Omega.
$$
Hence $C := \overline{\Omega} \cap \proj(V_1)$ is a non-empty $\tau_1(\SL(2,\Rb))$-invariant properly convex closed set in $\proj(V_1)$. Since $\tau_1$ is irreducible, $C$ must have non-empty interior in $\proj(V_1)$. So  $\tau_1(\SL(2,\Rb))$ preserves a properly convex domain in $\proj(V_1)$. 

Next suppose that  $\tau_1(\SL(2,\Rb))$ preserves a properly convex domain in $\Omega_1 \subset \proj(V_1)$. By Observation~\ref{obs: bounday inclusions in homog case} applied to $\tau_1$, 
$$
\zeta_1^{d-1}( \partial_\infty X) \subset \partial \Omega_1^*.
$$
Then Equation~\eqref{eqn:zeta versus zeta1} implies that 
$$
\proj(\ker \zeta^{d-1}(x)) \cap \Omega_1 = \varnothing
$$
for all $x \in \partial_\infty X$. 

Fix a point $p_0 \in \Omega_1$ and an affine chart $\mathbb{A} \subset \proj(\Rb^d)$ which contains $\Omega_1$ as a bounded set. Arguing as in the proof of Lemma~\ref{lem:1=>2 in lifting property},  there exists a connected neighborhood $U$ of $p_0$ in $\proj(\Rb^d)$ such that 
$$
\tau(\SL(2,\Rb)) \, U = \bigcup_{g \in \SL(2,\Rb)} \tau(g) \, U
$$
is bounded in $\mathbb{A}$.  Then the set $X := \Omega_1 \cup \tau(\SL(2,\Rb))\, U$ is connected, bounded in $\mathbb{A}$, and preserved by $\tau(\SL(2,\Rb))$. So  by Observation~\ref{obs:equivariance of convex hulls}
$$
\Omega := {\rm ConvHull}(X)
$$
is a properly convex domain where $\tau(\SL(2,\Rb)) \leq \Aut(\Omega)$. 
\end{proof}

\begin{lemma}  $\tau_1(\SL(2,\Rb))$ preserves a properly convex domain in $\proj(V_1)$ if and only if $d_1$ is odd. \end{lemma} 

\begin{proof} As described above, we can identify $V_1$ with the vector space of homogeneous polynomials in two variables $x_1,x_2$ with degree $d_1-1$. Under this identification one can check that 
$$
\zeta_1^1([a:b]) = [ (ax_2 + bx_1)^{d_1-1}]
$$
where we identify $\partial_\infty X = \proj(\Rb^2)$. 

\medskip 

\noindent \emph{Case 1:} Assume that $d_1$ is odd. Then
$$
\Omega := \{ [f] : f \in V_1 \text{ is convex  and $f > 0$ on $\Rb^2 \smallsetminus \{0\}$} \}
$$
is a properly convex domain in $\proj(V_1)$ preserved by $\tau_1(\SL(2,\Rb))$. (Notice that this set is \emph{properly} convex since any polynomial representing a point in $\Omega$ must have nonzero $x_1^{d_1-1}$ coefficient.).

\medskip 

\noindent \emph{Case 2:} Assume that $d_1$ is even. Suppose for a contradiction that $\tau_1(\SL(2,\Rb))$ preserves a properly convex domain $\Omega \subset \proj(V_1)$. Then by Observation~\ref{obs: bounday inclusions in homog case}
$$
\zeta_1^1(\partial_\infty X) \subset \partial \Omega \quad \text{and} \quad \zeta_1^{d-1}( \partial_\infty X) \subset \partial \Omega^*.
$$
However
$$
\zeta_1^1([1:t]) = [ (x_2+tx_1)^{d_1-1}] = [x_2^{d_1-1} + t x_2^{d_1-2} x_1 + \dots + t^{d_1 -1} x_1^{d_1-1}]
$$
and, since $d_1-1$ is odd, the curve $t \mapsto \zeta_1^1([1:t])$ passes through the hyperplane 
$$
H:=\proj\left( \ker \zeta_1^{d-1}([1:0]) \right)= \proj\left(\ip{ x_2^{d_1-1}, x_2^{d_1-2} x_1, \dots, x_2 x_1^{d_1-2} }\right). 
$$
So $H$ cannot be a supporting hyperplane of $\Omega$, but this contradicts Observation~\ref{obs: bounday inclusions in homog case}.
\end{proof}

\subsection{Proof of Proposition~\ref{prop: if tau preserves a properly convex domain}} Now suppose that $\tau(\mathsf{G})$ preserves a properly convex domain $\Omega_0 \subset \proj(\Rb^d)$. 

\begin{lemma}\label{lem:constructing Omega} There exists a properly convex domain $\Omega \subset \proj(\Rb^d)$ such that: 
\begin{enumerate}
\item $\Omega_0 \subset \Omega$, 
\item $\tau(\mathsf{G}) \leq \Aut(\Omega)$, 
\item $\zeta^1(\partial_\infty X) \subset \partial \Omega$ and $\zeta^{d-1}(\partial_\infty X) \subset \partial \Omega^*$, 
\item if $x,y \in \zeta^1(\partial_\infty X)$, then $(x,y)_\Omega \subset \Omega$,
\item if $x \in \partial_\infty X$, then $\zeta^1(x)$ is a $\Cc^1$-smooth point of $\partial \Omega$ and $T_{\zeta^1(x)} \partial \Omega = \proj(\ker \zeta^{d-1}(x))$. 
\item If $(g_n)_{n \geq 1}$ is a sequence in $\mathsf{G}$ with $g_n \to x \in \partial_\infty X$ and $g_n^{-1} \to y \in \partial_\infty X$, then 
$$
\tau(\gamma_n)(p) \to \zeta^1(x)
$$
for all $p \in \Omega$. 
\end{enumerate}

\end{lemma} 

\begin{proof} This is nearly identical to the proof of ~\cite[Prop.\ 4.4]{CZZ2022}. We sketch the proof here for completeness.

Fix a compact subset $K \subset \Omega_0^*$ with non-empty interior. Then let $D$ be the convex hull of $\overline{\tau(\mathsf{G}) \cdot K}$ in $\Omega_0^*$. Notice that $D$ is a properly convex domain since $K \subset D \subset \Omega_0^*$ and $K$ has non-empty interior. Then let $\Omega := D^*$. Then $\Omega$ is a properly convex domain, $\Omega_0 \subset \Omega$, and $\tau(\mathsf{G}) \leq \Aut(\Omega)$. Observation~\ref{obs: bounday inclusions in homog case} implies that $\zeta^1(\partial_\infty X) \subset \partial \Omega$ and $\zeta^{d-1}(\partial_\infty X) \subset \partial \Omega^*$. It remains to verify (4), (5), and (6).

Let $C$ be a connected component of the preimage of $\Omega$ in $\Rb^d$. Then $C$ is a properly convex cone. Also, by the strongly dynamics preserving property, 
$$
\overline{\tau(\mathsf{G}) \cdot K} = \tau(\mathsf{G}) \cdot K \cup \zeta^{d-1}(\partial_\infty X). 
$$

(4): Fix $x,y \in \zeta^1(\partial_\infty X)$ and $p \in (x,y)_\Omega$. Also fix lifts $\tilde{x}, \tilde{y} \in \overline{C}$ of $x,y$. Then $p = [ \lambda \tilde{x} + (1-\lambda)\tilde{y}]$ for some $\lambda \in (0,1)$. Suppose for a contradiction that $p \in \partial \Omega$. Then there exists $f \in \partial \Omega^* = \partial D$ such that $f(p) = 0$. We can write $f = [ \sum_{j=1}^m f_j]$ where $f_j \in \Rb^{d*}$, $f_j|_C > 0$, and 
$$
[f_j] \in \overline{\tau(\mathsf{G}) \cdot K}.
$$
 
 \noindent \emph{Case 1:} Assume $[f_1] \in \tau(\mathsf{G}) \cdot K$. Since $\tau(\mathsf{G}) \cdot K \subset \Omega_0^*$ and $\zeta^1(\partial_\infty X) \subset \partial \Omega_0$, then $f_1(\tilde{x}) > 0$ and $f_1(\tilde{y}) > 0$. So
$$
\sum_{j=1}^m f_j\left( \lambda \tilde{x} + (1-\lambda)\tilde{y}\right) \geq f_1\left( \lambda \tilde{x} + (1-\lambda)\tilde{y}\right) >0
$$
and hence $f(p) \neq 0$. Contradiction. 

\medskip
\noindent \emph{Case 2:} Assume $[f_1] \in \zeta^{d-1}(\partial_\infty X)$. Then by transversality, $f_1(x)$ and $f_1(y)$ cannot both be zero. Hence 
$$
\sum_{j=1}^m f_j\left( \lambda \tilde{x} + (1-\lambda)\tilde{y}\right) \geq f_1\left( \lambda \tilde{x} + (1-\lambda)\tilde{y}\right) >0
$$
and hence $f(p) \neq 0$. Contradiction. 

\medskip

(5): Fix $x \in \partial_\infty X$ and fix a supporting hyperplane $H$ at $\zeta^1(x)$. Then $H = \proj(\ker f)$ for some $f \in \partial \Omega^* = \partial D$. We can then write $f = [ \sum_{j=1}^m f_j]$ where $f_j \in \Rb^{d*}$, $f_j|_C > 0$, and 
$$
[f_j] \in \overline{\tau(\mathsf{G}) \cdot K}.
$$
Arguing as in the proof of (4), we see that $m=1$ and $[f_1] = \zeta^{d-1}(x)$. Hence $H =  \proj(\ker \zeta^{d-1}(x))$. Since $H$ was an arbitrary supporting hyperplane at $\zeta^1(x)$, we see that $\zeta^1(x)$ is a $\Cc^1$-smooth point of $\partial \Omega$ and $T_{\zeta^1(x)} \partial \Omega = \proj(\ker \zeta^{d-1}(x))$. 

\medskip

(6): Suppose that $g_n \to x \in \partial_\infty X$ and $g_n^{-1} \to y \in \partial_\infty X$. By the strongly dynamics preserving property
$$
\tau(g_n)(v) \to \zeta^1(x)
$$
for all $v \in \proj(\Rb^d) \smallsetminus \proj(\ker \zeta^{d-1}(y))$. Part (5) of this lemma implies that 
$$
\proj(\ker \zeta^{d-1}(y)) \cap \Omega = \varnothing
$$
and so $\tau(g_n)(p) \to \zeta^1(x)$ for all $p \in \Omega$.
\end{proof}

Let $\Cc$ denote the convex hull of $\zeta^1(\partial_\infty \Gamma)$ in $\Omega$. We will show that $\tau(\mathsf{G})$ acts cocompactly on $\Cc$. To do this we will use Lemma 8.7 in~\cite{DGK2017}, which is based on a result and argument of Kapovich-Leeb-Porti (namely, Theorem 1.1 in~\cite{KLP2017} and Proposition 5.26 in~\cite{KLP2018}). Alternatively, it is possible to give an elementary, but longer, argument following the proof of~\cite[Prop.\ 3.6]{Z2017}. 

\begin{lemma}\label{lem:co-compact on big convex hull} $\tau(\mathsf{G})$ acts cocompactly on $\Cc$. \end{lemma} 

\begin{proof} Fix a cocompact lattice $\Gamma \leq \mathsf{G}$. Then $\rho = \tau|_\Gamma$ is $\Psf_1$-Anosov and, if we identify $\partial_\infty \Gamma = \partial_\infty X$, then $\rho$ has Anosov boundary map $\zeta$. 

Let $C$ be a connected component of the preimage of $\Omega$ in $\Rb^d$. Then $C$ is a properly convex cone. Following the notation in~\cite[Sec.\ 8]{DGK2017}, let 
$$
\tilde{\Lambda}_{\rho(\Gamma)}^* := \left\{ f \in \Rb^{d*}: f|_C > 0 \text{ and } [f] \in \zeta^{d-1}(\partial_\infty X)\right\} 
$$
and 
$$
\Omega_{\rm max} := \left\{ [v] \in \proj(\Rb^d) : f(v) > 0 \text{ for all } f \in \tilde{\Lambda}_{\rho(\Gamma)}^* \right\}.
$$

Then $\Omega \subset \Omega_{\rm max}$ and so $\Omega_{\rm max} \neq \varnothing$. Further,  Lemma~\ref{lem:constructing Omega}(4) implies that $\Cc$ coincides with the convex hull of $\zeta^1(\partial_\infty \Gamma)$ in $\Omega_{\rm max}$. So Lemma 8.7 in~\cite{DGK2017} implies that $\rho(\Gamma)=\tau(\Gamma)$ acts cocompactly on $\Cc$. Thus $\tau(\mathsf{G})$ acts cocompactly on $\Cc$.
\end{proof} 

\begin{lemma}\label{lem:GH of big convex hull} $(\Cc, \dist_\Omega)$ is Gromov-hyperbolic. \end{lemma} 

\begin{proof} Since $\mathsf{G}$ contains uniform lattices, Lemma~\ref{lem:co-compact on big convex hull} and the fundamental lemma of geometric group theory implies that   $(\Cc, \dist_\Omega)$ is quasi-isometric to $X$. 
\end{proof}

We may now conclude the proof of our proposition:

\begin{proof}[Proof of Proposition~\ref{prop: if tau preserves a properly convex domain}]
Suppose that $\Gamma \leq \mathsf{G}$ is geometrically finite. 

We first observe that $\zeta^1( \Lambda_X(\Gamma)) = \Lambda_\Omega( \tau(\Gamma))$. Fix $x \in \Lambda_\Omega(\tau(\Gamma))$. Then there exists $p \in \Omega$ and a sequence $(\gamma_n)_{n \geq 1}$ in $\Gamma$ such that $\tau(\gamma_n)(p) \to x$. Passing to a subsequence we can suppose that $\gamma_n \to x^+ \in \Lambda_X(\Gamma)$ and $\gamma_n^{-1} \to x^- \in \Lambda_X(\Gamma)$. Then Lemma~\ref{lem:constructing Omega} part (6) implies that $x = \zeta^1(x^+) \in \zeta^1(\Lambda_X(\Gamma))$. Conversely, fix $x \in \zeta^1( \Lambda_X(\Gamma))$. Then there exists a sequence $(\gamma_n)_{n \geq 1}$ in $\Gamma$ such that $\gamma_n \to x$. Passing to a subsequence we can suppose that $\gamma_n^{-1} \to y$. By Lemma~\ref{lem:constructing Omega} part (6)
$$
\tau(\gamma_n)(p) \to \zeta^1(x)
$$
for all $p \in \Omega$. So $x \in \Lambda_\Omega( \tau(\Gamma))$. Thus $\zeta^1( \Lambda_X(\Gamma)) = \Lambda_\Omega( \tau(\Gamma))$. 

Then Lemma~\ref{lem:constructing Omega} parts (4) and (5) imply that $\tau(\Gamma)$ is a projectively visible subgroup of $\Aut(\Omega)$. Since $\zeta^1$ induces a homeomorphism $\Lambda_X(\Gamma) \to \Lambda_\Omega( \tau(\Gamma))$, we see that $\tau(\Gamma)$ acts geometrically finitely on its limit set. 

The inclusion $(\Cc_\Gamma,\dist_\Omega) \hookrightarrow (\Cc,\dist_\Omega)$ is isometric and hence Lemma~\ref{lem:GH of big convex hull} implies that $(\Cc_\Gamma, \dist_\Omega)$ is Gromov-hyperbolic. 
\end{proof}

\part{Miscellaneous examples}

\section{Ping-pong with unipotents in projective space} \label{sec:pingpong}

In this section we show that certain free products are relatively $\Psf_1$-Anosov. Before stating the result we need to introduce some terminology.

For $k \leq d/2$, let $\Fc_{k,d-k}=\Fc_{k,d-k}(\Kb^d)$ denote the space of partial flags of the form $F^k \subset F^{d-k} \subset \Kb^d$ where $\dim F^j = j$. A subgroup $\Gamma \leq \SL(d,\Kb)$ is \emph{$\Psf_k$-divergent} if $\lim_{n \to \infty} \frac{\mu_k}{\mu_{k+1}}(\gamma_n)= \infty$ for every escaping sequence $(\gamma_n)_{n \geq 1}$  in $\Gamma$. Such a group has well-defined limit set in $\Fc_{k,d-k}$ defined by
$$
\Lambda_{k,d-k}(\Gamma) := \{ F : \exists (\gamma_n)_{n \geq 1} \text{ in }  \Gamma \text{ with } \gamma_n \to \infty \text{ and } F = \lim (U_k, U_{d-k})(\gamma_n)\}.
$$

For relatively Anosov groups the following holds.

\begin{observation}\label{obs:limit set of rel anosov} If $(\Gamma,\peripherals)$ is relatively hyperbolic and $\rho\colon \Gamma \to \SL(d,\Kb)$ is $\Psf_k$-Anosov relative to $\peripherals$ with Anosov boundary map $\xi$, then $\rho(\Gamma)$ is $\Psf_k$-divergent and $\xi$ induces a homeomorphism 
$$
\partial(\Gamma,\peripherals) \to \Lambda_{k,d-k}(\rho(\Gamma)).
$$
In particular, if $P \in \peripherals$, then $\Lambda_{k,d-k}(\rho(P))$ consists of a single point. 
\end{observation} 

\begin{proof} The strongly dynamics preserving property and Observation~\ref{obs:strongly_dynamics_pres_div_cartan} imply that $\rho(\Gamma)$ is $\Psf_k$-divergent and $\xi$ induces a homeomorphism $\partial(\Gamma,\peripherals) \to \Lambda_{k,d-k}(\rho(\Gamma))$.
\end{proof} 

Recall that an element $g \in \SL(d,\Kb)$ is \emph{$\Psf_1$-proximal} if $\lambda_1(g) > \lambda_2(g)$. In this case, let $\ell_g^+ \in \proj(\Kb^d)$ denote the eigenline corresponding to $\lambda_1(g)$. Then there exists a unique $g$-invariant codimension one subspace $H_g^- \in \Gr_{d-1}(\Kb^d)$ such that $\ell_g^+ \oplus H_g^- = \Kb^d$. 

An element $g \in \SL(d,\Kb)$ is \emph{$\Psf_1$-biproximal} if both $g$ and $g^{-1}$ are $\Psf_1$-proximal. In this case, we let $\ell_g^-:= \ell_{g^{-1}}^+$ and $H_g^+ := H_{g^{-1}}^-$. Notice that in this case $\ell_g^+ \subset H_g^+$ and $\ell_g^- \subset H_g^-$. Moreover, by writing a $\Psf_1$-biproximal element $g$ in Jordan normal form, one can show that 
$$
g^n(F) \xrightarrow{n\to +\infty} (\ell_g^+, H_g^+)
$$
for all $F \in \Fc_{1,d-1}$ transverse to $(\ell_g^-, H_g^-)$.

\begin{proposition} \label{prop:pingpong with unipotents}
Suppose that $\gamma \in \SL(d,\Kb)$ is $\Psf_1$-biproximal, $U \leq \SL(d,\Kb)$ is a $\Psf_1$-divergent discrete weakly unipotent group where $\Lambda_{1,d}(U)=\{F_U\}$ is a single element, and $F_U$ is transverse to the flags $F_\gamma^+:=(\ell_\gamma^+, H_\gamma^+)$ and $F_\gamma^{-}:=(\ell_\gamma^-, H_\gamma^-)$. 

Then there exist $N \geq 1$ and a finite-index subgroup $U^\prime \leq U$ such that the group $\Gamma$ generated by $\gamma^N$ and $U^\prime$ is naturally isomorphic to the free product $\ip{\gamma^N} * U^\prime$ and the inclusion 
$
\Gamma \into \SL(d,\Kb)
$
is  $\Psf_1$-Anosov relative to $\left\{ U^\prime \right\}$.
\end{proposition}

\begin{remark}
It is possible to add more $\Psf_1$-biproximal elements or weakly unipotent groups, as long as their limit flags are transverse. We skip this more general case as the proof is the same, just with more notation.
\end{remark} 

The rest of the section is devoted to the proof of Proposition~\ref{prop:pingpong with unipotents}, so fix $\gamma$ and $U$ as in the statement.

Let $\Fc:=\Fc_{1,d-1}(\Kb^d)$ and let $\dist_{\Fc}$ be the distance on $\Fc$ defined by 
$$
\dist_{\Fc}(F_1, F_2) = \dist_{\proj(\Kb^d)}(F_1^1, F_2^1) + \dist_{\Gr_{d-1}(\Kb^d)}(F_1^{d-1}, F_2^{d-1}). 
$$

Fix $\epsilon > 0$ such that the metric balls 
$$
\Bc_{\Fc}(F_U,2\epsilon), \ \Bc_{\Fc}(F_\gamma^+,2\epsilon), \ \Bc_{\Fc}(F_\gamma^-,2\epsilon)
$$
are disjoint and any two flags in different balls are transverse. Let 
$$
\mathcal{Z}_U:= \overline{\Nc_{\Fc}\left(\{ F \in \Fc : F \text{ is not transverse to } F_U \}, \epsilon\right)}
$$
and 
$$
\mathcal{Z}_{\gamma}^\pm:= \overline{\Nc_{\Fc}\left(\{ F \in \Fc : F \text{ is not transverse to } F_\gamma^\pm \}, \epsilon \right)}.  
$$
After possibly shrinking $\epsilon >0$, we may also assume that 
$$
\Oc:= \Fc \smallsetminus \overline{ \mathcal{Z}_U \cup \mathcal{Z}_{\gamma}^+ \cup \mathcal{Z}_{\gamma}^-}
$$
is open and non-empty. 

\begin{lemma}\label{lem:gamma contraction} By replacing $\gamma$ by a sufficiently large power, we may assume that $\gamma^{\pm 1}$ is $\epsilon$-Lipschitz on $\Fc \smallsetminus \mathcal{Z}_{\gamma}^\mp$ and  $\gamma^{\pm 1}\left(\Fc \smallsetminus \mathcal{Z}_{\gamma}^\mp \right) \subset \Bc_{\Fc}(F_\gamma^\pm,\epsilon)$.
\end{lemma} 

\begin{proof} By conjugating we can assume that 
$$
F_\gamma^+  = ( \ip{e_1}, \ip{e_1,\dots e_{d-1}}) \quad \text{and} \quad F_\gamma^-  = (\ip{e_d}, \ip{e_2, \dots, e_d}). 
$$
Then 
$$
\gamma = \begin{pmatrix} \lambda_1 & & \\ & A & \\ & & \lambda_2 \end{pmatrix}
$$
where $\lambda_1, \lambda_2 \in \Kb$, $A \in \GL(d-2,\Kb)$, and 
$$
\abs{\lambda_1} > \lambda_1(A) \geq \lambda_{d-2}(A) > \abs{\lambda_2}.
$$
Since  
$$
\lambda_1(A) = \lim_{n \to \infty} \mu_1(A^n)^{1/n} \quad \text{and} \quad \frac{1}{\lambda_{d-2}(A)} = \lim_{n \to \infty} \mu_1(A^{-n})^{1/n},
$$
the result follows from a straightforward calculation in affine charts. 
 \end{proof}

\begin{lemma}\label{lem:U contraction} By replacing $U$ with a finite-index subgroup, we may assume that: if $u \in U \smallsetminus \{\id\}$, then $u$ is $\epsilon$-Lipschitz on $\Fc \smallsetminus \mathcal{Z}_{U}$ and $u\left(\Fc \smallsetminus \mathcal{Z}_U \right) \subset \Bc_{\Fc}(F_U,\epsilon)$.
\end{lemma}

\begin{proof} By conjugating we can assume that 
$$
F_U  = ( \ip{e_1}, \ip{e_1,\dots e_{d-1}}).
$$
Since $U$ is $\Psf_1$-divergent and $\Lambda_{1,d-1}(U) = \{ F_U\}$, for any escaping sequence $(u_n)_{n \geq 1}$ in $U$ we have
$$
\lim_{n \to \infty} \frac{1}{\mu_1(u_n)} u_n = e_1\ip{\cdot, e_n} \in \End(\Rb^d)
$$
where $\ip{\cdot, \cdot}$ is the standard Euclidean inner product. Then a straightforward calculation in affine charts provides a finite subset $K \subset U$ such that: if $u \in U \smallsetminus K$, then $u$ is $\epsilon$-Lipschitz on $\Fc \setminus \mathcal{Z}_{U}$ and $u\left(\Fc \setminus \mathcal{Z}_U \right) \subset \Bc_{\Fc}(F_U,\epsilon)$.

By \cite[Th.\ 8.1(2)]{ZZ2022a}, $U$ is finitely generated. Then $U$ is residually finite by a theorem of Malcev~\cite{Malcev1940}. So there exists a finite-index subgroup $U^\prime \leq U$ with $U^\prime \cap K = \{\id\}$. 
\end{proof} 

\begin{lemma}\label{lem:free product}
The group $\Gamma$ generated by $\gamma$ and $U$ is naturally isomorphic to the free product $\ip{\gamma} * U$.
\end{lemma}

\begin{proof}
Let $\tau \colon \ip{\gamma} * U \to \Gamma$ be the obvious homomorphism. It is clearly onto and so we just have to show that it is one-to-one. Suppose that $w$ is a non-trivial word in $\ip{\gamma} * U$. Fix $F \in \Oc$. Then Lemmas~\ref{lem:gamma contraction}  and~\ref{lem:U contraction} imply that 
$$
\tau(w)F \in  \Bc_{\Fc}(F_\gamma^+,\epsilon) \cup  \Bc_{\Fc}(F_\gamma^-,\epsilon) \cup  \Bc_{\Fc}(F_U,\epsilon).
$$
So $\tau(w)F \notin \Oc$ and hence $\tau(w) \neq \id$. 
\end{proof}

For the arguments that follow fix a finite symmetric generating set of $U$ and let $\abs{u}$ denote the associated word length of an element $u \in U$.

Next we describe the Bowditch boundary of $\Gamma$. Let $S := \{ \gamma, \gamma^{-1}\} \cup U \smallsetminus \{\id\}$ and let 
$\Wc := \{x = x_1x_2 \cdots \}$ be the set of all \emph{finite and infinite} reduced words in $S$ (i.e.\ no letter is followed by its inverse) such that 
	\begin{itemize} 
		\item $x$ has no consecutive elements in $U$, and
		\item $x$ does not end in $U$.
	\end{itemize}
We assume that the empty word $\varnothing$ is an element of $\Wc$. Also, let $\Wc_\infty \subset \Wc$ denote the subset of infinite-length words.  
Informally, finite-length words correspond to parabolic boundary points; this will be made more precise presently.

Since $\Gamma$ is naturally isomorphic to the free product $\ip{\gamma} * U$, $\Wc$ admits a natural action of $\Gamma$, where $\Gamma$ acts on non-empty words by left-multiplication, $\gamma^{\pm 1} \cdot \varnothing = \gamma^{\pm 1}$, and $U \cdot \varnothing = \varnothing$. Notice that if $x=x_1 \cdots x_m \in \Wc \smallsetminus \Wc_\infty$, then 
$$
{\rm Stab}_\Gamma(x) =(x_1 \cdots x_m) U (x_1 \cdots x_m)^{-1}. 
$$
Further, $\Wc$ has a natural topology which can be described as follows. For $x=x_1x_2 \cdots \in \Wc_\infty$ and $N \geq 1$ let 
$$
B_N(x) := \{ y_1 y_2 \cdots \in \Wc : y_n = x_n \text{ for all } n \leq N\}. 
$$
For $x =x_1 \cdots x_m \in \Wc \smallsetminus \Wc_\infty$ and $N \geq 1$ let 
$$
B_N(x) := \{x \} \cup \{ y_1 y_2 \cdots \in \Wc : y_n = x_n \text{ for all } n \leq m, \ y_{m+1} \in U, \ \text{and} \ \abs{y_{m+1}} \geq N\}.
$$
Then $\{ B_N(x) : x \in \Wc, N \geq 1\}$ generates a topology on $\Wc$. 

With this topology, one can check that $\Gamma$ acts as a convergence group on $\Wc$, the points in $\Wc_\infty$ are conical limit points, and the points in $\Wc \smallsetminus \Wc_\infty$ are bounded parabolic points. So $\Gamma$ is relatively hyperbolic with respect to $ \peripherals:=\{U\}$ and we can identify $\partial(\Gamma,  \peripherals) = \Wc$. 

Next we define boundary maps for the inclusion $\Gamma \hookrightarrow \SL(d,\Kb)$. 

\begin{lemma} If $x = x_1 x_2 \cdots \in \Wc_\infty$ and $F \in \Oc$, then the limit 
$$
F_x : = \lim_{n \to \infty} x_1 \cdots x_n(F)
$$
exists and does not depend on $F \in \Oc$. 
\end{lemma} 

\begin{proof} If $F \in \Oc$, then Lemmas~\ref{lem:gamma contraction}  and~\ref{lem:U contraction} imply that 
$$
\dist_{\Fc}( x_1 \cdots x_{n+1}(F_1), x_1 \cdots x_n(F_1)) \leq \epsilon^n \diam \Fc
$$
and so $( x_1 \cdots x_n(F) )_{n \geq 1}$ is a Cauchy sequence and hence the limit exists. 

Further, if $F_1, F_2 \in \Oc$, then Lemmas~\ref{lem:gamma contraction}  and~\ref{lem:U contraction} imply that 
$$
\dist_{\Fc}( x_1 \cdots x_n(F_1), x_1 \cdots x_n(F_2)) \leq \epsilon^n \diam \Fc. 
$$
So the limit does not depend on $F \in \Oc$. 
\end{proof} 

Define $\xi \colon \partial(\Gamma,  \peripherals) \to \Fc$ by 
$$
\xi(x) = \begin{cases} 
F_x & \text{if } x \in \Wc_\infty \\
(x_1 \cdots x_m)F_U & \text{if } x=x_1 \cdots x_m \in \Wc \smallsetminus \Wc_\infty
\end{cases}.
$$
Notice that 
$$
\xi(x) = (x_1 \cdots x_m) \xi(x_{m+1} \cdots)
$$
for all $x = x_1 x_2 \cdots \in \Wc$ and hence $\xi$ is $\rho$-equivariant.

\begin{lemma} $\xi$ is continuous. \end{lemma} 

\begin{proof} Fix a converging sequence $y_n \to x$ in $\Wc$. 

\medskip 

\noindent \emph{Case 1:} Assume $x = x_1 x_2 \cdots \in \Wc_\infty$. Suppose $y_n=y_{n,1} y_{n,2} \cdots$. Then for any $j \geq 1$, $y_{n,j} = x_j$ for $n$ sufficiently large (depending on $j$). So for any $m \geq 1$, Lemmas~\ref{lem:gamma contraction}  and~\ref{lem:U contraction} imply that 
\begin{align*}
\limsup_{n \to \infty} \dist_{\Fc}(\xi(x), \xi(y_n))&= \limsup_{n \to\infty}  \dist_{\Fc}\Big( x_1 \cdots x_{m} \xi(x_{m+1}  \cdots), x_1 \cdots x_{m} \xi(y_{n,m+1}  \cdots)\Big)\\
& \leq \epsilon^m \diam \Fc.
\end{align*}
Since $m \geq 1$ was arbitrary and $\epsilon \in (0,1)$, we have $\xi(y_n) \to \xi(x)$. 

\medskip 

\noindent \emph{Case 2:} Assume $x=x_1 \cdots x_m \in \Wc \setminus \Wc_\infty$. We may assume that $y_n \neq x$ for all $n$. Then passing to a tail of $(y_n)_{n \geq 1}$ we may assume that $y_n = x_1 \cdots x_m y_{n,m+1} \bar{y}_n$ where $\bar{y}_n \in \Wc$, $y_{n,m+1} \in U$, and $\abs{y_{n,m+1}} \to \infty$. Then 
$$
\xi(y_n) = x_1 \cdots x_m y_{n,m+1} \xi(\bar{y}_n).
$$
The word $\bar{y}_n$ has to start with either $\gamma$ or $\gamma^{-1}$, hence Lemma~\ref{lem:gamma contraction} implies that $\xi(\bar{y}_n) \in \Bc(F_\gamma^+, \epsilon) \cup \Bc(F_\gamma^-, \epsilon)$. So, by our choice of $\epsilon > 0$, any accumulation point of $(\xi(\bar{y}_n))_{n \geq 1}$ is transverse to $F_U$. Thus, since $\abs{y_{n,m+1}} \to \infty$ and $\Lambda_{1,d-1}(U) = \{F_U\}$, Observation~\ref{obs:strongly_dynamics_pres_div_cartan} implies that
$$
\lim_{n \to \infty} \xi(y_n) = x_1 \cdots x_m \lim_{n \to \infty} y_{n,m+1} \xi(\bar{y}_n)=x_1 \cdots x_m F_U = \xi(x). 
$$

So $\xi$ is continuous. 
\end{proof} 

\begin{lemma} $\xi$ is transverse. \end{lemma} 

\begin{proof} Fix $x,y \in \Wc$ distinct. After possibly relabelling and translating by $\Gamma$ it is enough to consider the following cases. 

\medskip 

\noindent \emph{Case 1:} Assume $x \neq \varnothing$, $y \neq \varnothing$, and $x_1 \neq y_1$. Then 
$$
\xi(x) = x_1 \xi(x_2 \cdots) \quad \text{and} \quad \xi(y) = y_1 \xi(y_2 \cdots).
$$
So Lemmas~\ref{lem:gamma contraction}  and~\ref{lem:U contraction} imply that 
$$
\xi(x), \xi(y) \in \Bc_{\Fc}(F_U,\epsilon) \cup \Bc_{\Fc}(F_\gamma^+,\epsilon) \cup \Bc_{\Fc}(F_\gamma^-,\epsilon).
$$
Since $x_1 \neq y_1$, they are contained in different balls and so, by our choice of $\epsilon > 0$,   $\xi(x)$ and $\xi(y)$ are transverse. 

\medskip 

\noindent \emph{Case 2:} Assume $x =\varnothing$ and $y \neq \varnothing$. After possibly translating by an element of $U$, we may also assume that $y_1 \notin U$. Then 
$$
 \xi(y) \in  \Bc_{\Fc}(F_\gamma^+,\epsilon) \cup \Bc_{\Fc}(F_\gamma^-,\epsilon)
$$
and so $\xi(y)$ is transverse to $\xi(x) = F_U$. 
\end{proof} 

\begin{lemma} $\xi$ is strongly dynamics preserving. \end{lemma} 

\begin{proof} Suppose that $(\gamma_n)_{n \geq 1}$ is an escaping sequence in $\Gamma$ with $\gamma_n \to x \in \Wc$ and $\gamma_n^{-1} \to y \in \Wc$. We claim that
$$
\lim_{n \to \infty} \gamma_n F = \xi(x)
$$
for all $F \in \Oc$. To that end fix $F \in \Oc$. 

By Lemma~\ref{lem:free product}, we can write $\gamma_n = z_{n,1} z_{n,2} \cdots z_{n,m_n}$ as a reduced word in $S$ which has no consecutive elements in $U$.

\medskip 

\noindent \emph{Case 1:} Assume $x=x_1 x_2 \cdots  \in \Wc_\infty$. Then $z_{n,j} = x_j$ for $n$ sufficiently large (depending on $j$). For any $k \geq 1$ and $n$ sufficiently large (depending on $k$), Lemmas~\ref{lem:gamma contraction}  and~\ref{lem:U contraction} imply that 
\begin{align*}
\dist_{\Fc}( x_1 \cdots x_k F, \gamma_n F) & = \dist_{\Fc}( x_1 \cdots x_k F, x_1 \cdots x_k z_{n,k+1} \cdots z_{n,m_n} F) 
\leq \epsilon^k \diam \Fc. 
\end{align*}
So 
$$
\lim_{n \to \infty} \gamma_n F = \lim_{k \to \infty} x_1 \cdots x_k F = F_x = \xi(x). 
$$

\noindent \emph{Case 2:} Assume $x = x_1 \cdots x_m \in \Wc \smallsetminus \Wc_\infty$. Then passing to a tail of $\gamma_n$ we can assume that $z_{n,j} = x_j$ for all $1 \leq j \leq m$, $z_{n,m+1} \in U$, and $\lim_{n \to \infty} \abs{z_{n,m+1}}=\infty$. Let $\bar{\gamma}_n := z_{n,m+2} \cdots z_{n,m_n}$. If $\bar{\gamma}_n=\id$, then by Observation~\ref{obs:strongly_dynamics_pres_div_cartan}
$$
\lim_{n \to \infty} \gamma_n F = x_1 \cdots x_m \lim_{n \to \infty} z_{n,m+1} F =  x_1 \cdots x_m F_U = \xi(x)
$$
since $F \in \Oc$ is transverse to $F_U$ and $\Lambda_{1,d-1}(U) = \{F_U\}$.  Otherwise, if $\bar{\gamma}_n\neq \id$, then $z_{n,m+2} \in \{ \gamma, \gamma^{-1}\}$. So 
$$
\bar{\gamma}_n F \in \Bc_{\Fc}(F_\gamma^+,\epsilon) \cup \Bc_{\Fc}(F_\gamma^-,\epsilon).
$$ 
In particular, any accumulation point of $(\bar{\gamma}_n F)_{n \geq 1}$ is transverse to $F_U$. Then since $z_{n,m+1} \in U$, $\lim_{n \to \infty} \abs{z_{n,m+1}}=\infty$, and $\Lambda_{1,d-1}(U) = \{ F_U\}$, Observation~\ref{obs:strongly_dynamics_pres_div_cartan} implies that
\begin{equation*}
\lim_{n \to \infty} \gamma_n F = x_1 \cdots x_m \lim_{n \to \infty} z_{n,m+1} \bar{\gamma}_n F =  x_1 \cdots x_m F_U = \xi(x). 
\end{equation*}

Similar reasoning shows that 
$$
\lim_{n \to \infty} \gamma_n^{-1} F = \xi(y)
$$
for all $F \in \Oc$. Thus by Observation~\ref{obs:strongly_dynamics_pres_div_cartan}
$$
\lim_{n \to \infty} \gamma_n V = \xi^k(x)
$$
for all $V \in \Gr_k(\Kb^d)$ transverse to $\xi^{d-k}(y)$. 
\end{proof} 

Thus the inclusion $\Gamma \hookrightarrow \SL(d,\Kb)$ is  $\Psf_1$-Anosov relative to $\Pc=\left\{ U \right\}$.

\section{Pappus--Schwartz representations} \label{sec:pappus--schwartz}

In \cite{Schwartz}, certain representations of the (projectivized) modular group $\PSL(2,\Zb)$ into $\PGL(3,\Rb)$ were obtained by considering the iterated application of Pappus's Theorem, from projective geometry, on certain configurations of points and lines in the real projective plane. 

Here we establish that these representations are relatively Anosov. This mostly involves reformulating results in \cite{Schwartz} in the language of (relatively) Anosov representations.

We first define the configurations of points and lines we consider. Given points $p,q \in \proj(\Rb^3)$, write $pq$ to denote the projective line containing $p$ and $q$. Dually, given projective lines $P, Q \subset \proj(\Rb^3)$ write $PQ$ to denote the intersection of the lines $P$ and $Q$.
In the discussion that follows we identify elements of $\Gr_2(\Rb^3)$ with projective lines in $\proj(\Rb^3)$.

\begin{itemize}
\item An {\bf overmarked box} is a pair of 6-tuples $\left( (p,q,r,s,t,b), (P,Q,R,S,T,B) \right)$ in $( \proj(\Rb^3) )^6 \times ( \Gr_2(\Rb^3) )^6$ satisfying the incidence relations required by Pappus's Theorem (shown in the following figure.)

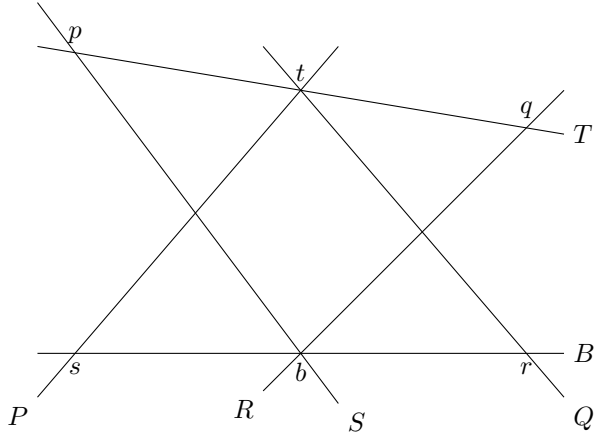
\begin{center} \begin{tikzpicture}
\draw (-0.5,49/12) -- (0,4) node[anchor=south]{$p$} -- (3,3.5) node[anchor=south]{$t$} -- (6,3) node[anchor=south]{$q$} -- (6.5,35/12) node[anchor=west]{$T$};
\draw (-0.5,0) -- (0,0) node[anchor=north]{$s$} -- (3,0) node[anchor=north]{$b$} -- (6,0) node[anchor=north]{$r$} -- (6.5,0) node[anchor=west]{$B$};
\draw (-0.5,-7/12) node[anchor=north east]{$P$} -- (0,0) -- (3,3.5) -- (3.5,49/12);
\draw (2.5, 49/12) -- (3,3.5) -- (6,0) -- (6.5,-7/12) node[anchor=north west]{$Q$};
\draw (-0.5,14/3) --(0,4) -- (3,0) -- (3.5,-2/3) node[anchor=north west]{$S$};
\draw (2.5,-0.5) node[anchor=north east]{$R$} -- (3,0) -- (6,3) -- (6.5,3.5);
\end{tikzpicture} \end{center}

\item A {\bf marked box} is an equivalence class of overmarked boxes under the involution 
$$ 
\left( (p,q,r,s,t,b), (P,Q,R,S,T,B) \right) \mapsto \left( (q,p,s,r,t,b), (Q,P,S,R,T,B) \right)
$$
(corresponding to ``flipping around the central axis $tb$''.) 

\item The {\bf convex interior} of a marked box is the open quadrilateral with vertices $p,q,r,s$ (in that order). We shall not make much direct use of convex interiors of marked boxes here, but they are useful mental tools for thinking of these objects geometrically.
\end{itemize}

\newcommand{\markbox}{\mathfrak{B}}

Given a marked box $\mathfrak{B}$, let $\rho_\markbox\colon \PSL(2,\Zb) \to \PGL(3,\Rb)$ be a Pappus--Schwartz representation as defined in \cite[Th.\ 2.4]{Schwartz} (see also \cite{BLV}). This representation is defined as follows. First, let 
$$
a = \begin{bmatrix} -1 & 1 \\ -1 & 0 \end{bmatrix} \quad \text{and} \quad d = \begin{bmatrix} 0 & 1 \\ -1 & 0 \end{bmatrix}.
$$
Then $\PSL(2,\Zb)$ has presentation  $\PSL(2,\Zb) = \langle a, d : a^3 = d^2 = 1 \rangle$. Then:
\begin{itemize} 
\item $\rho_\markbox(d)$ is the projective duality which sends 
$$
\markbox = \left[ \left( (p,q,r,s,t,b), (P,Q,R,S,T,B) \right) \right]
$$ 
to its ``dual'' / ``exterior'' marked box
$$
\iota(\markbox) := \left[ \left( (s,r,p,q,b,t), (R,S,Q,P,B,T) \right) \right] ,
$$ and 
\item $\rho_\markbox(a)$ is the 3-cycle which cycles between the original box, the dual to the ``top'' box produced by an application of Pappus's theorem to $\markbox$, and the dual to the ``bottom'' box (see \cite[Fig.\ 2.3]{Schwartz}).
In symbols, $\rho_\markbox(a)$ sends 
$$
\markbox = \left[ \left( (p,q,r,s,t,b), (P,Q,R,S,T,B) \right) \right]
$$ 
to
$$
\left[ \left( (PS,QR,p,q,(qs)(pr),t), (qs,pr,Q,P,(QR)(PS),T) \right) \right]
$$
to 
$$
\left[ \left( (s,r, PS, QR, b, (qs)(pr)), (S, R, qs, pr, B, (QR)(PS)) \right) \right]
$$
back to $\markbox$.
\end{itemize} 

Next let $\Hb^2_{\Rb}$ denote real hyperbolic 2-space and identify $\PSL(2,\Rb) = \Isom_0(\Hb^2_{\Rb})$ via the Poincar\'e upper half-plane model. If we let $\peripherals$ denote a set of representatives for the conjugacy classes of maximal parabolic subgroups in $\PSL(2,\Zb)$, then $\PSL(2,\Zb)$ is relatively hyperbolic with respect to $\peripherals$ and the Bowditch boundary naturally identifies with the Gromov boundary $\partial_\infty \Hb^2_{\Rb}$ of $\Hb^2_{\Rb}$. 

By \cite[Sec.\ 3.2, 3.3]{Schwartz} (see also \cite[Sec.\ 5.3]{BLV}), there is a continuous $\rho_\markbox$-equivariant map 
$$
\xi_\markbox = (\xi_\markbox^1,\xi^2_\markbox) \colon \partial_\infty \Hb^2_{\Rb} \to \proj(\Rb^3) \times \Gr_2(\Rb^3).
$$ 
Moreover, this map is transverse \cite[Th.\ 3.3]{Schwartz}.

The strongly dynamics preserving property follows from the proof of \cite[Lem.\ 4.2.3]{Schwartz}. For the reader's convenience we will derive the property directly from the statement of \cite[Lem.\ 4.2.3]{Schwartz}.

\begin{lemma}\cite[Lem.\ 4.2.3]{Schwartz} If $(\gamma_n)_{n \geq 1}$ is a sequence in $\PSL(2,\Zb)$ and $\epsilon > 0$, then there exist $N \geq 1$ and $x,y \in \partial_\infty \Hb^2_{\Rb}$ such that 
$$
\rho_\markbox(\gamma_N)\left( \proj(\Rb^3) \smallsetminus \Nc_{\proj}( \xi^2_\markbox(y), \epsilon)\right) \subset \Bc_{\proj}(\xi_\markbox^1(x), \epsilon)
$$
(where $\Nc_{\proj}$ and $\Bc_{\proj}$ denote respectively an open neighborhood and an open ball with respect to the angle metric defined on $\proj(\Rb^3)$ in Section~\ref{sec:angle metrics}).
\end{lemma}

\begin{proposition}
$\xi_\markbox$ is strongly dynamics-preserving.

\begin{proof} Suppose that $(\gamma_n)_{n \geq 1}$ is a sequence in $\PSL(2,\Zb)$ such that $\gamma_n \to x \in \partial_\infty \Hb^2_{\Rb}$ and $\gamma_n^{-1} \to y \in \partial_\infty \Hb^2_{\Rb}$. It is enough to verify that every subsequence of $(\gamma_n)_{n \geq 1}$ has a subsequence that verifies the strongly dynamics preserving property. 

So fix a subsequence $(\gamma_{n_j})_{j \geq 1}$. Replacing $(\gamma_{n_j})_{j \geq 1}$ by a subsequence we can suppose that for each $j \geq 1$ there exist $x_j, y_j \in \partial_\infty \Hb^2_{\Rb}$ such that
$$
\rho_\markbox(\gamma_{n_j})\Big( \proj(\Rb^3) \smallsetminus \Nc( \xi^2_\markbox(y_j), 2^{-j})\Big) \subset \Bc_{\proj}(\xi_\markbox^1(x_j), 2^{-j}).
$$
Passing to a further subsequence, we can assume that $x_j \to x_\infty$ and $y_j \to y_\infty$. Then 
$$
\rho_\markbox(\gamma_{n_j})v \to \xi_\markbox^1(x_\infty)
$$
for all $v \in \proj(\Rb^3) \smallsetminus \xi_\markbox^2(y_\infty)$. 

Fix $z \in \partial_\infty \Hb^2_{\Rb} \smallsetminus \{x,y,x_\infty, y_\infty\}$. Then by the transversality, equivariance, and continuity of the boundary map, 
$$
\xi_\markbox^1(x_\infty)= \lim_{j \to \infty} \rho_\markbox(\gamma_{n_j})\xi_\markbox^1(z) =  \lim_{j \to \infty} \xi_\markbox^1(\gamma_{n_j}(z))=  \xi_\markbox^1(x).
$$
By Observation~\ref{obs:strongly_dynamics_pres_div_cartan} and transversality, 
$$
\xi_\markbox^2(y_\infty)= \lim_{j \to \infty} \rho_\markbox(\gamma_{n_j})^{-1}\xi_\markbox^2(z). 
$$
So a similar argument also shows that $ \xi_\markbox^2(y_\infty)= \xi_\markbox^2(y)$. Thus
$$
\rho_\markbox(\gamma_{n_j})v \to \xi_\markbox(x)
$$
for all $v \in \proj(\Rb^3) \smallsetminus \xi^2_\markbox(y)$.
\end{proof}
\end{proposition}

Hence $\xi_\markbox$ is $\Psf_1$-Anosov relative to $\peripherals$.

\section{Semisimplification} \label{sec:semisimplification} 

A representation into $\SL(d,\Kb)$ is called \emph{semisimple} if the Zariski closure of its image is a reductive group. Associated to a representation $\rho\colon \Gamma \to \SL(d,\Kb)$ there is a natural conjugacy class of semisimple representations defined as follows. Let $\mathsf{G}$ be the Zariski closure of $\rho(\Gamma)$ in $\SL(d,\Kb)$ and  choose a Levi decomposition $\mathsf{G} = \mathsf{L} \ltimes \mathsf{U}$, where $\mathsf{U}$ is the unipotent radical of $\mathsf{G}$. Let $\rho^{ss}$ denote the representation obtained by composing $\rho$ with the projection onto $\mathsf{L}$. We call any representation in the conjugacy class of $\rho^{ss}$ a \emph{semisimplification} of $\rho$. Since $\mathsf{L}$ is unique up to conjugation, this definition does not depend on the chosen Levi decomposition.  

When $\Gamma$ is a word-hyperbolic group, it is known that $\rho$ is $\Psf_k$-Anosov if and only if some (any) semisimplifcation of $\rho$ is $\Psf_k$-Anosov \cite[Prop.\ 4.13]{GGKW}. This is quite useful, see for instance the proof of Theorem 1.2 in~\cite{CT2020} or the proof of Proposition 1.2 in~\cite{KP2022}. 
 
In this section we observe that the forward direction of this statement  is also true for relatively Anosov representations, while the backward direction is false. 

\begin{proposition} There exists a representation $\rho \colon \Gamma \to \SL(d,\Kb)$ of a relatively hyperbolic group $(\Gamma,\peripherals)$  where every semisimplification of $\rho$ is $\Psf_1$-Anosov relative to $\peripherals$, but $\rho$ is not $\Psf_1$-Anosov relative to $\peripherals$. \end{proposition}

\begin{proof} Let $\Gamma = \ip{a,b} \leq \PSL(2,\Rb)$ be a discrete free group where $a$ is hyperbolic and
$$
b = \begin{bmatrix} 1 & 1 \\ 0 & 1 \end{bmatrix}.
$$
Then $\Gamma$ is hyperbolic relative to $\peripherals=\{ \ip{b}\}$. Fix lifts $\tilde{a}, \tilde{b} \in \SL(2,\Rb)$ of $a,b \in \PSL(2,\Rb)$ and consider the representation $\rho \colon \Gamma \to \SL(4,\Rb)$ defined by 
$$
\rho(a) = \id_2 \oplus \tilde{a} \quad \text{and} \quad \rho(b) = \tilde{b} \oplus \tilde{b}.
$$
Notice that 
$$
\lim_{n \to \infty} \rho(b^n)[x_1 : x_2 : x_3 : x_4] =[x_2: 0 : x_4 : 0]
$$
for all $[x_1 :x_2:x_3 : x_4] \in \proj(\Rb^4)$ with $x_2 \neq 0$ or $x_4 \neq 0$.  So there cannot exist a $\rho$-equivariant strongly dynamics preserving map into $\proj(\Rb^d) \times \Gr_{d-1}(\Rb^d)$ and so $\rho$ is not $\Psf_1$-Anosov relative to $\peripherals$. However, the representation $\rho^{ss} \colon \Gamma \to \SL(4,\Rb)$ defined by 
$$
\rho^{ss}(a) = \id_2 \oplus \tilde{a} \quad \text{and} \quad \rho(b) = \id_2 \oplus \tilde{b}
$$
is a semisimplification of $\rho$ and is $\Psf_1$-Anosov relative to $\peripherals$. 
\end{proof}

\begin{proposition}\label{prop:ss is Anosov}
Suppose that $(\Gamma,\peripherals)$ is relatively hyperbolic. If $\rho\colon \Gamma \to \SL(d,\Kb)$ is $\Psf_k$-Anosov relative to $\peripherals$, then so is every semisimplification of $\rho$.
\end{proposition}

The rest of the section is devoted to the proof of Proposition~\ref{prop:ss is Anosov}. So fix a relatively hyperbolic group $(\Gamma,\peripherals)$ and a representation $\rho\colon \Gamma \to \SL(d,\Kb)$ which is $\Psf_k$-Anosov relative to $\peripherals$. Then fix a semisimplification $\rho^{ss}$ of $\rho$. 

If $\gamma \in \Gamma$ is a loxodromic element (see~\cite[Sec.\ 3.2]{ZZ2022a}), then let $\gamma^\pm \in \partial(\Gamma,\peripherals)$ denote the attracting/repelling fixed points of $\gamma$. 

Following the proof of \cite[Prop.\ 4.13]{GGKW}, there exists a $\rho^{ss}$-equivariant, transverse, continuous map $\xi_{ss} \colon \partial(\Gamma, \peripherals) \to \Gr_k(\Kb^d) \times \Gr_{d-k}(\Kb^d)$ with the following property (called \emph{dynamics preserving} in~\cite{GGKW}): if $\gamma \in \Gamma$ is a loxodromic element, then $\rho^{ss}(\gamma)$ is $\Psf_k$-proximal and $\xi^k_{ss}(\gamma^+)$, $\xi^{d-k}_{ss}(\gamma^-)$ are the attracting/repelling spaces of $\rho^{ss}(\gamma)$. 

It remains to show that $\xi_{ss}$ is strongly dynamics preserving. We begin by showing $\rho^{ss}$ is $\mathsf{P}_k$-divergent. 

\begin{lemma}\label{lem:the ss is p1 divergent} $\displaystyle \lim_{n \to \infty} \frac{\mu_k}{\mu_{k+1}}(\rho^{ss}(\gamma_n)) = \infty$ for any escaping sequence $(\gamma_n)_{n \geq 1}$  in $\Gamma$.
\end{lemma}

\begin{proof} By Theorem~\ref{thm:equivalence of definitions} there exists a weak cusp space $X$ for $(\Gamma,\peripherals)$ such that $\rho$ is $\Psf_k$-Anosov relative to $X$. Fix $x_0 \in X$. Then by~\cite[Th.\ 6.1]{ZZ2022a} there exist $\alpha, \beta > 0$ such that 
$$
\log \frac{\mu_{k}}{\mu_{k+1}}(\rho(\gamma)) \geq \alpha \dist_X(x_0, \gamma(x_0))-\beta
$$
for all $\gamma \in \Gamma$. So, for $\gamma \in \Gamma$ we have 
\begin{align}
\label{eqn:big gap on eigenvalues}
\log \frac{\lambda_{k}}{\lambda_{k+1}}(\rho^{ss}(\gamma))=\log \frac{\lambda_{k}}{\lambda_{k+1}}(\rho(\gamma))=\lim_{n \to \infty} \frac{1}{n} \log \frac{\mu_{k}}{\mu_{k+1}}(\rho(\gamma)^n) \geq \alpha \ell_X(\gamma)
\end{align}
where $\ell_X(\gamma) := \lim_{n \to \infty} \frac{1}{n} \dist_X(x_0, \gamma^n(x_0))$. 

Since $\rho^{ss}$ is semisimple, by~\cite[Th.\ 4.12]{GGKW} there exist  $C_1 > 1$ and a finite set $F_1 \subset \Gamma$ with the following property: 
for every $\gamma \in \Gamma$ there is some $f \in F_1$ such that
\begin{align}
\label{eqn:singular and eigen are comparable}
\frac{1}{C_1} \mu_j(\rho^{ss}(\gamma))  \leq  \lambda_j(\rho^{ss}(\gamma f)) \leq C_1 \mu_j(\rho^{ss}(\gamma)) ,
\end{align}
for all $1 \leq j \leq d$.

Now fix an escaping sequence $(\gamma_n)_{n \geq 1}$. It suffices to consider the case when 
$$
\lim_{n \to \infty} \frac{\mu_k}{\mu_{k+1}}(\rho^{ss}(\gamma_n))
$$
exists in $\Rb \cup \{\infty\}$ and show that the limit is infinite. Passing to a subsequence we can suppose that $\gamma_n \to x$ and $\gamma_n^{-1} \to y$. 
Pick $\alpha \in \Gamma$ such that $\alpha^{-1}(y) \notin F_1(x)$. For each $n$ fix $f_n \in F_1$ such that $\gamma_n\alpha f_n$ satisfies Equation~\eqref{eqn:singular and eigen are comparable}. Passing to a further subsequence we can suppose that $f: = f_n$ for all $n$. Then $\gamma_n \alpha f \to x$ and $(\gamma_n \alpha f)^{-1} \to f^{-1} \alpha^{-1}(y)$. By our choice of $\alpha$, we have $f^{-1} \alpha^{-1}(y) \neq x$ which implies that $\gamma_n \alpha f$ is a loxodromic element for $n$ sufficiently large. Further, $( \gamma_n \alpha f )^+ \to x$ and $(\gamma_n \alpha f )^- \to y$. Then, since  $(\gamma_n\alpha f)$ is escaping sequence, we must have $\lim_{n \to \infty} \ell_X(\gamma_n \alpha f) = \infty$.

Then by Equations~\eqref{eqn:singular and eigen are comparable} and~\eqref{eqn:big gap on eigenvalues}
\begin{align*}
\lim_{n \to \infty} \frac{\mu_k}{\mu_{k+1}}(\rho^{ss}(\gamma_n)) & \gtrsim \lim_{n \to \infty} \frac{\mu_k}{\mu_{k+1}}(\rho^{ss}(\gamma_n \alpha f)) \gtrsim \lim_{n \to \infty} \frac{\lambda_k}{\lambda_{k+1}}(\rho^{ss}(\gamma_n \alpha f)) = \infty.
\end{align*}
\end{proof}

To complete the proof that $\xi^{ss}$ is strongly dynamics preserving we recall a few results. First, since $\rho^{ss}$ is semisimple and $\rho^{ss}(\Gamma)$ contains a $\Psf_k$-proximal element, \cite{AMS} and~\cite[Cor.\ 6.3]{Benoist_actions_reductifs} imply that there exist a finite set $F_2 \subset \Gamma$ and some $C_2 > 0$ with the following property: for every $\gamma \in \Gamma$ there is some $f \in F$ such that $\rho^{ss}(\gamma f)$ is $\Psf_k$-proximal and 
\begin{align}
\label{eqn:products are almost the same} 
\frac{\mu_1 \cdots \mu_k}{\lambda_1 \cdots \lambda_k}(\rho^{ss}(\gamma f)) \leq C_2.
\end{align}
Also, by \cite[Prop.\ 2.5(i)]{Kostas2020}, there exists $C_3 > 0$ such that: if $g \in \SL(d,\Kb)$ is $\Psf_k$-proximal and $V^+_g \in \Gr_k(\Kb^d)$ is the attracting subspace, then 
\begin{align}
\label{kostas estimate}
\dist_{\Gr_k(\Kb^d)}\left( V^+_g, U_k(g)\right) \leq C_3\frac{\mu_{k+1}}{\mu_k}(g) \frac{\mu_1 \cdots \mu_k}{\lambda_1 \cdots \lambda_k}(g) .
\end{align}
Finally, by ~\cite[Lem.\ A.4]{BPS}, if $g,h \in \GL(d,\Kb)$, $\mu_k(g) > \mu_{k+1}(g)$, and $\mu_k(gh) > \mu_{k+1}(gh)$,  then 
\begin{equation}
\label{eqn:BPS estimate}
\dist_{\Gr_k(\Kb^d)}( U_k(gh), U_k(g) ) \leq \frac{\mu_1}{\mu_d}(h) \frac{\mu_{k+1}}{\mu_k}(g).
\end{equation}

\begin{lemma} $\xi_{ss}$ is strongly dynamics preserving. \end{lemma} 

\begin{proof} 
Fix an escaping sequence $(\gamma_n)_{n \geq 1}$ in $\Gamma$ such that $\gamma_n \to x$ and $\gamma_n^{-1} \to y$. By Lemma~\ref{lem:the ss is p1 divergent} and Observation~\ref{obs:strongly_dynamics_pres_div_cartan} it suffices to show that $U_k(\rho^{ss}(\gamma_n)) \to \xi^k_{ss}(x)$ and $U_{d-k}(\rho^{ss}(\gamma_n)^{-1}) \to \xi^{d-k}_{ss}(y)$. 

For each $n$, fix $f_n \in F_2$ such that $\rho^{ss}(\gamma_n f_n)$ is $\Psf_k$-proximal and satisfies Equation~\eqref{eqn:products are almost the same}. 
Then $\frac{\lambda_k}{\lambda_{k+1}}(\rho(\gamma_n f_n)) = \frac{\lambda_k}{\lambda_{k+1}}(\rho^{ss}(\gamma_n f_n)) > 1$, and hence by Proposition~\ref{prop:eigenvalue data in rel Anosov repn}(1) each $\gamma_n f_n$ must be a non-peripheral element of $(\Gamma, \peripherals)$. So by the dynamics preserving property, $\xi^k_{ss}( (\gamma_n f_n)^+)$ is the attracting $k$-plane of  $\rho^{ss}(\gamma_n f_n)$.

Since $F_2$ is a finite set and $\gamma_n \to x$, we must have $\gamma_n f_n \to x$. This in turn implies that $(\gamma_n f_n)^+ \to x$. Then by Equations ~\eqref{eqn:BPS estimate},~\eqref{kostas estimate}, and~\eqref{eqn:products are almost the same} 
\begin{align*}
\limsup_{n \to \infty} & \dist_{\Gr_k(\Kb^d)}\left( \xi_{ss}^k(x), U_k(\rho^{ss}(\gamma_n)) \right) = \limsup_{n \to \infty} \dist_{\Gr_k(\Kb^d)}\left( \xi_{ss}^k((\gamma_n f_n)^+), U_k(\rho^{ss}(\gamma_nf_n)) \right) \\
& \lesssim \limsup_{n \to \infty} \frac{\mu_{k+1}}{\mu_k}(\rho^{ss}(\gamma_n f) ) \frac{\mu_1 \cdots \mu_k}{\lambda_1 \cdots \lambda_k}(\rho^{ss}(\gamma_n f)) =0. 
\end{align*}
So $U_k(\rho^{ss}(\gamma_n)) \to \xi^k_{ss}(x)$.

The proof that $U_{d-k}(\rho^{ss}(\gamma_n)^{-1}) \to \xi^{d-k}_{ss}(y)$ is nearly identical.
\end{proof}

\bibliographystyle{alpha}
\bibliography{geom} 

\end{document}